\documentclass[12pt]{article}

\usepackage{graphicx,amsmath,amsthm,amssymb,latexsym}

\usepackage[center,it]{titlesec}

\DeclareMathOperator{\divergence}{div}
\DeclareMathOperator{\curl}{curl}

\DeclareMathOperator{\arsinh}{arsinh}
\DeclareMathOperator{\Area}{Area}
\DeclareMathOperator{\Ric}{Ric}
\DeclareMathOperator{\tr}{tr}

\newcommand{\ud}{\mathrm{d}}
\newcommand{\dm}[1]{\ud \mu_{#1}}

\newcommand{\Dp}{\mathcal{D}_+}
\newcommand{\Cb}{\underline{C}}
\newcommand{\tb}{\langle t \rangle}
\newcommand{\tu}{\underline{t}}

\newcommand{\us}{{u^\ast}}
\newcommand{\vs}{{v^\ast}}

\newcommand{\gs}{g\!\!\!/}
\newcommand{\hs}{h\!\!\!/}

\newcommand{\gsc}{\stackrel{\circ}{\gs}}

\newcommand{\nablas}{\nabla\!\!\!\!/\:}
\newcommand{\divs}{\divergence\!\!\!\!\!/\:}
\newcommand{\curls}{\curl\!\!\!\!\!/\:}
\newcommand{\ds}{\mathrm{d}\!\!\!/}
\newcommand{\Laplaces}{\triangle\!\!\!\!/\:}
\newcommand{\Gammas}{\Gamma\!\!\!\!/\:}
\newcommand{\thetas}{{\theta\!\!\!/}}

\newcommand{\gammac}{\stackrel{\circ}{\gamma}}

\newcommand{\nablac}{\stackrel{\circ}{\nabla}}

\newcommand{\divc}{\stackrel{\circ}{\divergence}}
\newcommand{\Laplacec}{\mathring{\triangle}}
\newcommand{\otimesh}{\hat{\otimes}}

\newcommand{\Lp}{L^\prime}
\newcommand{\Lb}{\underline{L}}

\newcommand{\Lbp}{\Lb^\prime}

\newcommand{\Db}{\underline{D}}
\newcommand{\Dbh}{\hat{\Db}}

\newcommand{\omegab}{\underline{\omega}}

\newcommand{\omegas}{\omega\!\!\!/}
\newcommand{\omegabs}{\omegab\!\!\!/}

\newcommand{\chib}{\underline{\chi}}
\newcommand{\chibh}{\hat{\chib}}
\newcommand{\chibp}{{\chib^\prime}}
\newcommand{\chih}{\hat{\chi}}
\newcommand{\chip}{{\chi^\prime}}

\newcommand{\chibhp}{{\chibh^\prime}}

\newcommand{\etab}{\underline{\eta}}
\newcommand{\betab}{\underline{\beta}}
\newcommand{\alphab}{\underline{\alpha}}

\newcommand{\lVerts}{\lVert\!\!\!\!-}
\newcommand{\rVerts}{\rVert\!\!\!\!-}
\newcommand{\nLp}[1]{\lVerts #1 \rVerts_{\mathrm{L}^{p}(S)}}
\newcommand{\nLpq}[2]{\lVerts #1 \rVerts_{\mathrm{L}^{#2}(S)}}

\newcommand{\Linfty}[1]{\lVert #1 \rVert_{\mathrm{L}^\infty(S)}}
\newcommand{\dLp}[3]{\lVerts #3 \rVerts_{\mathrm{L}^{#1}(S_{#2})}}
\newcommand{\Linf}[2]{\lVert #2 \lVert_{\mathrm{L}^\infty(S_{#1})}}

\newcommand{\mub}{\underline{\mu}}
\newcommand{\kappab}{\underline{\kappa}}

\newcommand{\Ceq}[1]{(#1) in \cite{C:09}}
 \newcommand{\CCh}[1]{Chapter~#1 in \cite{C:09}}
 \newcommand{\CLemma}[1]{Lemma~#1 in \cite{C:09}}
 \newcommand{\CProp}[1]{Proposition~#1 in \cite{C:09}}

\theoremstyle{plain}
\newtheorem*{theorem}{Theorem}
\newtheorem*{prop}{Proposition}
\newtheorem{proposition}{Proposition}[section]
\newtheorem{lemma}[proposition]{Lemma}
\newtheorem{corollary}[proposition]{Corollary}

\theoremstyle{definition}

\theoremstyle{remark}
\newtheorem{remark}[proposition]{Remark}

\numberwithin{equation}{section}
\begin{document}

\title{\textbf{\Large Optical functions in de Sitter}}
\author{Volker Schlue\bigskip\\{\small University of Melbourne}\\{\small volker.schlue@unimelb.edu.au}}
\date{{\small \today}}

\maketitle

\begin{abstract}
  This paper addresses pure gauge questions in the study of (asymptotically) de Sitter spacetimes.
  We construct global solutions to the eikonal equation on de Sitter, whose level sets give rise to double null foliations, and give detailed estimates for the structure coefficients in this gauge. We show two results which are relevant for the foliations used in the stability problem of the expanding region of Schwarzschild de Sitter spacetimes \cite{Schlue:16:Weyl}: (i)~Small perturbations of round spheres on the cosmological horizons lead to spheres that pinch off at infinity. (ii)~Globally well behaved double null foliations can be constructed from  infinity using a choice of spheres related to the level sets of a \emph{mass aspect function}.  While (i) shows that in the above stability problem a \emph{final gauge choice} is necessary, the proof of (ii) already outlines a strategy for the case of spacetimes with decaying, instead of vanishing, conformal Weyl curvature.
  
\end{abstract}

\tableofcontents

\section{Introduction}
\label{sec:intro}


The general covariance of the Einstein equations allows for various formulations of the evolution problem in general relativity depending on the choice of coordinates. While less important for the local evolution, specific gauge choices are crucial for  global evolution problems, and are related to identifying the gravitational degrees of freedom in a given setting.\footnote{This is evident in many recent works on the global evolution problem in various settings, see for instance \cite{L:17,HV:18,DHR:19,KS:18,RS:18a}; see also \cite{A:15,ABK:15}.} This paper addresses a number of pure gauge questions that arise in the global evolution problem for asymptotically de Sitter spacetimes in double null gauge.

Many of the concepts relevant to this paper already appear in the \emph{global analysis of asymptotically flat spacetimes}: Specifically 
in the original proof of the stability of Minkowski spacetime \cite{CK:93} a manifold $\mathcal{M}=\bigcup_{t\geq 0}\Sigma_t$ is constructed as a foliation with respect to the level sets of a \emph{maximal time function} $t$. The proof of global existence proceeds by a continuity argument in $t$, which relies on the careful construction of an \emph{optical function} $u$ on the ``last slice'' $\Sigma_{t_\ast}$.
In general, an \textbf{optical function} $u$ is a solution to the eikonal equation on a Lorentzian manifold $(\mathcal{M},g)$,
\begin{equation}\label{eq:intro:eikonal}
  g(\nabla u,\nabla u)=0\,.
\end{equation}
Its level sets are \emph{null hypersurfaces} generated by null geodesic segments of $g$.
Indeed, in the context of \cite{CK:93} the level sets $S_u^\ast$ of $u$ are chosen on $\Sigma_{t_\ast}$, and then extended to $t\leq t_\ast$ by solving \eqref{eq:intro:eikonal}, essentially by the method of characteristics, which also allows one to suitably propagate coordinates $(\vartheta^1,\vartheta^2)$ on $S_{u^\ast}$ thus complementing $(t,u)$ to a coordinate system on $\mathcal{M}$. This construction is an example of a \textbf{final gauge} choice, where $u$ is adapted to the late stages of evolution. In fact, $S_u^\ast$ is constructed by solving an equation of motion for surfaces on $\Sigma_{t_\ast}$, which is induced by a specific choice for the \emph{mass aspect function} $\mu$ which in turn ensures good asymptotic properties of the instrinsic geometry of the spheres $S_{u}^\ast$.

\begin{figure}[tb]
  \centering
  \includegraphics[scale=1.5]{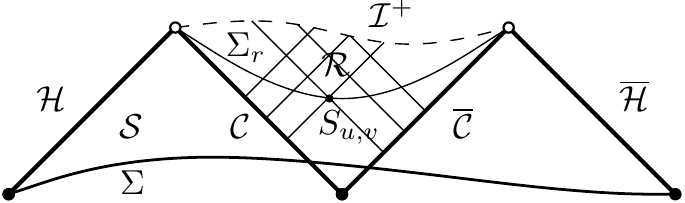}
  \caption{Global geometry of Schwarzschild de Sitter spacetimes.}
  \label{fig:sds:intro}
\end{figure}

This paper is motivated specifically by the global evolution problem for the Einstein equations with positive cosmological constant in \emph{the expanding region of Schwarzschild de Sitter spacetimes in double null gauge}; see Figure~\ref{fig:sds:intro}. In the approach taken in \cite{Schlue:16:Weyl}, the cosmological region $\mathcal{R}$  is viewed as the domain of development of characteristic initial data prescribed on the cosmological horizons $\mathcal{C}$, $\overline{\mathcal{C}}$, and foliated by the level sets of two optical functions $u$, and $v$,  intersecting in surfaces $S_{u,v}$ diffeomorphic to $\mathbb{S}^2$. A time function in $\mathcal{R}$ is obtained from the \emph{area radius} $r(u,v)$ of the spheres $S_{u,v}$,\footnote{The fact that $r$ \emph{is} increasing towards the future at every point $\mathcal{R}$ is a reflection of the \emph{expansion} of this region.}
\begin{equation}\label{eq:area:intro}
  \Area(S_{u,v})=4\pi r^2(u,v)\,,
\end{equation}
and provides a foliation $\mathcal{R}=\bigcup \Sigma_r$ of the cosmological region, similar to the maximal function $t$ in the asymptotically flat setting.
Moreover, by choosing coordinates $(\vartheta^1, \vartheta^2)$ on say the sphere $\mathcal{C}\cap\overline{\mathcal{C}}$, and suitable propagation along the characteristics, the metric on $\mathcal{R}$ is globally expressed in the form
\begin{equation}
  \label{eq:g:doublenull:intro}
  g=-4\Omega^2\ud u \ud v + \gs_{AB}\bigl(\ud \vartheta^A-b^A\ud v\bigr)\bigl(\ud \vartheta^B-b^B\ud v\bigr)\,.
\end{equation}

\begin{remark}[Double null gauge]
  This is an example of a \textbf{double null gauge} \cite{KN:03,C:09}.\footnote{In \cite{KN:03} it was used in particular to provide an alternative treatment of the exterior region of the spacetimes constructed in \cite{CK:93}, and its use of various \emph{renormalised mass aspect functions} play an important role in this paper as well.} The advantage of this gauge is that it respects the domain of dependence property of the Einstein equations, and consequently the (spacelike) hypersurfaces $\Sigma_r$ can in general be constructed locally. In closer analogy to the \emph{maximal gauge} of \cite{CK:93}, one could approach the stability problem in $\mathcal{R}$ in a \emph{constant mean curvature gauge}\footnote{See for instance \cite{RS:18a,RS:18b}, in particular the problem of ``synchronizing the singularity,'' which like the conformal boundary of $\mathcal{R}$, is a \emph{spacelike} hypersurface.}, namely define $\Sigma_r$ directly as CMC hypersurfaces, which however introduces an additional non-local aspect to the problem.
\end{remark}

In this paper we focus on the pure gauge aspect of the problem. In fact, \emph{we fix the metric to be} \textbf{de Sitter}, and we study the properties of global double null foliations, and their construction using global solutions to the eikonal equation on de Sitter.
We are interested in the behavior of the \textbf{structure coefficients} $\Gamma$ of \eqref{eq:g:doublenull:intro} \emph{depending on the choice of optical functions} $u$, and $v$. Moreover, we are interested in the geometry of the surfaces $S_{u,v}=C_u\cap \underline{C}_v$, the conformal properties of the induced metric $\gs_{u,v}$, and geometry of the hypersurfaces 
\begin{equation}\label{eq:Sigmar:intro}
  \Sigma_r=\bigcup_{\substack{(u,v):\\r_{u,v}=r}} S_{u,v}\,.
\end{equation}
The main aim is to identify sufficient conditions (on the level of the surfaces where $u$, and $v$ are chosen, either initially or finally) such that the spacelike hypersurface $\Sigma_r$ provide a well behaved foliation of $\mathcal{R}$, and in particular so that $\Sigma_\infty$ can be identified with the \emph{conformal boundary at infinity} ($\mathcal{I}^+$ in Figure~\ref{fig:sds:intro}) in a suitable sense.

\begin{remark}[Decay of the Weyl curvature]
  The results in this paper in particular provide a justification for the assumptions made on the structure coefficients in $\cite{Schlue:16:Weyl}$. It is proven therein that under suitable assumptions on $\|\Gamma\|_{\mathrm{L}^\infty(S_{u,v})}$, and $\|\nabla\Gamma\|_{\mathrm{L}^4(S_{u,v})}$, \emph{the conformal Weyl curvature of the spacetime decays} according to $\|W\|_{\mathrm{L}^4(S_{u,v})}\lesssim r_{u,v}^{-3}$. Here we give examples of explicit double null foliations of de Sitter whose structure  coefficients \emph{fail} to satisfy these assumptions, which shows in particular that there are obstructions which are unrelated to the specific decay of the Weyl curvature, because they occur already on \textbf{conformally flat} spacetimes.
  We then overcome  this obstacle \emph{in the simplest case}, namely on de Sitter spacetime, using a final gauge construction to prove the existence of global double null foliations which \emph{do} satisfy the assumptions made in \cite{Schlue:16:Weyl}.
\end{remark}

\begin{remark}[Stability of Kerr de Sitter]
  The nonlinear stability problem for the region $\mathcal{S}$, namely the domain to the past of the event horizon $\mathcal{H}$, and the cosmological horizon $\mathcal{C}$ can be treated independently of the above discussed problem for $\mathcal{R}$ due to the \emph{domain of dependence property} of the Einstein equations, and has been fully resolved in \cite{HV:18}. Their proof is based on a \emph{generalised harmonic gauge}, and also relies on a \emph{final gauge choice}, which is found iteratively as part of the analysis of the resonances of the linearised equations.
\end{remark}

\subsection{de Sitter space}
\label{sec:desitter:intro}

The simplest solution to the Einstein equations with positive cosmological constant $\Lambda$,
\begin{equation}
  \label{eq:eve:lambda:intro}
  \Ric(g)=\Lambda g
\end{equation}
 is \textbf{de Sitter space}, which plays  an analogous role to Minkowski space in the case $\Lambda=0$. While Minkowski space is flat, de Sitter is conformally flat. It can be realised  as an embedded (timelike) hypersurface in $5$-dimensional Minkowski space, with induced metric of the ambient Minkowski metric; see Fig.~\ref{fig:desitter:hyperboloid}.

\begin{figure}[tb]
  \centering
  \includegraphics[scale=0.8]{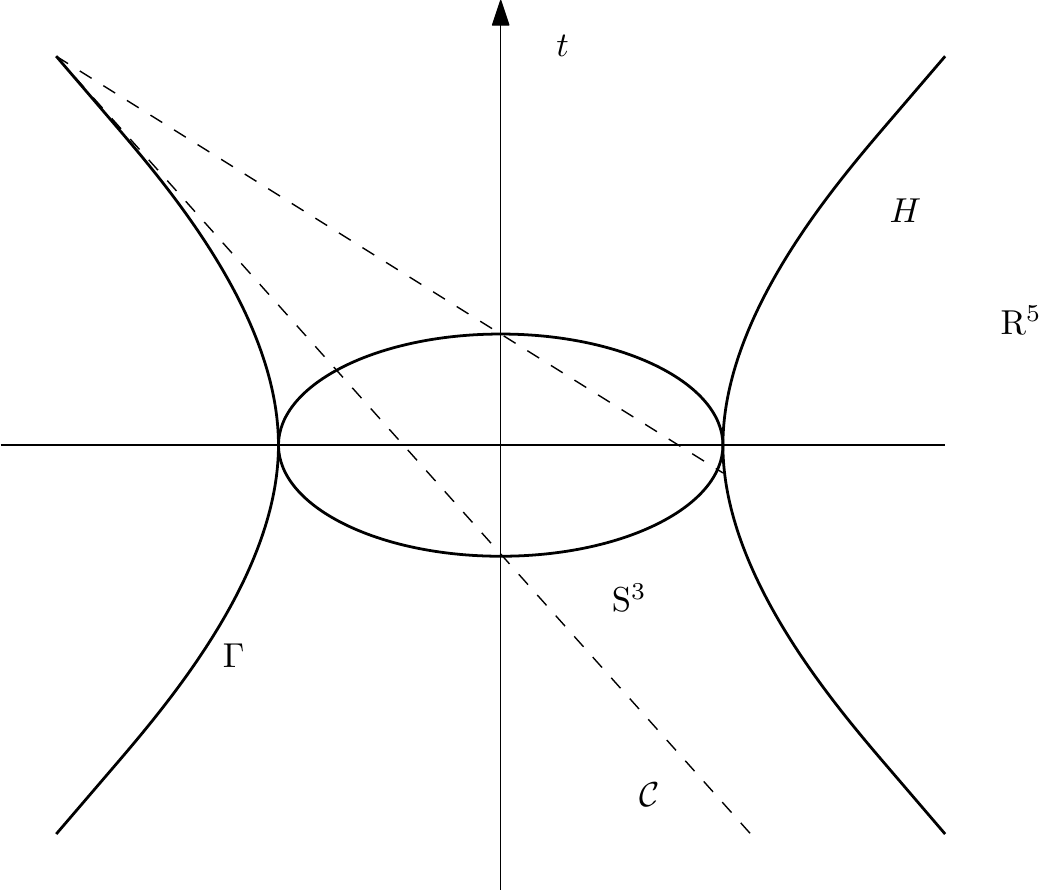}
  \caption{Standard model of de Sitter spacetime.}
  \label{fig:desitter:hyperboloid}
\end{figure}

 Concretely, we set $\Lambda/3=1$ for convenience. With $(t,x)\in\mathbb{R}^{1}\times\mathbb{R}^4$ we denote by
\begin{equation}\label{eq:H}
  H=\bigl\{(t,x):-t^2+\lvert x\rvert^2=1\bigr\}
\end{equation}
the timelike hyperboloid in the ambient Minkowski spacetime $(\mathbb{R}^{1+4},m)$ with metric $m=-\ud t^2+\lvert\ud x\rvert^2$.
The \emph{de Sitter space-time}   $(\mathcal{M}=H,h=m\rvert_H)$  has topology $\mathbb{R}\times\mathbb{S}^3$,  and can be viewed as the simplest model of \emph{closed} universe. Moreover its spatial sections $\Sigma_t=\{x\in\mathbb{R}^4:|x|^2=1+t^2\}$ are round $3$-spheres whose radius is rapidly increasing in time, which shows that this spacetime is manifestly \emph{expanding}.

In Section~\ref{sec:desitter:charts} we will discuss several different coordinates to represent the de Sitter metric.

We will first introduce \emph{stereographic coordinates} which show explicitly that the de Sitter metric is  \emph{conformal} to the Minkowski metric.\footnote{One can view the construction in Section~\ref{sec:stereographic} as the \emph{hyperbolic} analogue of the classical stereographic projection of the sphere on the plane.} That immediately implies that the \textbf{conformal Weyl curvature} of the de Sitter metric vanishes identically.

Second we view de Sitter as spherically symmetric space relative to a timelike geodesic $\Gamma$, see Fig.~\ref{fig:desitter:hyperboloid},
\begin{equation}
  \Gamma=\bigl\{(t,-\langle t\rangle,0,0,0):t\in\mathbb{R}\bigr\}\,.
\end{equation}
In this picture the relevance of de Sitter to the above discussion of the expanding region  of Schwarzschild de Sitter is particularly clear:
The past $\mathcal{S}$ of any observer $\Gamma$ is \emph{not} the entire spacetime, and thus has a future boundary called the \textbf{cosmological horizon} $\mathcal{C}$. (See also explicit discussion in Section~\ref{sec:spherical}.)

In Section~\ref{sec:static} we introduce coordinates relative to which the de Sitter metric  is static in $\mathcal{S}$, and expanding in the cosmological region $\mathcal{R}$ (namely the future of $\mathcal{C}\cup\overline{\mathcal{C}}$, where $\overline{\mathcal{C}}$ is the cosmological horizon of the antipodal observer $\overline{\Gamma}$), in the sense discussed for Schwarzschild de Sitter above.

This leads in particular to the Penrose diagram of de Sitter of Figure~\ref{fig:desitter:penrose}.


\begin{figure}[tb]
  \centering
  \includegraphics{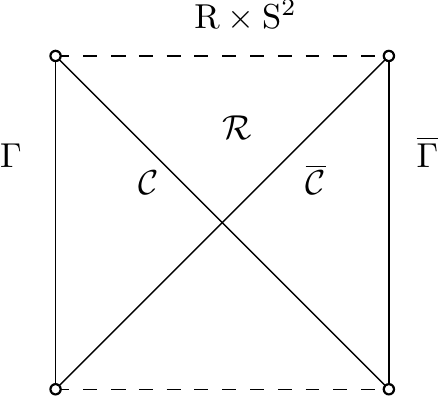}
  \caption{Penrose diagram of de Sitter}
  \label{fig:desitter:penrose}
\end{figure}

\begin{remark}
  The stability of de Sitter as a solution to \eqref{eq:eve:lambda:intro} is known since \cite{F:86}. Also \cite{HV:18} can be applied to the past of a point on the conformal boundary to yield convergence to de Sitter up to a gauge transformation. However, in \cite{F:86} \emph{asymptotic functional degrees of freedom} are present, namely the metric does not globally converge to de Sitter in a suitable gauge. This insight equally applies to the stability problem in $\mathcal{R}$ discussed above. See also \cite{R:08}.
\end{remark}

\subsection{Results for the eikonal equation on de Sitter}
\label{sec:results}

We now summarise the results in this paper on the global geometry of double null foliations in de Sitter spacetime.

Let $\Gamma$, $\overline{\Gamma}$ be antipodal timelike geodesics in the de Sitter spacetime $(H,h)$, and $\mathcal{C}$, $\overline{\mathcal{C}}$ their respective cosmological horizons, see Section~\ref{sec:desitter:intro} and Figures~\ref{fig:desitter:hyperboloid}, \ref{fig:desitter:penrose}.
For any solution $u$ to the eikonal equation
\begin{equation}
  \label{eq:eikonal:intro}
  h(\nabla u,\nabla u)=0
\end{equation}
we denote by $C_u$ the level set of $u$; see also  Section~\ref{sec:eikonal:general}.

\begin{figure}[tb]
  \centering
  \includegraphics[scale=1.2]{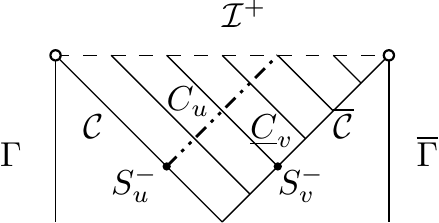}
  \caption{Family of shear-free null hypersurfaces $\Cb_v$.}
  \label{fig:Cv}
\end{figure}

Throughout this paper we fix a family of shear-free null hypersurfaces $\Cb_v$, which are always intersected transversally by the null hypersurfaces $C_u$. In fact, $\Cb_v$ are the level sets of a solution to the eikonal equation \eqref{eq:eikonal:intro}  which is chosen so that $\overline{\mathcal{C}}\cap\underline{C}_v=S_v^-$ are round spheres, see Figure~\ref{fig:Cv}. In other words, the hypersurfaces $\underline{C}_v$ are unchanged from the spherically symmetric foliation in Section~\ref{sec:spherical}.

The first observation is that the double foliation resulting from a choice of $u$ on $\mathcal{C}$ is generally not well behaved, in the sense that the geometry of the spheres $C_u\cap\underline{C}_v$ diverges from the spherically symmetric picture  as the conformal boundary $\mathcal{I}^+$ is approached.

\begin{prop} 
  Let $u$ be the solution to the eikonal equation \eqref{eq:eikonal:intro} corresponding to a choice of level sets $S_u^-=C_u\cap\mathcal{C}\simeq\mathbb{S}^2$ on $\mathcal{C}$.
  There exist arbitrarily small deformations of the round spheres $\mathbb{S}^2\leadsto S_u^-$, such that  $S_{u,v}:=C_u\cap \underline{C}_v$ has the property that for some $(u_+,v_+)$, $p\in S_{u_+,v_+}$, \[\text{Area}(S_{u_+,v_+})=\infty \qquad  \text{, while }\qquad  S_{u_+,v_+}\setminus{p}\subset H\,.\]
\end{prop}

\begin{figure}[tb]
  \centering
  \includegraphics[scale=0.6]{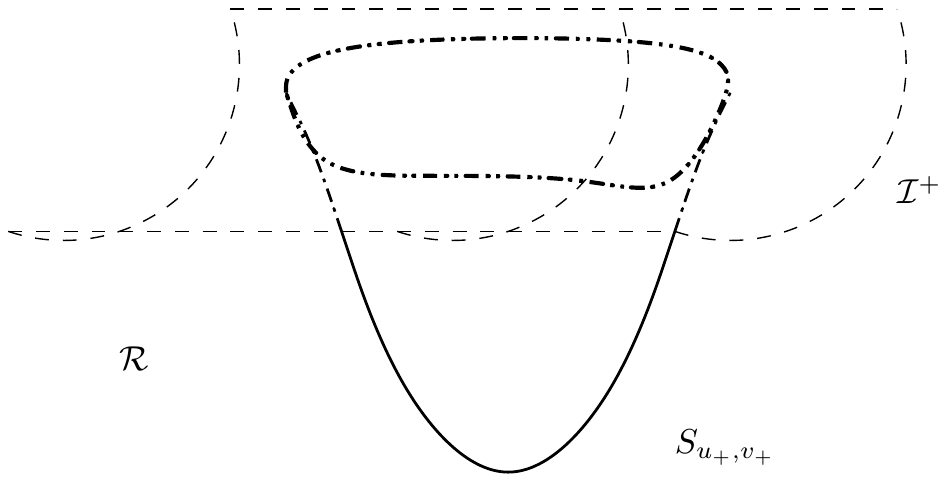}
  \caption{Example of a sphere of infinite area partially contained in the spacetime.}
  \label{fig:pinch}
\end{figure}

In other words, generically small deformations of the initial data  on $\mathcal{C}$ lead to spheres of intersection in the expanding region which ``pinch off'' at infinity, see Fig~\ref{fig:pinch}. In Section~\ref{sec:ellipsoidal} we will demonstrate this phenomenon explicitly for ellipsoidal deformations of the round spheres. These examples show in particular that a double null gauge chosen \emph{initially} generally does \emph{not} parametrize correctly future null infinity, i.e.~for these foliations it is \emph{not} true that $\mathcal{I}^+=\{r=\infty\}$.

In Section~\ref{sec:ellipsoidal:explicit} we derive explicitly the transformation of all optical structure coefficients, and characterise globally the gauge transformation induced by small ellipsoidal deformations on the cosmological horizons. In particular, the statement of the proposition follows from the estimates derived therein.

\bigskip

The second result states the existence of double null foliations which are free of the above pathological behavior.
The starting point of their construction is now a sphere $S_{u_+,v_+}^\infty\subset\mathcal{I}^+$ which is properly contained in $\mathcal{I}^+$, and further a prescription of spheres $S_{u,v_+}\subset \underline{C}_{v_+}\equiv\underline{C}_+$ which are the final level sets of an optical function $u$, for which we solve globally in the domain $\mathcal{D}_+$; see Figure~\ref{fig:D}.

\begin{theorem} [Informal version]
  For an open set of initial data on $\underline{C}_{+}$ there exists a solution to the eikonal equation on $\mathcal{D}_+$ with the property that all spheres $S_{u,v}=C_u\cap\underline{C}_{v}$ are \emph{uniformly} conformal to $\mathbb{S}^2$. 
\end{theorem}

\begin{figure}[tb]
  \centering
  \includegraphics{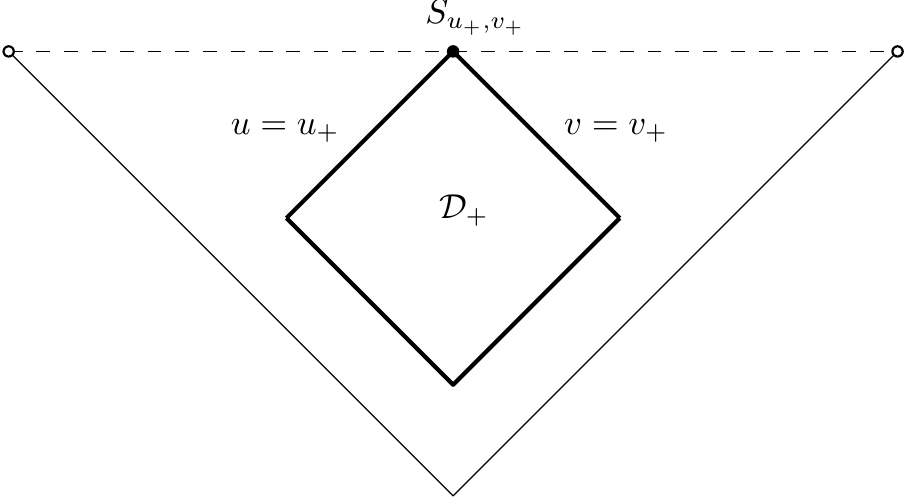}
  \caption{Domain for the construction of globally well-behaved double null foliations in the past of a sphere $S_{u,v}\subset \mathcal{I}^+$.}
  \label{fig:D}
\end{figure}

Recall that  for simplicity we have chosen the transversal null hypersurfaces $\underline{C}_v$ to  be shear-free, and thus restrict our attention to the construction of non-trivial optical function $u$.
While previously we have chosen \emph{initial} data on $\mathcal{C}$ such that  $\mathcal{C}\cap C_u=S^-_{u}$ is a small perturbation of a round sphere, we now choose \emph{final} data  on $\underline{C}_{+}$, such that $C_u\cap\underline{C}_{v_+}=S^+_{u}\simeq \mathbb{S}^2$ with the property that $S^+_{u}\to S^\infty_{u_+,v_+}\subset \mathcal{I}^+$, as $u\to u_+$.

The notion of ``uniformly conformal'' refers to \emph{uniform bounds} on the conformal factor $\psi$, when the geometry of  $(S_{u,v},\gs)$ is  compared to $(\mathbb{S}^2,\gammac)$:
\begin{equation}\label{eq:gs:intro}
  \gs_{u,v}=r^2(u,v)\,e^{2\psi}\,\gammac
\end{equation}
This form of the metric is immediate in the shear-free case,\footnote{See also Section~\ref{sec:shear-free:graphs}.} but more generally a consequence of the uniformization theorem in  Section~\ref{sec:uniformization}.
The application of the latter typically relies on good bounds for the Gauss curvature $K$ of $S_{u,v}$, which however is not easily obtained by propagation along the characteristics of $C_u$. A closely related quantity is the \emph{mass aspect function} which we define by
\begin{equation}\label{eq:mu:intro}
        \mu := K-\divs\eta+\frac{1}{4}\tr\chi\tr\chib-\frac{\Lambda}{3}
\end{equation}
Here $\eta$ is the \emph{torsion},\footnote{namely the torsion with respect to the null geodesic vectorfield $L'=-\ud u^\sharp$, $\eta\cdot X=\frac{1}{2}g(\nabla_X L',\Lb)$, where $g(L',\Lb)=-2$, $g(\Lb,X)=0$, a 1-form on $S_{u,v}$ whose non-vanishing is, geometrically, the obstruction to the integrability of the distribution of planes $(TS_{u,v})^\perp\subset TH$; see Chapter~1 of \cite{C:09}.} and $\tr\chi$, $\tr\chib$ are the outgoing, and ingoing null mean curvatures, respectively.
The crucial property of globally well-behaved double null foliations is the asymptotic behavior of $\mu$; see also the discussion in Section~\ref{sec:prelim}.

In fact, it turns out to be convenient to prescribe $S_u^+$ as the level sets of a related function $\breve{\mu}$,
\begin{equation}\label{eq:mu:breve:intro}
  \breve{\mu}:=K-\divs\eta+\frac{1}{2}\tr\chi\tr\chib-2\frac{\Lambda}{3}
\end{equation}
which ``decouples'' several equations in double null gauge. As observed in \cite{KN:03} the gauge choice that $\breve{\mu}$ equals its average $\overline{\breve{\mu}}$ on $S_u^+$ corresponds to an equation of motion for surfaces on $\underline{C}_+$ which determines $S_u^+$ from a single section, say  $S_{u_+}^+$. We will not study this equation of motion in this paper, but simply assume that a foliation of $\Cb_+$ by surfaces $S_u^+$ has been constructed and has the property that
\begin{equation}
  \label{eq:mu:tilde:intro}
  \tilde{\mu}:=\breve{\mu}-\overline{\breve{\mu}}\,,\qquad \overline{\breve{\mu}}(u,v)=\frac{1}{4\pi r^2(u,v)}\int_{S_{u,v}} \breve{\mu}\ \ud \mu_{\gs_{u,v_+}}
\end{equation}
satisfies suitable upper bounds. In this paper we use the dimensionless norms
\begin{equation}
  \dLp{4}{u,v}{\theta}:= \Bigl( \frac{1}{4\pi r^2}\int_{S_{u,v}}|\theta|^4_{\gs}\dm{\gs}\Bigr)^{1/4}\,.
\end{equation}

\begin{theorem}[Precise version]
  Suppose the surfaces $S_{u}^+\subset\Cb_+$ are prescribed such that
  \begin{equation}
    r^3 \dLp{4}{u,v_+}{\tilde{\mu}}+r^4 \dLp{4}{u,v_+}{\nablas \tilde{\mu}}\lesssim 1.
  \end{equation}
  Then the geometry of the surfaces $(S_{u,v}=C_u\cap\Cb_v,\gs_{u,v})$, where $\gs$ is given by \eqref{eq:gs:intro}, is controlled by
  \begin{equation}\label{eq:psi:intro}
         \Linf{u,v}{\psi}\lesssim 1\qquad          \dLp{4}{u,v}{r\ds\psi}\lesssim 1\qquad          \dLp{4}{u,v}{r^2\nablas^2\psi}\lesssim 1
       \end{equation}
       on a domain $\mathcal{D}_+$,
and moreover all bounds (\textbf{A:I,\underline{I},II,\underline{II}}) of Section~\ref{sec:BA} are satisfied, provided \eqref{D:small} holds, in particular
\begin{align}
  |2\omega-\Omega\tr\chi|\lesssim 1 &\qquad |2\omegab-\Omega\tr\chib|\lesssim 1\\
  r^3 \dLp{4}{u,v}{ \nablas\eta}\lesssim 1 &\qquad r^3 \dLp{4}{u,v}{ \nablas\etab}\lesssim 1\,.
\end{align}
\end{theorem}


\begin{remark}
  This theorem in particular provides a justification for the assumptions made in \cite{Schlue:16:Weyl}.
  Therein we proved a theorem on the condition that bounds of the form (\textbf{A:I,\underline{I},II,\underline{II}}) are verified,
  that states that the Weyl curvature of a spacetime -- expressed in double null gauge -- decays with respect to a foliation of the form \eqref{eq:Sigmar:intro}. The above  theorem  indicates \emph{consistency}, in the context of the Einstein equations. Namely to prove the global existence of the developments considered in \cite{Schlue:16:Weyl} --- corresponding to the regions $\mathcal{R}$ above --- we need to recover assumptions of the form  (\textbf{A:I,\underline{I},II,\underline{II}}), which correspond directly to geometric properties of solutions to the eikonal equation, from knowledge of data and the behavior of the Weyl curvature of the spacetime. Now the main obstacles in doing so already appear to leading order, meaning they are unrelated to the specific decay rate of the Weyl curvature, but already occur when the conformal Weyl curvature vanishes identically, namely in the setting of de Sitter which is treated in this paper.

\end{remark}

\subsection{Overview}

The geometry of the de Sitter spacetime is described from several points  of view in Section~\ref{sec:desitter:charts}.
The construction of spherical and ellipsoidal double null foliations in Section~\ref{sec:desitter:double}, and the proof of the Proposition, are explicit, and can be read independently of the rest of the paper.

In Section~\ref{sec:global:discussion} we discuss the setup relevant for Theorem, and give a fast paced preliminary discussion --- assuming familiarity with the null structure equations given in Section~\ref{sec:null:structure} --- highlighting the \emph{necessary} behavior of various geometric quantities, for the conclusions of the Theorem to hold. This informs the bootstrap assumptions presented in Section~\ref{sec:BA}, and explains the role of the mass aspect function in this problem.

The overall argument for the proof of the Theorem follows the strategy in \cite{C:09}:
Based on the bootstrap assumptions we derive $\mathrm{L}^\infty$ and $\mathrm{L}^4$ estimates for the connection coefficients from propagation equations in Sections~\ref{sec:Linfty}, \ref{sec:L4}. These give sufficient control on the conformal factor in \eqref{eq:gs:intro}, see in particular the discussion of the uniformization theorem in Section~\ref{sec:uniformization},  to derive the validity of various elliptic estimates for Hodge systems on $(S_{u,v},\gs)$ (by reduction to the Calderon-Zygmund estimates on $\mathbb{S}^2$), which are discussed in Section~\ref{sec:elliptic}. In Section~\ref{sec:coupled} we then turn to the analysis of the coupled systems, of propagation equations along the characteristics and elliptic systems on the spheres, which together with a suitable choice of data at infinity yield the improvement of the initial bootstrap assumptions.

A general discussion of the eikonal equation is given in Section~\ref{sec:eikonal:general}, and we have included a discussion of the sections of shear free null hypersurfaces as graphs over round spheres in Section~\ref{sec:shear-free:graphs}.








\begin{quote}
  \textbf{Acknowledgements.}
  I would like to thank Mihalis Dafermos for an enlightening discussion at the ``Nonlinear Waves Trimester'' at the IH\'ES in 2016 which motivated some of the considerations in this paper. 
  The author gratefully acknowledges the support of the \emph{Fondation Sciences Math\'ematiques de Paris} and \emph{ERC consolidator grant EPGR 725589}, and the support of \emph{ERC advanced grant 291214 BLOWDISOL} and \emph{ERC Starting Grant StabMAEinstein}, as well as the hospitality of the \emph{Institut Henri Poincar\'e} in the fall semester 2016, and the \emph{Institut Mittag Leffler} in the fall semester 2019, where part of this research was conducted.
\end{quote}

\section{Construction of global solutions from infinity}
\label{sec:global:discussion}


\subsection{Set-up of the characteristic problem}

Let $(H,h)$ be de Sitter spacetime. We will express $g=h$ in double null coordinates $(u,v;\vartheta^1,\vartheta^2)$ on a domain $(u,v)\in\mathcal{D}_+$, where for fixed $u_+>0$, and  $v_+>0$,
\begin{equation}
  \mathcal{D}_+:=\bigl\{(u,v):   0 \leq u \leq u_+, 0 \leq v \leq v_+\bigr\}\,.
\end{equation}

The metric takes the form \eqref{eq:g:doublenull:intro}.
Here $u$, and $v$, are solutions to the eikonal equation \eqref{eq:eikonal:intro}, which are increasing to the future.
We denote the level sets of $u$, a null hypersurface, by $C_u\subset H$, similary the level set of $v$ by $\Cb_v\subset H$.
We refer the reader to Chapter~1 in \cite{C:09} for basic definitions of all geometric quantities, and assume familiarity with the notation introduced therein.

The hypersurfaces $C_+:=C_{u_+}$, and $\underline{C}_+:=\underline{C}_{v_+}$ play the role  of a ``last slice'', in the sense that we construct solutions to the eikonal equation in the past of $C_+\cup\Cb_+$. In all constructions $S_{u,v}=C_u\cap\Cb_v$ are compact spacelike surfaces diffeomorphic to $\mathbb{S}^2$, and the area radius is defined by \eqref{eq:area:intro}. The coordinates  and $(\vartheta^1,\vartheta^2)$ are chosen  on $S_{+}$, then first transported along $C_+$, then along $\underline{C}_{v}$; see Chapter~{1.4} in \cite{C:09} for further details.

We are interested in the case that $S_+:= C_+\cap\Cb_+$ can be identified with a sphere $S_+^\infty\subset\mathcal{I}^+$,\footnote{In this case $C_+$ and $\Cb_+$ in fact do not intersect in $H$, see Section~\ref{sec:desitter:double}.}  in particular  $r(u_+,v_+)=\infty$. Moreover we take $C_+$ to be spherically symmetric, and $\Cb_v$ to be shear-free, as described next in Section~\ref{sec:setup:initial:data}.




\subsubsection{Data on $C_+$}
\label{sec:setup:initial:data}

Let $p\in H\subset\mathbb{R}^{4+1}$ be a point in de Sitter, and $C_+=C_p^+\cap H$ where $C_p^+$ is the future-directed cone emanating from $p$ in the ambient Minkowski spacetime $\mathbb{R}^{4+1}$.

A preferred foliation of the null hypersurface $C_+$ is obtained by viewing the spacetime $H$ as spherically symmetric with respect to a timelike geodesic $\Gamma$ passing through $p$. By a suitable choice of cartesian coordinates $(t,x,x')\in\mathbb{R}^{1+1+3}$  we can take $\Gamma(t)=(t,-\langle t\rangle,0)$, and $p=\Gamma(t_p)$.

Consider the null hypersurfaces $\underline{C}(t)=C_{\underline{\Gamma}(t)}^+\cap H$ emanating from the antipodal geodesic $\underline{\Gamma}(t)=(t,\langle t\rangle,0)$. The sections $S_t^+=C_+\cap\underline{C}(t)$ give a preferred foliation of $C_+$ by round 2-spheres.

In fact by Lemma~\ref{lemma:levels} the null hypersurfaces $C(t)=C_{\Gamma(t)}^+\cap H$, and $\Cb(t)$, are the level sets $C_u$, and $\Cb_v$ of the  functions
\begin{equation}\label{eq:uv:intro}
  u(t,x,x')=-\frac{1}{4}\log\Bigl(\frac{t+x}{t-x}\frac{|x'|+1}{|x'|-1}\Bigr)\,,\quad
  v(t,x,x')=-\frac{1}{4}\log \Bigl(\frac{t-x}{t+x}\frac{|x'|+1}{|x'|-1}\Bigr)\,.
\end{equation}
In these double null coordinates the metric $h$ takes the form \eqref{eq:metric:double:spherical}.

In the following we will keep the null hypersurface $C_+=C(\Gamma(t_p))=C_{u_+}$, and the global solution to the eikonal equation $v$, whose level sets trace out a preferred family of sections on $C_+$.

\smallskip

We compute \textbf{on }$\mathbf{C_+}$:

\begin{equation}
  \gs_{(u_+,v)}=r^2(u_+,v)\gammac
\end{equation}

\begin{equation}
  \chih=0\qquad 
      \tr\chi=\frac{2}{r}\Omega
\end{equation}


\begin{equation}
\Omega^2=r^2-1\qquad 
    \omega=r
  \end{equation}

  

  \begin{equation}
    \eta=0\qquad 
    \etab=-\eta=0
  \end{equation}

  \begin{equation}
    K=\frac{1}{r^2}
  \end{equation}

\begin{equation}\label{eq:tr:chib:data}
  \tr\chib=\frac{2}{r}\Omega
\end{equation}

The null hypersurfaces $\Cb_v$ are shear-free which can also be proven purely from the data on $C_+$:

\begin{lemma}\label{lemma:shear-free}
  Let $\Omega^2 \chibh\rvert_{S_{u_+,v_+}}=0$ then $\chibh=0$ everywhere, i.e.~the null hypersurfaces $\Cb_v$ are all shear-free.
\end{lemma}

\begin{proof}
  Since  $\eta=\etab=0$, $\chih=0$ on $C_0$, first note that \eqref{eq:D:chib} reduces in view of \eqref{eq:D:Omega:tr:chib} to
  \begin{equation}
    \hat{D}(\Omega\chibh)=0\qquad \text{: on }C_+
  \end{equation}
  Therefore, again using that $\chih=0$ on $C_+$:
  \begin{equation}
    D|\Omega^2\chibh|^2=(2\omega-\Omega\tr\chi)|\Omega^2\chibh|^2
  \end{equation}
  which obviously implies $\chibh=0$ on $C_+$ provided $\Omega^2\chibh\bigr\rvert_{S_{u_+,v_+}}=0$.

  Now we can use \eqref{eq:Db:chibh} to derive
 \begin{equation}
   \Db|\Omega^2\chibhp|^2=2(2\omega-\Omega\tr\chib)|\Omega^2\chibhp|^2-2\Omega\chibh^{\sharp\sharp}(\Omega^2\chibhp,\Omega^2\chibhp)
 \end{equation}
 which implies
 \begin{equation}
   |\Db|\Omega^2\chibhp|^2|\leq 2\bigl(\Omega|\chibh|+|2\omega-\Omega\tr\chi|\bigr)|\Omega^2\chibhp|^2
 \end{equation}
 Hence by Gronwall's inequality $\chibh\equiv 0$.
\end{proof}


\subsubsection{Data on $\Cb_+$}
\label{sec:data:Cb}

Let $\Cb_+$ be the ``last'' null hypersurface $\Cb_v$ that intersects $C_+$, namely $\Cb_+=\Cb_{v_+}$ where $v_+$ is chosen such that $S_+=C_+\cap\Cb_+$ has infinite area, i.e.~$r(u_+,v_+)=\infty$. We will prescribe the level sets of $u$ on $\Cb_+$, and extend $u$ as a solution to the eikonal equation \eqref{eq:eikonal:intro} to the past of $C_+\cup\Cb_+$.

The Theorem makes reference to a foliation of $\Cb_+$ by level sets $S_u^+$ relative to which the mass aspect function $\mu$ defined in \eqref{eq:mu:intro}, or more precisely the renormalised quantity $\breve{\mu}$ of \eqref{eq:mu:breve:intro}, is asymptotically constant.
A few comments are in order about how to achieve such a foliation ``on the last slice'':
Since 
\begin{equation}
    \eta+\etab=2\ds\log\Omega\label{eq:eta:etab:log:Omega}\\
  \end{equation}
  we have
  \begin{equation}
  \breve{\mu}+\breve{\mub}=-2\rho[W]+(\chih,\chibh)+\frac{1}{2}\tr\chi\tr\chib-2-2\Laplaces\log\Omega\label{eq:Laplace:log:Omega}
\end{equation}
where $\mub$ is the conjugate mass aspect function defined in \eqref{eq:mub}.
Here $\rho[W]=0$, and $\chibh=0$. Moreover since $\breve{\mub}$, and $\tr\chi\tr\chib$, are independent of $\Omega$, the prescription of $\breve{\mu}$ yields an equation for $\Omega$.\footnote{Since $\etab$, and $\tr\chi\tr\chib$, and hence $\mub$ do not depend on $\Omega$, the function $\mub$ on $S_u^+$ can be computed from knowledge of the section $S_u^+$ alone. See Section~3.3.3 in \cite{KN:03} for details.} For example the choice $\breve{\mu}=\overline{\breve{\mu}}$ yields the \textbf{equation of motion of surfaces}:
\begin{equation}\label{eq:motion}
  2\Laplaces\log\Omega=\frac{1}{4}\Bigl(\tr\chi\tr\chib-\overline{\tr\chi\tr\chib}\Bigr)+\divs\etab
\end{equation}
In other words given a section $S_u^+$ this elliptic equation determines the lapse of the foliation of $\Cb_+$ by the surfaces $S_u^+$,
and thus can be viewed as defining a flow of surfaces along the generators of $\Cb_+$.
We do not study this geometric flow in this paper, but the existence theory is developed in \cite{KN:03} (in the asymptotically flat setting).

While this is a problem of \emph{constructing the initial (or final) data}, we focus in this paper on controlling the solutions to the eikonal equation with this prescribed data. In fact, instead of working with a ``constant mass aspect foliation'' we will  assume directly that the foliation is such that
\begin{equation}
  \Linf{u,v_+}{\breve{\mu}-\overline{\breve{\mu}}}\leq C/r^3(u,v_+)\,.
\end{equation}

Further assumptions on the foliation of $\Cb_{v_+}$ are made in Sections~\ref{sec:L4}, and \ref{sec:coupled}. It can however be expected that these properties follow from the analysis of \eqref{eq:motion}, and a suitable choice of a single section $S_{u}^+$; see Section~3 in \cite{KN:03}.
  
\subsubsection{de Sitter values}

For future reference we note here de Sitter values in the spherically symmetric double null foliation \eqref{eq:uv:intro}, c.f.~Section~\ref{sec:spherical}:
  \begin{gather}
        \Omega^2=r^2-1\\
    \tr\chi=\tr\chib=\frac{2}{r}\Omega\\
    \omega=\omegab=r
  \end{gather}



  \subsection{Bootstrap assumptions}
  \label{sec:BA}

We will make the following bootstrap assumptions on the null structure coefficients:  

  \begin{equation}\label{BA:log:Omega:r}
  \bigl\lvert \log \Bigl( \frac{\Omega}{r} \Bigr) \bigr\rvert \leq \Delta_0 \tag{\textbf{A:0}}
\end{equation}

\begin{subequations}
  \begin{align}
     \tr\chi &\geq 1 \label{BA:tr:chi} \tag{\textbf{A:I}.i}\\
 \tr\chib &\geq 1 \label{BA:tr:chib} \tag{\textbf{A:\underline{I}}.i}
\end{align}
\end{subequations}

\begin{subequations}
  \begin{align} 
    | 2\omega-\Omega\tr\chi| &\leq \Delta_I \label{BA:omega:tr:chi} \tag{\textbf{A:I}.ii}\\
    | 2\omegab-\Omega\tr\chib| &\leq \Delta_I \label{BA:omegab:tr:chib} \tag{\textbf{A:\underline{I}}.ii}
\end{align}
\end{subequations}

\begin{subequations}
  \begin{align} 
    | \Omega\tr\chi-\overline{\Omega\tr\chi} | &\leq \Delta_I \label{BA:tr:chi:average} \tag{\textbf{A:I}.iii}\\
     | \Omega\tr\chib-\overline{\Omega\tr\chib} | &\leq \Delta_I \label{BA:tr:chib:average} \tag{\textbf{A:\underline{I}}.iii}
\end{align}

\end{subequations}

\begin{subequations}
  \begin{align}
    \Omega^2 |\ds\omega| &\leq \Delta_{II} \label{BA:Omega:ds:omega} \tag{\textbf{A:II}.i}\\
      \Omega^2 |\ds\omegab| &\leq \Delta_{II} \label{BA:Omega:ds:omegab} \tag{\textbf{A:\underline{II}}.i}
\end{align}

\end{subequations}

\begin{subequations}
    \begin{align}
            \dLp{4}{u,v}{\Omega^3 \nablas^2 \omega} &\leq \Delta_{II} \label{BA:Omega:dd:omega} \tag{\textbf{A:II}.ii}\\
      \dLp{4}{u,v}{\Omega^3 \nablas^2 \omegab} &\leq \Delta_{II} \label{BA:Omega:dd:omegab} \tag{\textbf{A:\underline{II}}.ii}
    \end{align}

\end{subequations}

Many of these assumptions will be recovered under a smallness condition on the domain in $(u,v)$ coordinates, and we will refer to this condition as
\begin{equation}
  u_+-u\ll 1\qquad v_+-v\ll 1 \label{D:small} \tag{\boldmath{$\epsilon$}}
\end{equation}



\subsection{Preliminary discussion}
\label{sec:prelim}

Consider the domain $\mathcal{D}_+$ is the past of $C_+\cup\Cb_+$ foliated by the transversal null hypersurfaces $C_u$, $\Cb_v$, $u\leq u_+$, $v\leq v_+$. As we have seen above we can achieve the shear-free case
\begin{equation}
  \chibh =0\,,
\end{equation}
 by prescribing spherically symmetric data for $v$ on $C_+$, and $\Omega^2\chibh\rvert_{S_{+}}=0$. From non-trivial data for $u$ on $\Cb_+$ we will obtain non-trivial spheres of intersection $S_{u,v}=C_u\cap\Cb_v$, and we are interested in  the  geometry of the spheres $S_{u,v}$ as the area radius $r(u,v)\to \infty$. Here we give a preliminary discussion under what conditions these spheres remain conformal to the round sphere with a uniformly bounded conformal factor, and what are the necessary conditions for  $r^2(u,v) K(u,v)$, where $K$ is the Gauss curvature, to be uniformly bounded on $\mathcal{D}_+$.


\subsubsection{Estimates of the conformal factor}
\label{sec:prelim:conformal}


We have 
\begin{equation}
  \Db\gs=2\Omega\chib=\Omega\tr\chib\gs
\end{equation}
and since $\gs_{u_+,v}=r^2(u_+,v)\gammac$, we know that $\gs$ is always conformal to $\gammac$.
We parametrize the conformal factor by
\begin{equation}\label{eq:conformal:factor}
  \gs=r^2e^{2\psi}\gammac
\end{equation}
then we find that
\begin{equation}\label{eq:Db:psi:shear-free}
  2 \Db\psi = \Omega\tr\chib-\overline{\Omega\tr\chib}
\end{equation}
and
\begin{equation}
  2 \Db\ds\psi = \ds\bigl(\Omega\tr\chib\bigr)
\end{equation}

We could in fact exploit the Codazzi equations here, which imply in the shear-free case that
\begin{equation}\label{eq:codazzi:shear-free}
  \ds(\Omega\tr\chib)=\Omega\tr\chib\eta
\end{equation}
hence
\begin{equation}\label{eq:Db:dpsi:shear-free}
  \Db(\Omega^2|\ds\psi|^2)=(2\omegab-\Omega\tr\chib)\Omega^2|\ds\psi|^2+\Omega\tr\chib (\Omega\ds\psi,\Omega\eta)
\end{equation}

Moreover
\begin{equation}
  \begin{split}
    2\Db\nablas\ds\psi&=\nablas(\Omega\tr\chib\eta)+i\cdot \ds\psi\qquad |i|\leq |\nablas(\Omega\tr\chib)|\\
    &=\Omega\tr\chib\, \eta\cdot\eta+\Omega\tr\chib\nablas\eta+i\cdot \ds\psi
\end{split}
\end{equation}
hence
\begin{multline}\label{eq:Db:ddpsi:shear-free}
  2\Db\bigl(\Omega^4|\nablas\ds\psi|^2\bigr)=2(2\omegab-\Omega\tr\chib)\Omega^4|\nablas\ds\psi|^2\\+\Omega\tr\chib\Bigl(\Omega^2 \nablas\ds\psi,\Omega\eta\cdot(\Omega\eta+\Omega\ds\psi)+\Omega^2\nablas\eta\Bigr)
\end{multline}

\begin{lemma}\label{lemma:prelim:conformal}
  Suppose $\gs_{u_+,v}=r^2(u_+,v)\gammac$, and $\chibh=0$ on $\mathcal{D}_+$.
  Moreover assume that \emph{(\textbf{A:\underline{I}})} holds on $\mathcal{D}_+$, and:
  \begin{subequations}
  \begin{align}
    |\Omega^2\eta|&\lesssim 1\\
    \dLp{4}{u,v}{\Omega^3\nablas\eta}&\lesssim  1
  \end{align}
\end{subequations}
Then
\begin{equation}
  |\psi|+\Omega|\ds\psi|+\dLp{4}{u,v}{\Omega^2\nablas^2\psi}\lesssim u_+-u
\end{equation}

\end{lemma}
\begin{proof}
  Note that $\psi=0$ on $C_+$. Hence the bound for $\psi$ follows immediately from \eqref{eq:Db:psi:shear-free}.

  From \eqref{eq:Db:dpsi:shear-free} and Lemma~\ref{lemma:ode} below we infer
  \begin{equation*}
    \Omega|\ds\psi|\lesssim  e^{2\Delta (u_+-u)}\int_u^{u_+} \sup \tr\chib \Omega^2|\eta| \ud u
  \end{equation*}

  Finally let us apply Lemma~\ref{lemma:dLp:gronwall} to \eqref{eq:Db:ddpsi:shear-free} to infer that
  \begin{multline*}
      \dLp{4}{u,v}{\Omega^2\nablas\ds\psi} 
      \lesssim e^{\Delta u_+}\int_u^{u_+}  \Linf{u',v}{\Omega^2\eta}\bigl(\Linf{u',v}{\Omega\eta}+\Linf{u',v}{\Omega\ds\psi}\bigr)\\+\dLp{4}{u',v}{\Omega^3\nablas\eta}\ud u'
\end{multline*}
\end{proof}

\begin{remark}
The significance of the above bounds on $\psi$ is that under these uniforms bounds on $\psi$ we can establish the elliptic estimates for the conformally covariant systems on the sphere in Section~\ref{sec:elliptic}, in particular the Calderon-Zygmund estimate of Lemma~\ref{lemma:calderon:div:curl}.
\end{remark}

\begin{remark}
  The Lemma shows in particular the \emph{necessary} behavior of the torsion $\eta$. The analysis of the coupled sytem in Section~\ref{sec:coupled:mu:tilde} will provide these bounds. See also Section~\ref{sec:Linfty:eta}, and Section~\ref{sec:L4:nablas:eta}, using the propagation equations for $\eta$, and $\etab$.
\end{remark}

\subsubsection{Null expansions}

One of the most important bootstrap assumption is \eqref{BA:omega:tr:chi}, \eqref{BA:omegab:tr:chib}.
In Section~\ref{sec:L:infty:omega} we will see that the propagation equation for $2\omega-\Omega\tr\chi$ involves
\begin{equation}\label{eq:mu:breve}
          \breve{\mu}:=K-\divs\eta+\frac{1}{2}\tr\chi\tr\chib-2
\end{equation}

\begin{lemma}\label{lemma:null:expansion:mu}
  Suppose (\textbf{A:\underline{I,II}}) hold, and moreover $\Omega^2 | \breve{\mu} | \lesssim 1 $, 
then
\begin{equation*}
   |2\omega-\Omega\tr\chi|\lesssim u_+-u
\end{equation*}

\end{lemma}

\begin{proof}
See Section~\ref{sec:L:infty:omega}.  
\end{proof}

In other words to recover \eqref{BA:omega:tr:chi}, it is \emph{necessary} that $|\breve{\mu}|\lesssim r^{-2}$.

\subsubsection{Torsion system}
\label{sec:prelim:torsion}

We saw above that to obtain the required bounds on the conformal factor  it is necessary to assume that
\begin{equation}
  \Linf{u,v}{\Omega^2 \eta}+\dLp{4}{u,v}{\Omega^3\nablas\eta}\lesssim 1
\end{equation}

Recall that the torsion is defined by
\begin{equation}
  \eta(X)=\frac{1}{2}g(\nabla_X L',\Omega^2 \Lb')
\end{equation}
where $L'$, and $\Lb'$ are the null geodesic vectorfields $L'=-2\ud u^\sharp$, $\Lb'=-2\ud v^\sharp$.

We shall now discuss how these bounds are obtained from systems of elliptic equations on the sphere coupled to propagation equations.

From the above definition of \eqref{eq:mu:breve}, in the case $\rho[W]=0$, $\sigma[W]=0$, $\chibh=0$,
\begin{align}
  \divs\eta&=-\breve{\mu}+\frac{1}{4}\tr\chi\tr\chib-1\\
  \curls\eta&=0
\end{align}
However, $\breve{\mu}$ cannot decay faster than  $|\breve{\mu}|\lesssim r^{-2}$ (consider the spherically symmetric case), so the above system cannot yield the desired bound on $\eta$.




We define the  \emph{mass aspect function}  by
    \begin{equation}\label{eq:mu}
        \mu := K-\divs\eta+\frac{1}{4}\tr\chi\tr\chib-1=-\rho[W]+\frac{1}{2}(\chih,\chibh)-\divs\eta\,.
    \end{equation}

    \begin{remark}
      Note that in the case of a section $S_{u,v}$ contained in the shear-free null hypersurface $\Cb_v$,  $\chibh=0$, in an  ambient spacetime of vanishing Weyl curvature, $\rho[W]=0$, we have
    \begin{equation}
      \int_{S_{u,v}}\mu\dm{\gs}=0
    \end{equation}
    This can be interpreted that \emph{no mass} is contained in the spheres $S_{u,v}$ in de Sitter.
    For related notions of mass, and mass aspect, in asymptotically de Sitter spacetimes see  \cite{ST:15,T:15,CJK:13,CI:16,CGNP:18,CWY:16a,CWY:16b}.
  \end{remark}

Thus in the shear-free case on de Sitter $\eta$ is determined from
\begin{equation}\label{eq:div:curl:eta:intro}
  \divs\eta=-\mu\qquad \curls\eta=0
\end{equation}
and Lemma~\ref{lemma:calderon:div:curl} --- which applies if the conclusions of Lemma~\ref{lemma:prelim:conformal} hold --- implies
        \begin{equation}
      \dLp{4}{u,v}{r^2 \eta}+\dLp{4}{u,v}{r^3\nablas\eta}\lesssim  \dLp{4}{u,v}{r^3\mu}
    \end{equation}
Thus it is \emph{necessary} to establish that $r^3|\mu|\lesssim 1$.

See  Section~\ref{sec:propagation:gauss}, \ref{sec:propagation:mu} for the derivation of the propagation equation of the  mass aspect function.
We note here under what further assumptions it can be expected that $\mu$ has the desired property.

\begin{lemma}\label{lemma:mu}
  Suppose  (\textbf{A:0},\textbf{A:I}) hold. Let us also assume that:
\begin{subequations}
\begin{gather}
  \Omega^2|\chih|  \lesssim 1\qquad 
  \Omega^2 |\etab|+\Omega^2|\eta|\lesssim 1 \label{eq:mu:chi}\\
  \Omega|\nablas(\Omega\chih)| \lesssim 1\qquad \Omega^2|\nablas\eta|\lesssim 1\\  
 \Omega^4|\Laplaces \log \Omega| \lesssim 1
\end{gather}
\end{subequations}
Then
\begin{equation}
  r^3|\mu|\leq C
\end{equation}

\end{lemma}
\begin{proof}
By the propagation equation for $\mu$,
\begin{equation*}
  \begin{split}
    D(\Omega^3\mu)=& \frac{3}{2} \bigl(2\omega-\Omega\tr\chi\bigr) \Omega^3\mu-\frac{1}{2}\Omega^4\tr\chi\bigl(\divs\eta+\divs\etab+|\etab|^2)-\frac{1}{4}\tr\chib\Omega^4|\chih|^2\\
        &+2\Omega\divs(\Omega\chih)\cdot \Omega^2(\eta-\etab)+2(\Omega^2\chih,\Omega^2\nablas\eta)+2\Omega^4\chih(\etab,\etab)
  \end{split}
\end{equation*}
which implies the statement of the Lemma.
\end{proof}


Alternatively, we return to \eqref{eq:mu:breve}, and pass from $\breve{\mu}$ to
\begin{equation}
  \tilde{\mu}=\breve{\mu}-\overline{\breve{\mu}}
\end{equation}
which satisfies:
\begin{equation}
  \divs\eta = -\tilde{\mu}+\frac{1}{4}\Bigl(\tr\chi\tr\chib-\overline{\tr\chi\tr\chib}\Bigr)\,.
\end{equation}
The resulting system for $\breve{\mu}$ has advantages compared  to the corresponding system for $\mu$,\footnote{In the context of a double null foliation, the system for $\breve{\mu}$  decouples the dependence on $\nablas^2\log\Omega$, which  already featured as  one of the assumptions in Lemma~\ref{lemma:mu}. C.f.~Section~4.3.9 in  \cite{KN:03}.} see Section~\ref{sec:prelim:decoupling}.

\subsubsection{Bounds on the null lapse}

Finally we note a conclusion about the null lapse that can be achieved under a suitable gauge condition.
Recall that a foliation of $\Cb_+$ by surfaces $S_u^+$ can be constructed by solving \eqref{eq:motion}, where we are free to prescribe the average of $\log\Omega$ on $S_u^+$.

In view of \eqref{eq:eta:etab:log:Omega}, Lemma~\ref{lemma:morrey} implies
\begin{equation}
  \Linfty{\log\Omega-\overline{\log\Omega}}\lesssim \nLp{r \nablas\log\Omega}\lesssim \nLp{r \eta}+\nLp{r \etab}\lesssim r^{-1}
\end{equation}
so the required bounds on $\eta$ and $\etab$ also imply that
\begin{equation}
  r \Linfty{\log\bigl(\frac{\Omega}{r}\bigr)-\overline{\log\bigl(\frac{\Omega}{r}\bigr)}}\lesssim 1\,.
\end{equation}

\begin{lemma}\label{lemma:log:Omega:r:average}
  Assume \eqref{BA:omega:tr:chi}.
 Set $\overline{\log r^{-1}\Omega}=0$: on $\Cb_+$.
Then
\begin{equation}
  \overline{\log\Bigl(\frac{\Omega}{r}\Bigr)}\lesssim r^{-1}
\end{equation}

\end{lemma}

 In particular, the Lemma shows that with this gauge choice
  \begin{equation}
    \frac{\Omega}{r}=1+\mathcal{O}(r^{-1})\,.
  \end{equation}

\section{Representations of the de Sitter geometry}
\label{sec:desitter:charts}



\subsection{Stereographic coordinates}
\label{sec:stereographic}

Let us choose coordinates $(t,x,x')\in\mathbb{R}^5$, with $t,x\in\mathbb{R}$, and $x'\in\mathbb{R}^3$, and fix $(0,-1,0)$ as the base point of the projection. We then project the hyperboloid $H$ on the plane $x=1$, which is $3+1$-dimensional Minkowski space with the induced metric of the ambient space. In other words, we assign to every point $(t,x,x')\in H$ coordinates $(u,y)$, such that $(u,1,y)$ is on the line from $(0,-1,0)$ to $(t,x,x')$, c.f.~Fig~\ref{fig:desitter:projection}. This implies that for some $\lambda\in\mathbb{R}$: $\lambda\bigl((u,1,y)-(0,-1,0)\bigr)=(t,x,x')-(0,-1,0)$, or
\begin{equation}
  \lambda u=t\qquad 2\lambda=x+1\qquad \lambda y=x'\,.
\end{equation}
\begin{figure}[tb]
  \centering
  \includegraphics{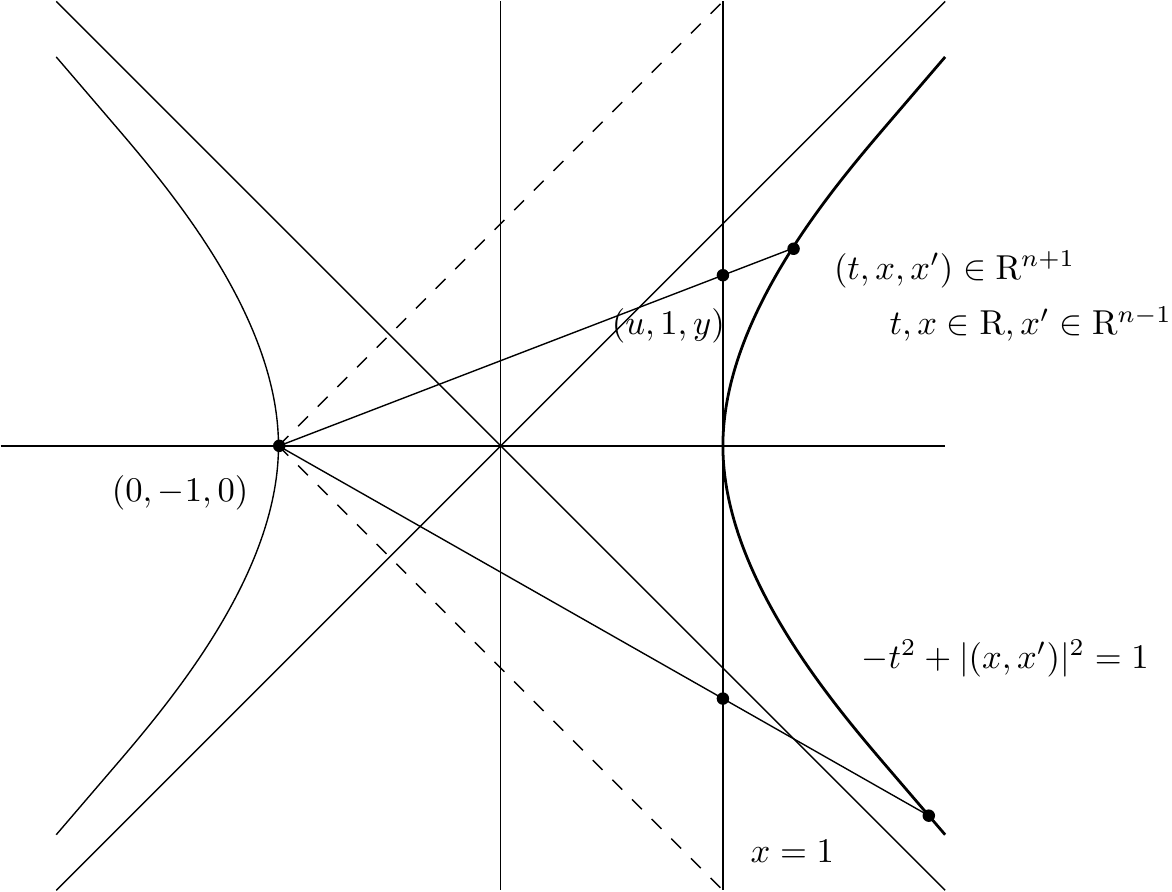}
  \caption{Stereographic projection of de Sitter spacetime on Minkowski space.}
  \label{fig:desitter:projection}
\end{figure}
Moreover, since $(t,x,x')\in H$, namely
\begin{equation}
  1=-t^2+x^2+\lvert x'\rvert^2=-\lambda^2u^2+(2\lambda-1)^2+\lambda^2\lvert y\rvert^2
\end{equation}
we obtain 
\begin{equation}
  \lambda=\frac{4}{4-u^2+\lvert y\rvert^2}
\end{equation}
and conclude that
\begin{equation}
  t=\frac{4u}{4-u^2+\lvert y\rvert^2}\,,\quad x=\frac{4+u^2-\lvert y\rvert^2}{4-u^2+\lvert y\rvert^2}\,,\quad x'=\frac{4y}{4-u^2+\lvert y\vert^2}\,.
\end{equation}

Note that the map thus defined
\begin{equation}
  \phi:(u,y)\mapsto (t,x,x')
\end{equation}
only maps a subset of $x=1$ onto a subset of $H$, c.f.~Fig.~\ref{fig:desitter:projection:domain}. More precisely, the domain is
\begin{equation}
  \mathcal{D}=\bigl\{(u,y):-u^2+\lvert y\rvert^2>-4\bigr\}\subset\mathbb{R}^{1+3}
\end{equation}
and the image is
\begin{equation}
  \mathcal{R}=\bigl\{(t,x,x')\in H: x>-1\bigr\}\subset H\subset\mathbb{R}^{5}\,.
\end{equation}

\begin{figure}[tb]
  \centering
  \includegraphics{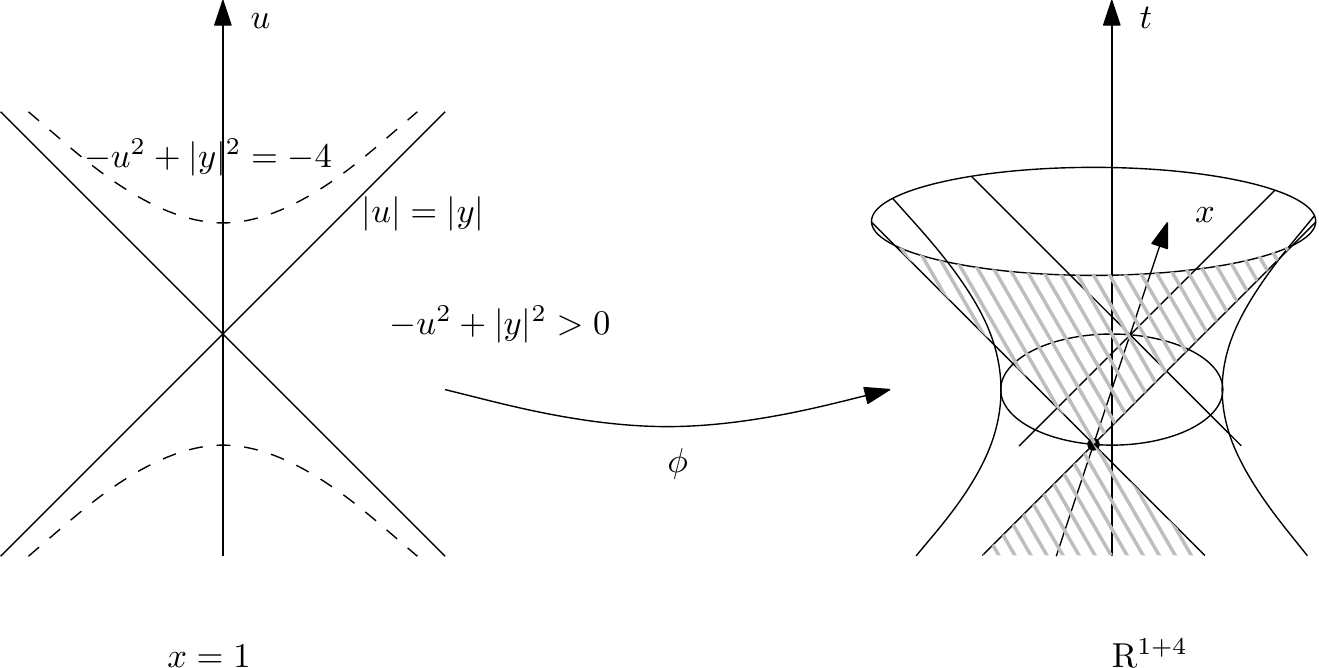}
  \caption{Domain and range of stereographic projection of de Sitter. The domain of $\phi$ is the region bounded by two sheets of spacelike hyperbolas in $3+1$-dimensional Minkowski space (dashed curves), and the range of $\phi$ is the time-like hyperboloid in $R^{5}$ minus the future and past of the base point (shaded region).}
  \label{fig:desitter:projection:domain}
\end{figure}

\begin{lemma}\label{lemma:desitter:stereo}
  In $(u,y)$ coordinates the de Sitter metric takes the form
\begin{equation}
  h=e^{-2\Phi(u,y)}\Bigl(-\ud u^2+\lvert \ud y\rvert^2\Bigr)\,,
\end{equation}
where 
\begin{equation}
  e^{\Phi(u,\phi)}=1+\frac{1}{4}(-u^2+\lvert y\rvert^2)\,.
\end{equation}
\end{lemma}

Given that the map $\phi:\mathcal{D}\to\mathcal{R}\subset H$ is explicit, and $h=m\rvert_H$ in induced by the Minkowski metric, the proof of the Lemma is an elementary calculation of the components of
\begin{equation}
  h(v_,w)=m(\ud\phi \cdot v,\ud\phi\cdot w)\,,\qquad v,w\in \mathrm{T}_{(u,y)}R^{3+1}\,.
\end{equation}

\begin{proof}
Consider a vector
\begin{equation}
  v=v^u\frac{\partial}{\partial u}+v^i\frac{\partial}{\partial y^i}\in\mathrm{T}_{(u,y)}R^{3+1}
\end{equation}
in the $x=1$ plane at a point $(u,y)$. The differential of $\phi$, $\ud \phi$ will send this vector to vector at $\phi(u,y)\in \mathbb{R}$ tangential to $H$,
\begin{equation}
  \ud \phi \cdot v\in\mathrm{T}_{\phi(u,y)}H\,,
\end{equation}
and the de Sitter metric on $H$ is simply the induced Minkowski metric $m$ of the ambient space $\mathbb{R}^{1+4}$.
We thus wish to calculate the components of
\begin{equation}
  g(v,w)=m(\ud\phi \cdot v,\ud\phi\cdot w)\,,
\end{equation}
which are
\begin{subequations}
  \begin{gather}
    g_{uu}=m(\ud\phi\cdot \frac{\partial}{\partial u},\ud\phi\cdot \frac{\partial}{\partial u})=m_{\mu\nu}\frac{\partial\phi^\mu}{\partial u}\frac{\partial\phi^\nu}{\partial u}\\
        g_{ui}=m(\ud\phi\cdot \frac{\partial}{\partial u},\ud\phi\cdot \frac{\partial}{\partial y^i})=m_{\mu\nu}\frac{\partial\phi^\mu}{\partial u}\frac{\partial\phi^\nu}{\partial y^i}\\
            g_{ij}=m(\ud\phi\cdot \frac{\partial}{\partial y^i},\ud\phi\cdot \frac{\partial}{\partial y^j})=m_{\mu\nu}\frac{\partial\phi^\mu}{\partial y^i}\frac{\partial\phi^\nu}{\partial y^j}\,.
  \end{gather}
\end{subequations}
Since
\begin{multline}
\begin{pmatrix}
  \frac{\partial t}{\partial u} & \frac{\partial t}{\partial y^i}\\
  \frac{\partial x}{\partial u} & \frac{\partial x}{\partial y^i} \\
  \frac{\partial {x'}^k}{\partial u} & \frac{\partial {x'}^k}{\partial y^i}
\end{pmatrix}
=\frac{4}{\bigl(4-u^2+\lvert y\rvert^2\bigr)^2}\times\\\times
\begin{pmatrix}
  4+u^2+\lvert y\rvert^2 & -2u y^i\\
  4u & -4y^i\\
  2uy^k & (4-u^2+\lvert y\rvert^2)\delta_{ik}-2 y^i y^k
\end{pmatrix}
\end{multline}
we obtain from the above
\begin{subequations}
  \begin{gather}
    g_{uu}=-\frac{16}{\bigl(4-u^2+\lvert y\rvert^2\bigr)^2}\\
    g_{ui}=0\\
    g_{ij}=\frac{16}{\bigl(4-u^2+\lvert y\rvert^2\bigr)^2}\delta_{ij}
  \end{gather}
\end{subequations}
or 
\begin{equation}
  g=\frac{1}{\bigl(1+\frac{1}{4}(-u^2+\lvert y\rvert^2)\bigr)^2}\Bigl(-\ud u^2+\lvert \ud y\rvert^2\Bigr)\,.
\end{equation}

\end{proof}

In view of the conformal property of the Weyl curvature $W$, it follows in particular that
\begin{equation}
  W[h]=e^{-2\Phi}W[m]=0
\end{equation}
namely that $h$ is conformally flat.

\subsection{Static coordinates}
\label{sec:static}

The idea is to introduce a \emph{spherical} coordinate system relative to a fixed observer.
We choose this observer, or the origin of this coordinate system to be the curve
\begin{equation}
  t\mapsto(t,x=\sqrt{1+t^2},x'=0)
\end{equation}
in the coordinates of the ambient space $\mathbb{R}^5$.
Each level set of $t$ in $H$ is an $\mathbb{S}^3$ of radius $1+t^2$:
\begin{equation}
  x^2+\lvert x'\rvert^2=1+t^2
\end{equation}
We can think of that $\mathbb{S}^3$ as $(0,\pi)\times\mathbb{S}^2$, where then $\lvert x'\rvert$ is the radius of the $\mathbb{S}^2$, c.f.~Fig.~\ref{fig:S3}.
\begin{figure}[tb]
  \centering
  \includegraphics[scale=0.7]{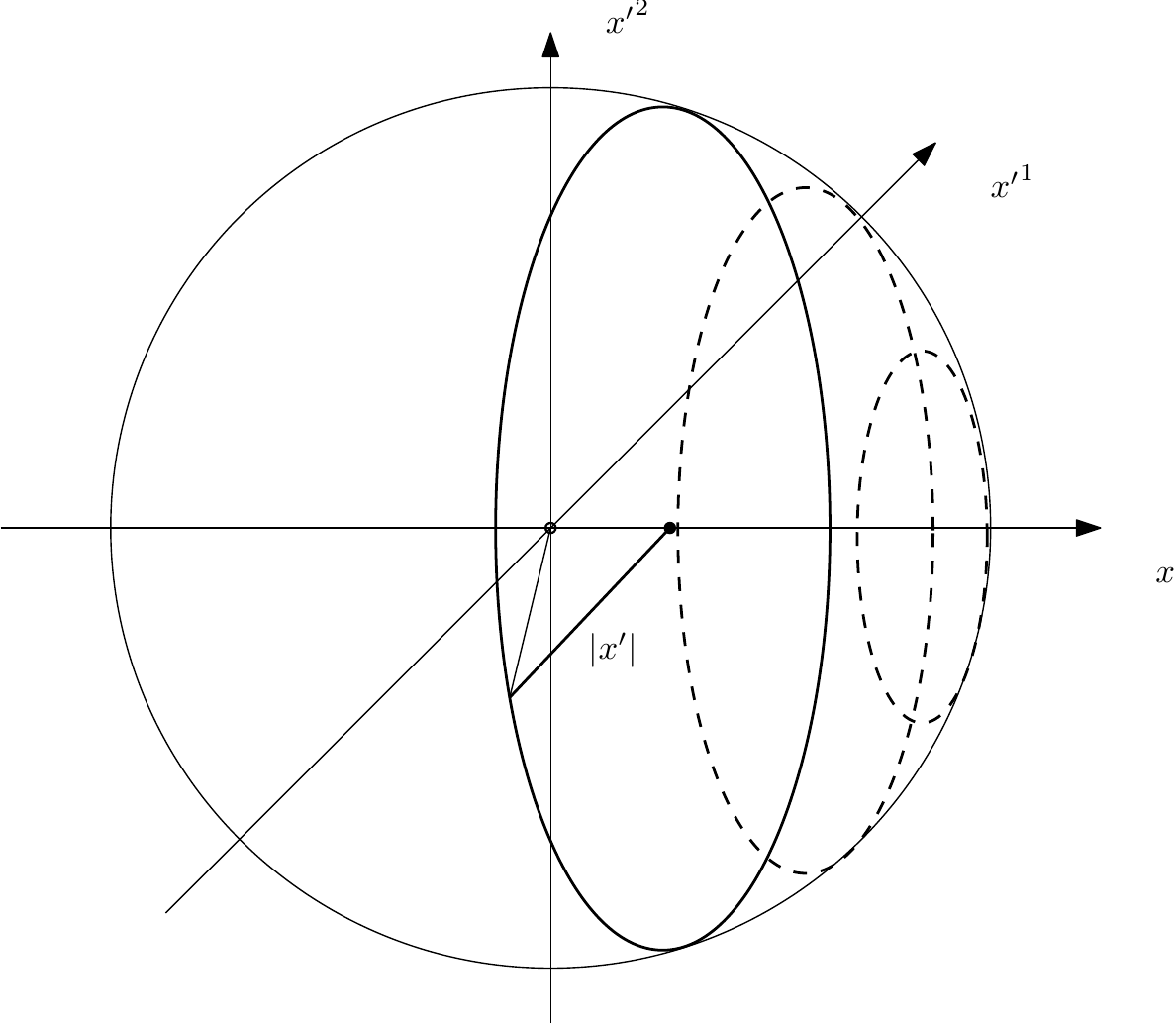}
  \caption{$\mathbb{S}^3$ as spherically symmetric space foliated by $\mathbb{S}^2$ of radius $|x'|$.}
  \label{fig:S3}
\end{figure}
Thus we introduce as one coordinate
\begin{equation}\label{eq:desitter:r:x}
  r=\lvert x'\rvert=\frac{4\lvert y\rvert}{4-u^2+\lvert y\rvert^2}
 \end{equation}
 \begin{figure}[tb]
   \centering
  \includegraphics[scale=0.8]{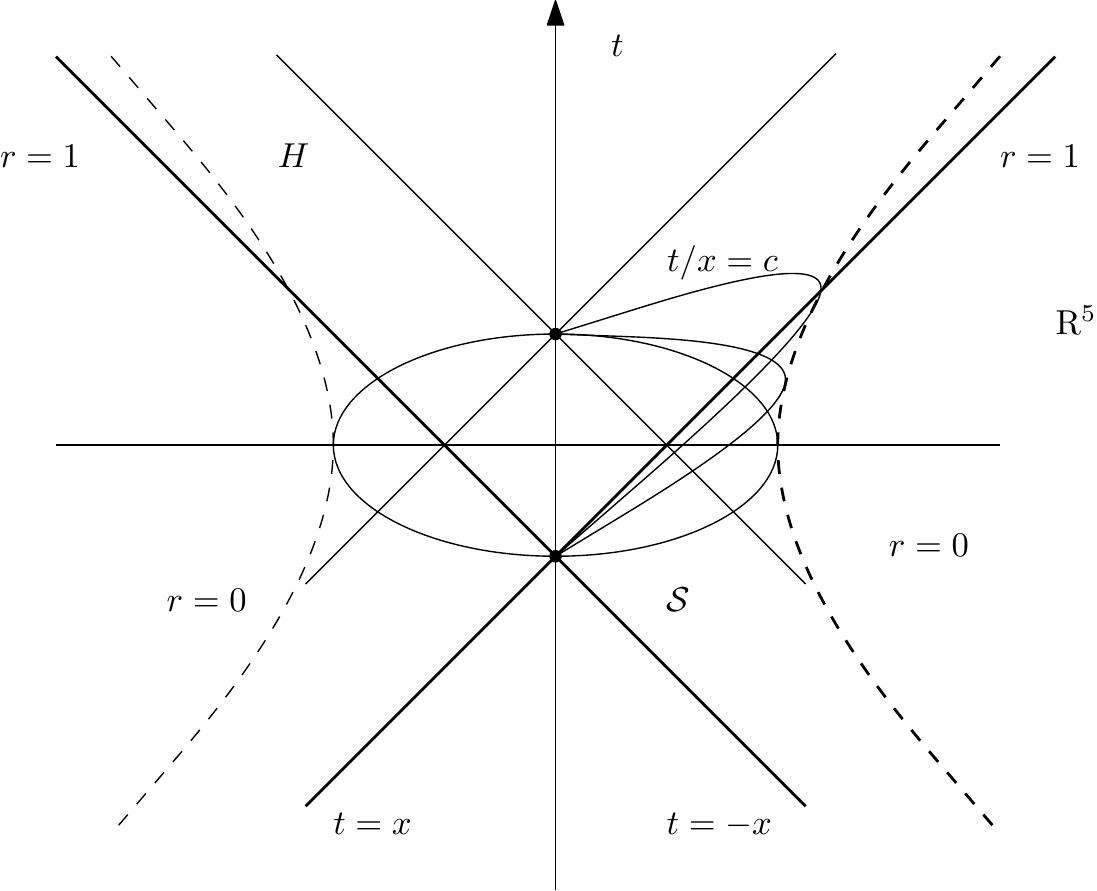}
   \caption{Level sets of $r$ in de Sitter space.}
   \label{fig:desitter:r}
\end{figure}

In addition we introduce a new time coordinate $t'$ which is constant on the level sets of $t/x$. 
An appropriate choice is 
\begin{equation}\label{eq:t:x}
  t'=\frac{1}{2}\log\Bigl\lvert\frac{x+t}{x-t}\Bigr\rvert=\frac{1}{2}\log\Bigl\lvert\frac{(2+u)^2-\lvert y\rvert^2}{(2-u)^2-\lvert y\rvert^2}\Bigr\rvert\,.
\end{equation}

\begin{lemma}
  The de Sitter metric in $(t',r)$ coordinates reads
\begin{equation}\label{eq:metric:desitter:static}
  h=-\bigl(1-r^2\bigr)\ud {t'}^2+\frac{1}{1-r^2}\ud r^2+r^2\mathring{\gamma}\,.
\end{equation}

\end{lemma}

\begin{proof}
  Let us  express the metric on $\mathbb{R}^{1+3}$ in spherically symmetric form,
\begin{equation}
  -\ud u^2+|\ud y|^2=-\ud u^2+\ud |y|^2+|y|^2\stackrel{\circ}{\gamma}\,,
\end{equation}
where $\stackrel{\circ}{\gamma}$ is the standard metric on $\mathbb{S}^2$; (we keep here $|y|$ as the notation for the radial variable, to avoid introducing more notation for yet another variable, say $\rho=|y|$).
Now we calculate
\begin{equation}
  \ud r=\frac{4}{\bigl(4-u^2+|y|^2\bigr)^2}\bigl(2u|y|\ud u+(4-u^2-|y|^2)\ud |y|\bigr)
\end{equation}
\begin{equation}
  \ud t'=\frac{4}{\bigl((2+u)^2-|y|^2\bigr)\bigl((2-u)^2-|y|^2\bigr)}\bigl((4-u^2-|y|^2)\ud u+2u|y|\ud|y|\bigr)
\end{equation}
and therefore clearly 
\begin{multline}
  -\frac{1}{16}\bigl((2+u)^2-|y|^2\bigr)^2\bigl((2-u)^2-|y|^2\bigr)^2\ud {t'}^2+\frac{1}{16}\bigl(4-u^2+|y|^2)^4\ud r^2=\\
  =\Bigl((4-u^2+|y|^2)^2-16 |y|^2\Bigr)\bigl(-\ud u^2+\ud |y|^2\bigr)
\end{multline}
which indeed yields
\begin{multline}
  g=\frac{16}{\bigl(4-u^2+|y|^2\bigr)^2}\bigl(-\ud u^2+|\ud y|^2\bigr)=\\
    =-\frac{\bigl(4-u^2+|y|^2)^2-16|y|^2}{(4-u^2+|y|^2)^2}\ud {t'}^2+\frac{\bigl(4-u^2+|y|^2\bigr)^2}{(4-u^2+|y|^2)^2-16|y|^2}\ud r^2+\frac{16|y|^2}{(4-u^2+|y|^2)^2}\stackrel{\circ}{\gamma}
\end{multline}
or
\begin{equation}\label{eq:metric:desitter:static}
g=-\bigl(1-r^2\bigr)\ud {t'}^2+\frac{1}{1-r^2}\ud r^2+r^2\mathring{\gamma}\,.
\end{equation}

\end{proof}

We have derived this form of the metric in the domain
\begin{equation}
  \mathcal{S}=\Bigl\{(t,x,x')\in H : x>|t|\Bigr\}
\end{equation}
None of the metric coefficients depend on $t'$, and $\partial_{t'}$ is orthogonal to the level sets of $t'$, thus the spacetime metric is indeed \emph{static} on $\mathcal{S}$.

The metric in this form is \emph{spherically symmetric}, the center $r=0$ being a time-like curve that lies opposite to the base point of the stereographic projection; 
since the latter was chosen arbitrarily, static coordinate patches can be introduced for any given timelike geodesic in $H$.
 The static region $\mathcal{S}$ is the intersection of the past and future of the central line $r=0$. (This is particularly easy to see in the stereographic coordinates.)

The boundary of $\mathcal{S}$ is a bifurcate null hypersurface $\mathcal{C}$: the \emph{cosmological horizon}. Indeed, $r=1$ implies $t=\lvert x\rvert$ which is a null line, namely  the straight lines in $\mathbb{R}^5$ that rules the hyperboloid $H$. In the stereographic picture this is the set $u=\pm\lvert 2-| y|\rvert$, a null hypersurface in $\mathbb{R}^{1+3}$.

We summarize the causal geometry thus described in the Penrose diagram; see Fig.~\ref{fig:desitter:penrose}.


\section{Explicit double null foliations of de Sitter}
\label{sec:desitter:double}


A convenient feature of the representation of de Sitter as a time-like hyperboloid embedded in Minkowski space is that the \emph{causal structure} of de Sitter is then induced by the ambient Minkowski space. In particular cones in Minkowski space emanating from a point on the hyperboloid trace out null hypersurfaces in de Sitter spacetime. This yields an elegant approach to construct solutions of the eikonal equation in de Sitter, which we will employ in particular to construct \emph{non-}spherical double null foliations.

\subsection{Spherically symmetric foliations}
\label{sec:spherical}

We first construct spherically symmetric double null foliations by intersecting cones emanating from antipodal lines.
This will lead to coordinates such that the metric takes the form
\begin{equation}\label{eq:metric:double:spherical}
  h=-4\bigl(r^2-1\bigr)\ud u^\ast \ud v^\ast +r^2\stackrel{\circ}{\gamma}\,.
\end{equation}
Most importantly, we will find the function $u^\ast,v^\ast:H\to\mathbb{R}$ explicitly in terms of the ambient coordinates, which will be used in Section~\ref{sec:ellipsoidal:explicit}.

For any point $(t,x,x')\in\mathbb{R}\times\mathbb{R}\times\mathbb{R}^3$ let us denote by $C^\pm_{(t,x,x')}$ the forward/backward cone emanating from $(t,x,x')$ in $\mathbb{R}^{1+4}$:
\begin{equation}
  C^\pm_{(t,x,x')}=\Bigl\{(t',y,y'):  t'-t=\pm \sqrt{(y-x)^2+\lvert y'-x'\rvert^2}\Bigr\}
\end{equation}
Now consider the two anti-podal geodesics $\Gamma,\underline{\Gamma}$ in $H$:
\begin{equation}
  \Gamma\cup\underline{\Gamma}=H\cap\{(t,x,x'):x'=0\}=\{(t,x,0):-t^2+x^2=1\}
\end{equation}
In fact let us define,
\begin{equation}
  \Gamma(t)=(t,-\tb,0)\qquad   \underline{\Gamma}(t)=(t,\tb,0)\qquad \tb=\sqrt{1+t^2}
\end{equation}
and
\begin{equation}
 C(t)=C^+_{\Gamma(t)}\cap H\qquad \underline{C}(t)=C^+_{\underline{\Gamma}(t)}\cap H
\end{equation}

\begin{lemma}\label{lemma:C:t}
  The null hypersurfaces $C(t)$ and $\underline{C}(\tu)$ are given by
  \begin{gather}
    C(t)=\Bigl\{(s,x,x'): x=-\frac{1+st}{\tb}, |x'|=\frac{s-t}{\tb}, s\geq t\Bigr\}\\
    \underline{C}(\tu)=\Bigl\{(s,x,x'): x=\frac{1+s\tu}{\langle \tu\rangle}, |x'|=\frac{s-\tu}{\langle \tu \rangle}, s\geq \tu\Bigr\}
\end{gather}
Moreover, for any $t+\tu<0$ the intersection $S(t,\tu)=C(t)\cap\underline{C}(\tu)$ is a round $2$-spheres of radius
\begin{equation}
  |x'|=-\frac{\langle t\rangle+\langle \tu \rangle}{\tu+t}
\end{equation}

\end{lemma}
\begin{remark}
  In particular ``future null infinity'' can be identified with\footnote{The null hypersurfaces $C(t)$ and $\underline{C}(-t)$ \emph{do not} intersect, in other words $S(t,-t)$ is not a sphere in $H$, and can thus only be thought of as ``attached'' to $H$ in a larger space $H\cup\mathcal{I}^+$.}
  \begin{equation}
    \mathcal{I}^+\dot{=}\bigcup_{t\in\mathbb{R}}S(t,-t)
  \end{equation}

\end{remark}

\begin{proof}
If $(s,x,x')\in C(t)$ then it satisfies the equations
\begin{gather}
  s-t=\sqrt{(x+\langle t\rangle)^2+|x'|^2}\\
  -s^2+x^2+|x'|^2=1\label{eq:sx:hyperboloid}
\end{gather}
which implies
\begin{equation}\label{eq:sx:sx}
  s=-\frac{1+x\tb}{t}
\end{equation}
Similarly, for $(s,x,x')\in \underline{C}(\tu)$,
\begin{equation}\label{eq:sx:sx:b}
  s=-\frac{1-x\langle \tu\rangle}{\tu}
\end{equation}
Solving these for $x$ we obtain, respectively
\begin{equation}
  x=-\frac{1+st}{\tb}\qquad x=\frac{1+s\tu}{\langle \tu \rangle}
\end{equation}
It also follows from \eqref{eq:sx:hyperboloid} and \eqref{eq:sx:sx}, and \eqref{eq:sx:sx:b} respectively, that
\begin{equation}
  |x'|=\frac{|\tb+x|}{|t|}\qquad   |x'|=\frac{|\langle \tu\rangle-x|}{|\tu|}
\end{equation}
and thus, respectively,
\begin{equation}\label{eq:sx:xprime}
  |x'|=\frac{|t-s|}{\tb}\qquad |x'|=\frac{|\tu-s|}{\langle \tu\rangle}
\end{equation}
Since $s>t$, and $s>\tu$, respectively we can drop the absolute value accordingly.
For $(s,x,x')\in S(t,\tu)$ then both formulas for $x'$ are satisfied and we can solve for
\begin{equation}
  s=-\frac{1+\tb\langle\tu\rangle-t\tu}{\tu+t}
\end{equation}
Inserting back into the formula for \eqref{eq:sx:xprime} we obtain
\begin{equation}
  |x'|=-\frac{\tb+\langle \tu\rangle}{\tu+t}\,.
\end{equation}
\end{proof}

The \emph{cosmological horizons} of the observers $\Gamma$, and $\underline{\Gamma}$ are:
\begin{gather}
  \mathcal{C}=\lim_{t\to-\infty} C_{\Gamma(t)}^+\cap H=\Bigl\{(s,x,x'): x=s, |x'|=s, s\in\mathbb{R}\Bigr\}\\
  \overline{\mathcal{C}}=\lim_{t\to-\infty} C_{\underline{\Gamma}(t)}^+\cap H=\Bigl\{(s,x,x'): x=-s, |x'|=1, s\in\mathbb{R}\Bigr\}
\end{gather}

Recall that with $t'=t'(x,t)$ defined as in \eqref{eq:t:x} the metric in $(t',r)$-coordinates takes the form \eqref{eq:metric:desitter:static};
in fact in the region $t>|x|$, let us set
\begin{equation}\label{eq:t:x:sign}
  t'=\frac{1}{2}\log \frac{t-x}{t+x}
\end{equation}
Since null coordinates are given by
\begin{gather}
  u^\ast=\frac{1}{2}\bigl(r^\ast+t'\bigr)\qquad v^\ast=\frac{1}{2}\bigl(r^\ast-t'\bigr) \label{eq:u:ast:v:ast:r:t}\\
  -4(r^2-1)\ud u^\ast \ud v^\ast = -\frac{1}{r^2-1}\ud r^2+(r^2-1)\ud t'^2
\end{gather}
where
\begin{equation}\label{eq:r:ast:r}
  r^\ast=-\int_r^\infty \frac{\ud r}{r^2-1}=\frac{1}{2}\log  \frac{r-1}{r+1}
\end{equation}
we are led to the following functions whose level sets are the null hypersurfaces $C_t$:
\begin{lemma}\label{lemma:levels}
  Let $u,v:H\to\mathbb{R}$ be the following functions:
  \begin{subequations}
    \begin{gather}
      u(t,x,x')=\frac{\lvert x'\rvert+1}{\lvert x'\rvert-1}\frac{t+x}{t-x}\\
      v(t,x,x')=\frac{\lvert x'\rvert+1}{\lvert x'\rvert-1}\frac{t-x}{t+x}
    \end{gather}
  \end{subequations}
  then for all $t,\underline{t}\in\mathbb{R}$,
  \begin{subequations}
    \begin{gather}
      C(t)=\Bigl\{(t,x,x'):u(t,x,x')=\frac{\langle t\rangle- t}{\langle t\rangle+ t}\Bigr\}\\
      \underline{C}(t)=\Bigl\{(t,x,x'):v(t,x,x')=\frac{\tb-t}{\tb+t}\Bigr\}
    \end{gather}
  \end{subequations}

\end{lemma}

\begin{proof}
By Lemma~\ref{lemma:C:t} we have that for any $(s,x,x')\in C(t)$:
  \begin{subequations}
    \begin{gather}
      \frac{\lvert x'\rvert+1}{\lvert x'\rvert-1}=\frac{s-(t-\langle t\rangle)}{s-(t+\langle t\rangle)}\\
      \frac{s+x}{s-x}=\frac{s\, (\tb-t)-1}{s\,(\tb+t)+1}
    \end{gather}
  \end{subequations}
and thus by direct computation:
\begin{equation}
  \frac{s+x}{s-x}\frac{\lvert x'\rvert+1}{\lvert x'\rvert-1}=\frac{\tb-t}{\tb+t}\frac{s^2+\bigl((\tb-t)-(\tb-t)^{-1}\bigr)s-1}{s^2+\bigl((t+\tb)^{-1}-(t+\tb)\bigr)s-1}
\end{equation}
and the result follows, because the coefficient to both quadratics in the second term $=-2t$.
Moreover it follows from Lemma~\ref{lemma:C:t} that if $(s,x,x')\in \underline{C}(\tu)$, then $(s,-x,x')\in C(\tu)$, so
\begin{equation}
  v(t,x,x')=u(t,-x,x')=\frac{\tb-t}{\tb+t}\,.
\end{equation}

\end{proof}

Note in particular that
\begin{subequations}
\begin{gather}
  \lim_{t\to-\infty}u\rvert_{C_t}=\infty\qquad \lim_{t\to 0}u\rvert_{C_t}=1\qquad   \lim_{t\to\infty}u\rvert_{C_t}=0\\
  \lim_{t\to-\infty}v\rvert_{\underline{C}_t}=\infty\qquad\lim_{t\to 0}v\rvert_{\underline{C}_t}=1\qquad   \lim_{t\to\infty}v\rvert_{\underline{C}_t}=0
\end{gather}
\end{subequations}
A suitable set of null coordinates are thus obtained by setting
\begin{equation}
  2 u^\ast(t,x,x')=-\frac{1}{2}\log u(t,x,x') \qquad 2 v^\ast(t,x,x')=-\frac{1}{2}\log   v(t,x,x')
\end{equation}
which then have the range $(-\infty,\infty)$.

Since apparently 
\begin{equation}
  2 u^\ast\rvert_{C(t)}=\frac{1}{2}\log u^{-1}\rvert_{C(t)}=\arsinh(t)
\end{equation}
we will simply denote by 
\begin{subequations}
\begin{gather}
  C_{u^\ast}=C(\sinh(2u^\ast))\\
    C_{v^\ast}=\Cb(\sinh (2v^\ast))
\end{gather}
\end{subequations}
and
\begin{equation}
  S_{u^\ast,v^\ast}=C_{u^\ast}\cap\Cb_{v^\ast}
\end{equation}

\begin{corollary}\label{cor:t:x:in:S:u:v}
  If  $(t,x,x')\in S_{u^\ast,v^\ast}$ then
  \begin{gather}
      t=-\frac{\cosh(u_\ast-v_\ast)}{\sinh(u_\ast+v_\ast)} \qquad     x=\frac{\sinh(u_\ast-v_\ast)}{\sinh(u_\ast+v_\ast)}\\
    |x'|=-\frac{\cosh(u_\ast+v_\ast)}{\sinh(u_\ast+v_\ast)}=\Bigl( \tanh |r^\ast| \Bigr)^{-1}
  \end{gather}

\end{corollary}

\begin{proof}
Let us begin with the formula for the radius:
  
  On one hand, we know that
  \begin{equation*}
    S_{u^\ast,v^\ast}=C(\sinh 2u^\ast)\cap \Cb(\sinh 2v^\ast)
  \end{equation*}
  and thus by Lemma~\ref{lemma:C:t},
  \begin{equation*}
  |x'|=-\frac{\langle \sinh 2u^\ast \rangle+\langle \sinh 2v^\ast \rangle}{\sinh 2u^\ast+\sinh 2v^\ast}=-\frac{ \cosh 2u^\ast + \cosh 2v^\ast}{\sinh 2u^\ast+\sinh 2v^\ast}
\end{equation*}
On the other hand it follows directly from \eqref{eq:u:ast:v:ast:r:t} that
\begin{equation*}
  u^\ast+v^\ast=\frac{1}{2}\log \frac{r-1}{r+1}
\end{equation*}
and hence
\begin{equation*}
  r=-\frac{e^{2u^\ast}+e^{-2v^\ast}}{e^{2u^\ast}-e^{-2v^\ast}}=-\frac{\cosh{u^\ast+v^\ast}}{\sinh{u^\ast+v^\ast}}
\end{equation*}
To see that the two formulas are the same simply insert \eqref{eq:u:ast:v:ast:r:t} in the first one, to obtain also
\begin{equation*}
  |x'|=-\frac{\cosh r^\ast}{\sinh r^\ast}
\end{equation*}
This formula follows of course also directly from \eqref{eq:r:ast:r}.

Let us now turn to the formulas for $t$, and $x$:

By \eqref{eq:u:ast:v:ast:r:t} and \eqref{eq:t:x:sign} we have
\begin{equation*}
  u_\ast-v_\ast=\frac{1}{2}\log \frac{t-x}{t+x}
\end{equation*}
hence
\begin{equation*}
  \frac{t}{x}=-\frac{1}{\tanh(u_\ast-v_\ast)}\,.
\end{equation*}
Since $(t,x,x')\in H$ we can now solve for $x$:
\begin{equation*}
  x^2=\frac{\sinh^2(u_\ast-v_\ast)}{\sinh^2(u_\ast+v_\ast)}
\end{equation*}
and infer the sign from the condition that for $v_\ast=0$ we must have $x=1$.
The formula for $t$ then follows.

\end{proof}

Finally, these are of course precisely the double null coordinates introduced above, such that the metric takes the form \eqref{eq:metric:double:spherical}, namely
\begin{equation}
  h=-4\Omega^2\ud u^\ast \ud v^\ast +r^2\stackrel{\circ}{\gamma}_{AB}\ud \vartheta^A\ud\vartheta^B
\end{equation}
where now
\begin{equation}\label{eq:Omega:ast}
  \Omega^2(t,x,x')=r^2-1=\lvert x'\rvert^2-1\,.
\end{equation}

\subsection{Examples of ellipsoidal double null foliations}
\label{sec:ellipsoidal}


We shall now construct a double null foliation for which the intersections are not spheres but \emph{ellipsoids}. Recall that the spherically symmetric foliation was constructed by considering intersections of cones emanating from antipodal geodesics $\Gamma$, $\underline{\Gamma}$. We shall now consider cones emanating from points which are slightly displaced from these geodesics.

For simplicity let us leave the null hypersurfaces $\Cb(t)$ emanating from points $\underline{\Gamma}(t)=(t,\langle t\rangle,0)$ unchanged. In fact, let us fix in particular the null hypersurface $\Cb(0)=C^+_{\underline{\Gamma}(0)}\cap H$ emanating from $(0,1,0)$, and observe that $\Cb(0)$ is contained in the hyperplane $x=1$:
\begin{equation}
  \Cb(0)=C^+_{\underline{\Gamma}(0)}\cap H=\Bigl\{(t,1,x'):t=\lvert x'\rvert\Bigr\}
\end{equation}
Note that indeed on $\Cb(0)$ we have that $v\rvert_{\Cb(0)}=1$, and thus $v^\ast\rvert_{\Cb(0)}=0$.
On $\Cb(0)$ we also have a simple relation between $v^\ast$ and $t=\lvert x'\vert>1$:
\begin{equation}
  u^\ast=\frac{1}{2}\log\frac{t-1}{t+1}\qquad\text{: on }\Cb(0)
\end{equation}
or
\begin{equation}\label{eq:r:Cb:zero}
  \lvert x'\rvert = t =-\frac{e^{2 u^\ast}+1}{e^{2 u^\ast}-1}\qquad u^\ast\in(-\infty,0)\qquad\text{: on }\Cb(0)\,.
\end{equation}

In Section~\ref{sec:spherical} we considered the cones intersected with $H$ and vertices on $\Gamma$:
\begin{equation}
  C(\epsilon)=C^+_{\Gamma(\epsilon)}\cap H\,,\quad \Gamma(\epsilon)=(\epsilon,-\sqrt{1+\epsilon^2},0)\,,\quad \epsilon<0
\end{equation}
Then the intersection of the null hypersurfaces $C(\epsilon)$ and $\Cb(0)$ is a sphere
\begin{equation}
  S(\epsilon)=\Cb(0)\cap C(\epsilon)=\Bigl\{(t,1,x'):t=\lvert x'\rvert=\frac{1+\sqrt{1+\epsilon^2}}{\lvert \epsilon\rvert}\Bigr\}
\end{equation}
Also note that $u^\ast\rvert_{C(\epsilon)}= \frac{1}{2}\epsilon+\mathcal{O}(\epsilon^2)$.
In particular the ``sphere at infinity'' where $\underline{C}(0)$, and $C(0)$ ``meet in the Penrose diagram'' is identified with $S(0)$.

Let us now consider a null hypersurface in $H$ emanating from a vertex slightly displaced from $\Gamma$:
\begin{equation}
  C(\epsilon,\delta)= C^+_{(\epsilon,x(\epsilon,\delta),x'_\delta)}\cap H\qquad \epsilon<0,0<\delta<\sqrt{1+\epsilon^2}
\end{equation}
where $x'_\delta\in\mathbb{R}^3$ with $\lvert x'_\delta\rvert=\delta$,  and $x(\epsilon ,\delta)=-\sqrt{1+\epsilon^2-\delta^2}$.
We also introduce an angle in the plane spanned by $x_\delta'$, and $x'$ such that
\begin{equation}
  \langle x',x'_\delta \rangle =\lvert x'\rvert\lvert x'_\delta\rvert\cos\vartheta
\end{equation}
It is then easy to calculate that
\begin{equation}
\begin{split}
  S(\epsilon,\delta)&=C(\epsilon,\delta)\cap \Cb(0)\\&=\Bigl\{(t,1,x'): \langle x',x'_\delta\rangle = \lvert x'\rvert \delta \cos\vartheta, \lvert x'\rvert(\vartheta)=t(\vartheta)=\frac{1+\sqrt{1+\epsilon^2-\delta^2}}{\lvert \epsilon\rvert+\delta\cos\vartheta}\Bigr\} 
\end{split}
\end{equation}
This means in particular that \emph{all sections $S{(\epsilon,\delta)}$ are ellipsoids} (with eccentricity $\delta/\epsilon$). Moreover, for fixed $\epsilon<0$ the deformation of the sphere $S(\epsilon)$ (which lies in $\Cb(0)$ ``away from infinity'') to the ellipsoid $S(\epsilon,\delta)$ cannot ``move any point to infinity'' \emph{as long as $\delta<\lvert\epsilon\rvert$.} However, if $\delta=\lvert\epsilon\rvert$ then $S(\epsilon,\delta)$ ``touches infinity'' at exactly one point (the antipodal point to $x_\delta'$); meanwhile it is clear that the intersections of ${C}(\delta,\epsilon)$ with the cosmological horizon $\underline{\mathcal{C}}$ is a ``small deformation'' of $\underline{\mathcal{C}}\cap {C}(\epsilon)$ (as we will show below). Finally, if $\delta>\lvert\epsilon\rvert$ then only the ``hemisphere'' $0\leq \vartheta\leq \arccos\epsilon/\delta$ of $S(\epsilon,\delta)$ remains in the spacetime. This case occurs also when we move the sphere $S({\epsilon,\delta})$ ``to infinty'' by taking $\epsilon\to 0$ while keeping $\delta>0$ fixed.

\begin{figure}[bt]
  \centering
  \includegraphics[scale=1.5]{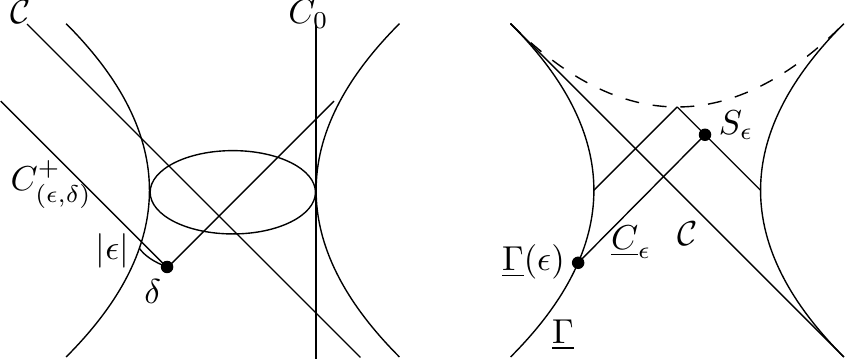}
  \caption{Construction of a double null foliation in de Sitter with ellipsoidal sections depicted in the ambient $4+1$-dimensional Minkowski space (left) and the Penrose diagram (right).}
  \label{fig:foliation:ellipsoid}
\end{figure}

\subsection{Explicit parametrizations of ellipsoidal foliations}
\label{sec:ellipsoidal:explicit}

Another way to parametrize the transformation of the foliation above is to introduce new coordinates $(t,\mathring{x},\mathring{x}')$ such that the vertex of $C(\epsilon,\delta)$ is at $(\epsilon,-\langle \epsilon\rangle,0)$ in the new coordinates: this can obviously be achieved by a \emph{rotation}. Then in these coordinates we can write down the level set of the cone emanating from this point, which we then can express in the original coordinates $(t,x,x')$ as desired.

To that end, consider coordinates $(t,\mathring{x},\mathring{x}')$ as follows:
\begin{equation}
  (t,\mathring{x},\mathring{x}'_1,\mathring{x}'_2,\mathring{x}'_3)=(t,x\cos\varphi-x'_1\sin\varphi,x\sin\varphi+x'_1\cos\varphi,x'_2,x'_3)
\end{equation}
Then, for a given $0<\varphi<\pi$ (now playing the role of $\delta$ above), consider the cone with vertex at $o=(t=\epsilon,\mathring{x}=-\langle\epsilon\rangle,\mathring{x}'=0)$, $\epsilon<0$. According to Lemma~\ref{lemma:levels}
\begin{equation}
        C_{\epsilon,\varphi}:=C_{o}^+\cap H=\Bigl\{(t,x,x'):u_\varphi(t,x,x')=\frac{\langle \epsilon\rangle+\lvert \epsilon\rvert}{\langle \epsilon\rangle-\lvert \epsilon\rvert}\Bigr\}
\end{equation}
where
\begin{equation}
\begin{split}
     u_\varphi(t,x,x')&=u(t,\mathring{x}(x,x'),\mathring{x}'(x,x'))\\
&=\frac{x\cos\varphi-x'_1\sin\varphi+t}{x\cos\varphi-x'_1\sin\varphi-t}\frac{1+\lvert \mathring{x}'\rvert}{1-\lvert \mathring{x}'\rvert}\,, \quad \mathring{x}'=(x\sin\varphi+x'_1\cos\varphi,x'_2,x'_3)
   \end{split}
 \end{equation}
  We shall express the level sets of $u_\varphi$ in $(u,v;\vartheta^1,\vartheta^2)$ coordinates. In fact, let $(\vartheta^1,\vartheta^2)$ be polar coordinates on the spheres of radius $r=\lvert x'\rvert$ such that $x_1'=r\cos\vartheta^1, x_2'=r\sin\vartheta^1\cos\vartheta^2, x_3'=r\sin\vartheta^1\sin\vartheta^2$.
\begin{lemma}
  Let $(u^\ast,v^\ast;\vartheta^1,\vartheta^1)$ be the (spherical) double null coordinates introduced in Section~\ref{sec:spherical}. Then for any $\epsilon<0$, and $-\pi<\varphi<\pi$, the level sets of
\begin{equation}
  2 u^\ast_\varphi(u^\ast,v^\ast;\vartheta^1,\vartheta^2)=-\frac{1}{2}\log  u_\varphi(u^\ast,v^\ast;\vartheta^1,\vartheta^2)
\end{equation}
are null hypersurfaces in $H$, where
\begin{multline}
  u_\varphi(u^\ast,v^\ast;\vartheta^1,\vartheta^2)=\frac{\sinh(\us-\vs)\cos\varphi+\cosh(\us+\vs)\cos\vartheta^1\sin\varphi-\cosh(\us-\vs)}{\sinh(\us-\vs)\cos\varphi+\cosh(\us+\vs)\cos\vartheta^1\sin\varphi+\cosh(\us-\vs)}\times\\
  \times \frac{\sinh(\us+\vs)-|r'_\varphi(u^\ast,v^\ast,\vartheta^1)|}{\sinh(\us+\vs)+|r'_\varphi(u^\ast,v^\ast,\vartheta^1)|}
\end{multline}
where $r_\varphi'$ is given by \eqref{eq:r:phi:prime}.

In fact,
\begin{equation}
  \Bigl\{(u^\ast,v^\ast;\vartheta^1,\vartheta^2):u_\varphi(u^\ast,v^\ast;\vartheta^1,\vartheta^2)=\frac{\langle \epsilon\rangle+\lvert \epsilon\rvert}{\langle \epsilon\rangle-\lvert \epsilon\rvert}\Bigr\}=C_{\epsilon,\varphi}\,.
\end{equation}
\end{lemma}

\begin{proof}
  Recall that Corollary~\ref{cor:t:x:in:S:u:v} allows us to substitute for $(t,x,|x'|)$ in terms of $(\us,\vs)$, if $(t,x,x')\in S_{\us,\vs}$.

As already discussed above, the null hypersurface $C_{\epsilon,\varphi}$ emanates from the point $(t,\mathring{x},\mathring{x}')=(\epsilon, -\langle\epsilon\rangle,0)$, $\epsilon<0$, and is given according to Lemma~\ref{lemma:levels} as the level set
\begin{equation}\label{eq:v:phi:epsilon}
        u(t,\mathring{x},\mathring{x}')=\frac{\mathring{x}+t}{\mathring{x}-t}\frac{1+\lvert \mathring{x}'\rvert}{1-\lvert \mathring{x}'\rvert}=\frac{\langle \epsilon\rangle+\lvert \epsilon\rvert}{\langle \epsilon\rangle-\lvert \epsilon\rvert}
\end{equation}
Since
\begin{subequations}
  \begin{gather}
    \mathring{x}=x\cos\varphi-x'_1\sin\varphi\\
    \mathring{x}_1'=x\sin\varphi+x'_1\cos\varphi\\
    \lvert\mathring{x}'\rvert=\sqrt{(x^2-x_1'^2)\sin^2\varphi+2xx_1'\sin\varphi\cos\varphi+\lvert x'\rvert^2}\\
    x_1'=\lvert x'\rvert \cos\vartheta^1
  \end{gather}
\end{subequations}
we can find explicitly:
\begin{equation}
  u_\varphi(t,x,x')=u(t,\mathring{x}(x,x'),\mathring{x}(x,x'))
\end{equation}
First note that
\begin{equation}
  \frac{\mathring{x}+t}{\mathring{x}-t}=\frac{\sinh(\us-\vs)\cos\varphi+\cosh(\us+\vs)\cos\vartheta^1\sin\varphi-\cosh(\us-\vs)}{\sinh(\us-\vs)\cos\varphi+\cosh(\us+\vs)\cos\vartheta^1\sin\varphi+\cosh(\us-\vs)}
\end{equation}
and secondly,
\begin{subequations}
\begin{equation}
  \frac{1+\lvert \mathring{x}'\rvert}{1-\lvert \mathring{x}'\rvert}=\frac{|\sinh(\us+\vs)|+|r'_\varphi(u^\ast,v^\ast,\vartheta^1)|}{|\sinh(\us+\vs)|-|r'_\varphi(u^\ast,v^\ast,\vartheta^1)|}
\end{equation}
where
\begin{multline}\label{eq:r:phi:prime}
  {r'_\varphi}^2=\sinh^2(\us-\vs)\sin^2\varphi-\cosh^2(\us+\vs)\cos^2\vartheta^1\sin^2\varphi\\
  -2\sinh(\us-\vs)\cosh(\us+\vs)\cos\vartheta^1\sin\varphi\cos\varphi+\cosh^2(\us+\vs)
\end{multline}
\end{subequations}
Since by definition $r^\ast<0$, $|\sinh(\us+\vs)|=-\sinh(\us+\vs)$.
\end{proof}

\subsubsection{Small displacement angles}

We are interested in the intersections $\underline{S}_{\epsilon,\varphi}:=C_{\epsilon,\varphi}\cap \underline{\mathcal{C}}$ and $S_{\epsilon,\varphi}^\infty:=C_{\epsilon,\varphi}\cap \Cb(0)$,
and will study these in the ``small displacement angle approximation'' when $|\varphi|\ll 1$.

$S_{\epsilon,\varphi}^\infty$: On $\Cb(0)$ we have $\vs=0$ and the formula for $u_\varphi$ reduces to:
\begin{multline}
  u_\varphi(u^\ast,0;\vartheta^1,\vartheta^2)=\frac{\sinh(\us)\cos\varphi+\cosh(\us)\cos\vartheta^1\sin\varphi-\cosh(\us)}{\sinh(\us)\cos\varphi+\cosh(\us)\cos\vartheta^1\sin\varphi+\cosh(\us)}\times\\
  \times \frac{\sinh(\us)-|r'_\varphi(u^\ast,0,\vartheta^1)|}{\sinh(\us)+|r'_\varphi(u^\ast,0,\vartheta^1)|}
\end{multline}
where
\begin{multline}
  {r'_\varphi}^2=\sinh^2(\us)\sin^2\varphi-\cosh^2(\us)\cos^2\vartheta^1\sin^2\varphi\\
  -2\sinh(\us)\cosh(\us)\cos\vartheta^1\sin\varphi\cos\varphi+\cosh^2(\us)
\end{multline}
For small $|\varphi|\ll 1$, and $|\epsilon|\ll 1$ the sphere $S_{\epsilon,\varphi}^\infty$ is thus well approximated by
\begin{equation}
  S_{\epsilon,\varphi}^\infty \simeq \Bigl\{ (\us,0;\vartheta^1,\vartheta^2) : u_\varphi^\infty = \frac{1+|\epsilon|}{1-|\epsilon|}\Bigr\}
\end{equation}
where
\begin{equation}
  u_\varphi^\infty(u^\ast;\vartheta^1,\vartheta^2)=\frac{\us+\cos\vartheta^1\varphi-1}{\us+\cos\vartheta^1\varphi+1} \frac{\us-1}{\us+1}
\end{equation}
Hence we easily find an approximate solution for $\us(\vartheta^1)$ on $S_{\epsilon,\varphi}^\infty$:
\begin{equation} \label{eq:us:theta:approx}
  2\us(\vartheta^1)\simeq -|\epsilon|-\varphi\cos\vartheta^1
\end{equation}

$\underline{S}_{\epsilon,\varphi}$: The intersection of $C_{\epsilon,\varphi}$ with the horizon $\underline{\mathcal{C}}$ can be viewed as the limiting sphere
\begin{equation}
  \underline{S}_{\epsilon,\varphi} = \lim_{\vs\to-\infty}C_{\epsilon,\varphi}\cap \Cb_{\vs}
\end{equation}
and is thus well approximated by
\begin{equation}
  \underline{S}_{\epsilon,\varphi} \simeq \Bigl\{ (\us,-\infty;\vartheta^1,\vartheta^2) : 2\us =- |\epsilon| \Bigr\}
\end{equation}

This can easily be verified as follows: As above $\underline{S}_{\epsilon,\varphi}$ is approximated by the level set
\begin{equation}\label{eq:ub:phi:level}
\underline{u}_\varphi(\us;\vartheta^1,\vartheta^2) = \frac{1+|\epsilon|}{1-|\epsilon|}  
\end{equation}
where $\underline{u}_{\varphi}$ is obtained from $u_\varphi$ by taking $\vs\ll \us\leq 0$, or replacing $\sinh(\us-\vs)$ by $\sinh(|\vs|)$, etc. Then taking the formal limit $|\vs|\to\infty$, we are left with an expression which we further approximate for small angles $|\varphi| \ll 1$:
\begin{equation*}
  \underline{u}_\varphi(u^\ast;\vartheta^1,\vartheta^2)=\frac{1+\cos\vartheta^1\varphi-1}{1+\cos\vartheta^1\varphi+1}\frac{-1 +\cos\vartheta^1\varphi-1}{-1 -\cos\vartheta^1\varphi+1}
\end{equation*}
To linear order in $\epsilon$ and $\varphi$ the equation \eqref{eq:ub:phi:level} is then identically satisfied, for all $\vartheta^1$.
Thus this level set of $\underline{u}_\varphi$ is well approximated by a level set of $\us$.

\smallskip
In conclusion, we have considered a change of the double null foliation from one with spherical sections $S_{\us,\vs}$ constructed from the intersection of cones  centered on antipodal curves $\Gamma$ and $\underline{\Gamma}$, to one with ellipsoidal sections constructed from the intersection of cones $C^+_{o(\varphi)}$ centered at points which are displaced from the antipodal curves by an angle $\varphi$.
We considered the intersection of a level set of the new optical function $u_\varphi$ induced from a cone with vertex on the sphere $t=\epsilon<0$ with 1) the cosmological horizon $\underline{\mathcal{C}}$ and 2) the null hypersurface $\Cb(0)$, which intersects $C(0)$ in a sphere at infinity. For small displacement angles $|\varphi|\ll 1$, we find that 1) $\underline{S}_{\epsilon,\varphi}$ is well approximated by $S_{\frac{\epsilon}{2},-\infty}$ namely the corresponding section of the spherical foliation while 2) the sphere $S_{\epsilon,\varphi}^\infty$ is parametrized by \eqref{eq:us:theta:approx}. In particular for $\varphi=|\epsilon|$ then $\us(\pi)=0$, and $S_{\epsilon,\varphi}^\infty$ touches infinity exactly at one point. For $\varphi>|\epsilon|$ an entire annular region of $S_{\epsilon,\varphi}^\infty$ is not contained in $H$.

\subsubsection{Transformation of the optical structure coefficients}

We will now derive the transformation formulas for the optical structure coefficients in the small angle approximation for the examples constructed above. See Section~\ref{sec:null:structure} for the definition of structure coefficients referred to in this section.

In the previous section we have constructed an explicit family of gauge transformations (parametrized by $\lvert\varphi\rvert<\pi$)
\begin{equation}\label{eq:gauge:change}
  \begin{split}
    u^\ast &\mapsto u=u_\varphi^\ast(u^\ast,v^\ast;\vartheta^1_\ast,\vartheta^2_\ast)\\
    v^\ast &\mapsto v=v^\ast\\
    \vartheta_\ast^A&\mapsto \vartheta^A=\vartheta_\ast^A
  \end{split}
\end{equation}
such that the new level sets $C_u$ of $u$ are again null hypersurfaces in de Sitter.
We have seen that for small $2u=\epsilon<0,\lvert\epsilon\rvert \ll 1$ the intersection of the null hypersurface $C_u$ with the cosmological horizon $\underline{\mathcal{C}}$ is a small deformation of the round sphere, while the intersection $S_{u,0}$ of $C_u$ with a fixed incoming null hypersurface $\Cb_0$ going to infinity is a large ellipsoidal deformation of a sphere near infinity, which contains points (first a point, and then annular regions surrounding this point) which ``run off to infinity'' as $\epsilon \to 0$, (while keeping $\varphi$ fixed).

We shall now calculate explicitly the transformations of all optical structure coefficients associated to the gauge transformation \eqref{eq:gauge:change}, at least for ``small displacement angles'', i.e.~for $\lvert\varphi\rvert\ll 1$; (the parameter $\varphi$ measures the displacement of the basepoint of the cones, see Fig.~\ref{fig:foliation:ellipsoid}, and corresponds to the eccentricity of the ellipsoids). We are interested in the details of the gauge transformation on the sphere near infinity, i.e.~for $2u=\epsilon, \lvert\epsilon\rvert\ll 1$.

We begin with the calculation of the null normals; in general, given two optical functions $u$, $v$ we define the corresponding null geodesic vectorfields by
\begin{equation}
  L^\prime=-2(\ud u)^\sharp\qquad \underline{L}^\prime=-2(\ud v)^\sharp\,.
\end{equation}
The null lapse is then defined by
\begin{equation}
  \Omega^2=-\frac{2}{g(\Lp,\Lbp)}
\end{equation}
and
\begin{equation}
  L=\Omega^2\Lp\qquad \Lb=\Omega^2\Lbp\,.
\end{equation}

\begin{lemma}\label{lemma:transformation:Omega}
  For small displacement angles $\lvert \varphi\rvert \ll 1$, the null vectorfields on $S_{u,0}$ for $2u=\epsilon<0$, $\lvert\epsilon\rvert\ll 1$ are given by
\begin{equation}
  L\simeq\frac{\partial}{\partial{v^\ast}}+\varphi\sin\vartheta^1\frac{\partial}{\partial \vartheta^1}\qquad \Lb\simeq \frac{\partial}{\partial{u^\ast}}
\end{equation}
up to terms quadratic in $(\varphi,\epsilon)$. Moreover $\Omega\simeq\Omega_\ast\simeq r$.
\end{lemma}

\begin{remark}
  The function $r$ that appears in the approximation for $\Omega$ is by no means constant on $S_{0,\epsilon}$. We will derive below an explicit dependence of $r(\vartheta^1)$ for small displacement angles, using the formulas obtained in Section~\ref{sec:ellipsoidal:explicit}.
\end{remark}

\begin{proof}
  Let us first calculate the derivatives of $u=u_\varphi^\ast$ on $\Cb_0$ where  $\vs=0$.
  In doing so we immediately employ the approximation $|\varphi|\ll 1$, and $|\us|\ll 1$:
\begin{align*}
  \partial_{u^\ast}u_\varphi^\ast\rvert_{v^\ast=0}=&-\frac{1}{4}\partial_{\us}\log u_\varphi\rvert_{\vs=0}\\
  \simeq&-\frac{1}{4}\bigl(\us+\cos\vartheta^1\varphi-1\bigr)^{-1}\bigl(1-\us\bigr)\\
  &+\frac{1}{4}\bigl(\us+\cos\vartheta^1\varphi+1\bigr)^{-1}\bigl(1+\us\bigr)\\
                                                   &-\frac{1}{4}\bigl(\us-1\bigr)^{-1}\bigl(1+\cos\vartheta^1\varphi-\us\bigr)\\
  &+\frac{1}{4}\bigl(\us+1\bigr)^{-1}\bigl(1   -\cos\vartheta^1\varphi+\us\bigr)\\
    \simeq& 1
\end{align*}
where we used that
\begin{gather*}
  \lvert r'_\varphi\rvert_{\vs=0}\simeq 1\\
    \partial_{\us}{r'_\varphi}\rvert_{\vs=0}\simeq   -\cos\vartheta^1\varphi+\us
  \end{gather*}
  Similarly,
  \begin{align*}
    \partial_{v^\ast}u_\varphi^\ast\rvert_{v^\ast=0} &\simeq 0\\
    \partial_{\vartheta^1}u_\varphi^\ast\rvert_{\vs=0} &\simeq -\frac{1}{2}\varphi\sin\vartheta^1
  \end{align*}
  where we used that
\begin{align*}
  \partial_{\vs}{r'_\varphi}\rvert_{\vs=0}\simeq &  \cos\vartheta^1\varphi+\us\\
      \partial_{\vartheta^1}{r'_\varphi}\rvert_{\vs=0}\simeq & 0
  \end{align*}
Since
\begin{gather*}
  g^{u^\ast u^\ast}=0\quad g^{u^\ast v^\ast}=-\frac{1}{2}\frac{1}{\Omega_\ast^2}\quad g^{v^\ast v^\ast}=0\\
  g^{u^\ast\vartheta^A_\ast}=0\quad g^{v^\ast\vartheta^A_\ast}=0\\
  g^{\vartheta^{1}_\ast\vartheta^1_\ast}=\frac{1}{r^2}\quad  g^{\vartheta^1_\ast\vartheta^2_\ast}=0 \quad  g^{\vartheta^{2}_\ast\vartheta^2_\ast}=\frac{1}{r^2}\frac{1}{\sin^2\vartheta^1_\ast}
\end{gather*}
where $\Omega_\ast$ refers to the null lapse \eqref{eq:Omega:ast},
it follows in particular that
\begin{equation}
  \Lp\simeq \frac{1}{\Omega_\ast^2} \partial_{v^\ast}+\frac{1}{r^2}\varphi\sin\vartheta^1\partial_{\vartheta^1}\qquad \Lbp = \frac{1}{\Omega^2_\ast}\partial_{u^\ast}
\end{equation}

Furthermore, in the above approximation
\begin{equation}
  \Omega^2=-\frac{2}{g(\Lp,\Lbp)}\simeq \Omega^2_\ast
\end{equation}
Therefore it immediately follows that
\begin{equation}
  L=\Omega^2\Lp\simeq\partial_{v^\ast}+\varphi\sin\vartheta^1\partial_{\vartheta^1}\qquad \Lb=\Omega^2\Lbp\simeq \partial_{u^\ast}\,.
\end{equation}

\end{proof}

It is now straighforward to calculate various connection coefficients:
\begin{subequations}\label{eq:null:structure:coeff:def}
  \begin{gather}
    \omega=L\log\Omega\qquad \omegab=\Lb\log\Omega\\
    \hat{\omega}=\frac{1}{\Omega}\omega\qquad\hat{\omegab}=\frac{1}{\Omega}\omegab\\
    \chi(X,Y)=g(\nabla_X (\Omega\Lp),Y)\qquad    \chib(X,Y)=g(\nabla_X (\Omega\Lbp),Y)
  \end{gather}
\end{subequations}

\begin{lemma}\label{lemma:transformation:chi}
 For small displacement angles $\lvert\varphi\rvert\ll 1$, 
the gauge transformation \eqref{eq:gauge:change} induces the following transformations of the null structure coefficients \eqref{eq:null:structure:coeff:def}, on $S_{u,0}$ for $2u=\epsilon<0$, with $\lvert\epsilon\rvert\ll 1$:
\begin{subequations}
\begin{gather}
      \hat{\omega}_\ast\mapsto\hat{\omega}\simeq\hat{\omega}_\ast\simeq 1\qquad     \hat{\omegab}_\ast\mapsto\hat{\omegab}\simeq\hat{\omegab}_\ast\simeq 1\\
  \tr\chi_\ast\mapsto\tr\chi\simeq 2+\frac{2\varphi}{r} \cos\vartheta^1\qquad \tr\chib_\ast\mapsto \tr\chib\simeq\tr\chib_\ast\simeq 2\\
  \hat{\chi}\simeq 0 \qquad \hat{\chib}\simeq 0
\end{gather}
\end{subequations}
\end{lemma}

\begin{proof}

Clearly, by Lemma~\ref{lemma:transformation:Omega}
\begin{gather*}
  \omega=D\log\Omega\simeq \omega_\ast=r\qquad \omegab=\Db\log\Omega\simeq \omegab_\ast=r\\
  \hat{\omegab}=\frac{1}{\Omega}\omegab\simeq\frac{1}{\Omega_\ast}\omegab_{\ast}\simeq 1
\end{gather*}

Next we calculate the components of the null second fundamental forms using the first variational formulas
\begin{gather*}
  D\gs=2\Omega\chi\qquad \Db\gs=2\Omega\chib\\
  2\Omega\chi(\partial_{\vartheta^A},\partial_{\vartheta^B})=L(g(\partial_{\vartheta^A},\partial_{\vartheta^B}))-g([L,\partial_{\vartheta^A}],\partial_{\vartheta^B})-g(\partial_{\vartheta^A},[L,\partial_{\vartheta^B}])
\end{gather*}
Here $\partial_{\vartheta^A}$ are the angular vectorfields in $(u,v;\vartheta^1,\vartheta^2)$ coordinates, and thus \emph{tangential} to the spheres $S_{u,v}=C_u\cap C_v$, as required. In view of \eqref{eq:us:theta:approx}, which describes the dependence of $u^\ast(\vartheta^1)$ on the sphere $S_{u,0}$, we note in particular that
\begin{equation}\label{eq:partial:theta}
  \frac{\partial}{\partial \vartheta^1}\Bigr\rvert_{u=\epsilon,v=0}=\frac{\partial}{\partial\vartheta^1_\ast}+\frac{\varphi}{2}\sin\vartheta^1\frac{\partial}{\partial u_\ast}
\end{equation}
We then find a non-vanishing commutator
\begin{equation*}
  [L,\partial_{\vartheta^1}]\simeq -\varphi\cos\vartheta^1\frac{\partial}{\partial \vartheta^1_\ast}
\end{equation*}
and obtain with
\begin{equation*}
  g_{\vartheta^1\vartheta^1}=r^2\quad g_{\vartheta^2\vartheta^2}=r^2\sin^2\vartheta^1
\end{equation*}
that
\begin{gather*}
    2\Omega\chi_{11}  \simeq 2r\Omega_\ast^2+2r^2 \varphi \cos\vartheta^1\\
    2\Omega\chi_{12} \simeq 0\\
    2\Omega\chi_{22} \simeq  2r\Omega_\ast^2\sin^2\vartheta^1+2\varphi r^2\sin^2\vartheta^1\cos\vartheta^1
\end{gather*}
Hence
\begin{gather*}
  \Omega\tr\chi\simeq \frac{2}{r}\Omega^2_\ast+2 \varphi  \cos\vartheta^1\\
  \tr\chi\simeq 2+\frac{2\varphi}{r}\cos\vartheta^1
\end{gather*}
and
\begin{equation*}
  \hat{\chi}_{AB}=\chi_{AB}-\frac{1}{2}\tr\chi \,g_{AB}\simeq 0\,.
\end{equation*}

\end{proof}

We emphasize that in the statements of Lemma~\ref{lemma:transformation:Omega}, \ref{lemma:transformation:chi}, the radius $r$ is a \emph{function} on $S_{\epsilon,\varphi}^\infty$.
In fact, we have seen in \eqref{eq:r:Cb:zero} that,
\begin{equation}
  r=-\frac{e^{2\us}+1}{e^{2\us}-1}\qquad \text{: on }\Cb(0)
\end{equation}
and moreover, for small displacement angles we have found in \eqref{eq:us:theta:approx}  the following relation between $u^\ast$ and $\vartheta^1$ on $S_{\epsilon,\varphi}^\infty$, $\lvert\epsilon\rvert\ll 1$,
\begin{equation}\label{eq:us:approx:S:infty}
  2\us(\vartheta^1)\simeq -|\epsilon|-\varphi\cos\vartheta^1
\end{equation}
Thus
\begin{equation}\label{eq:r:theta:S}
  r(\vartheta^1) \simeq \frac{2}{\lvert \epsilon\rvert+\varphi\cos\vartheta^1} \quad\text{: on }S_{\epsilon,\varphi}^\infty
\end{equation}
We summarize this formula for future reference in
\begin{corollary}
  On $S_{\epsilon,\varphi}^\infty$,
  \begin{equation*}
    \Omega(\vartheta^1) \simeq r(\vartheta^1) \simeq \frac{2}{\lvert \epsilon\rvert+\varphi\cos\vartheta^1}\,.
  \end{equation*}
\end{corollary}

\begin{remark}
  It is also clear from the formula that $S_{\epsilon,\varphi}^\infty$ is approximately an ellipsoid with eccentricity $|\varphi/\epsilon|$, which opens up to a paraboloid as $|\epsilon|\searrow |\varphi|$.
\end{remark}

Next we calculate the volume element on $S_{\epsilon,\varphi}^\infty$.

\begin{lemma}\label{lemma:area}
  The volume element on $S_{\epsilon,\varphi}^\infty$ is given by
  \begin{equation}
    \ud \mu_{\gs}=r^2(\vartheta^1)\sin\vartheta^1\ud\vartheta^1\wedge\ud\vartheta^2\\
  \end{equation}
provided that $\lvert\varphi\rvert\ll 1$, $\lvert \epsilon\rvert\ll 1$.
In particular,
\begin{equation}
  \text{Area}[S_{\epsilon,\varphi}^\infty]=\frac{16\pi}{\lvert\epsilon\rvert^2-\varphi^2}
\end{equation}
\end{lemma}

\begin{remark}
  Note that as expected
  \begin{equation*}
    \text{Area}[S_{\epsilon,\varphi}^\infty]\longrightarrow \infty\qquad \lvert\epsilon\rvert\searrow  \lvert\varphi\rvert
  \end{equation*}
  while keeping the displacement angle $\varphi$ fixed.
\end{remark}

\begin{proof}
  It follows from \eqref{eq:partial:theta} that
  \begin{equation*}
    \gs_{AB}=g(\partial_{\vartheta^A},\partial_{\vartheta^B})=g(\partial_{\vartheta^A_\ast},\partial_{\vartheta^B_\ast})
  \end{equation*}
hence
\begin{equation*}
  \sqrt{\det\gs}=r^2(\vartheta^1)\sin\vartheta^1
\end{equation*}
Therefore
\begin{equation*}
  \text{Area}[S_{\epsilon,\varphi}^\infty]=8\pi\int_0^\pi\frac{\sin\vartheta\ud\vartheta}{\bigl(\lvert\epsilon\rvert+\varphi\cos\vartheta\bigr)^2}=\frac{16\pi}{\lvert\epsilon\rvert^2-\varphi^2}
\end{equation*}

\end{proof}

Finally we calculate the average of the above transformed null structure coefficients:

\begin{lemma} \label{lemma:average:chi}
   We have, for small displacement angles $\lvert\varphi\rvert\ll 1$, the following averages on $S_{\epsilon,\varphi}^\infty$, provided $\lvert\epsilon\rvert\ll 1$:
    \begin{gather}
      \overline{\Omega}\simeq \frac{2\lvert\epsilon\rvert}{\lvert\epsilon\rvert^2-\varphi^2}\\
    \overline{\Omega\tr\chi}\simeq\frac{4\lvert\epsilon\rvert}{\lvert\epsilon\rvert^2-\varphi^2}
    -2\lvert\epsilon\rvert+\frac{1}{\varphi}\bigl(\lvert\epsilon\rvert^2-\varphi^2\bigr)\log\frac{\lvert\epsilon\rvert-\varphi}{\lvert\epsilon\rvert+\varphi}\\
    \overline{\Omega\tr\chib} \simeq \frac{4\lvert\epsilon\rvert}{\lvert\epsilon\rvert^2-\varphi^2}
  \end{gather}

\end{lemma}

\begin{remark}
    Note that in particular,
  \begin{equation*}
    \overline{\Omega\tr\chi}[S_{\epsilon,\varphi}^\infty]\longrightarrow \infty\qquad \lvert\epsilon\rvert\searrow \lvert\varphi\rvert
  \end{equation*}
  at the same rate as the area, namely
  \begin{equation*}
    \lim_{\lvert\epsilon\rvert\searrow \lvert\varphi\rvert}\frac{\overline{\Omega\tr\chi}}{\text{Area}[S_{\epsilon,\varphi}^\infty]}=\frac{\lvert\varphi\rvert}{4\pi}
  \end{equation*}

\end{remark}

\begin{proof}
We denote for brevity by $A_{\epsilon,\varphi}=\text{Area}[S_{\epsilon,\varphi}^\infty]$.
  Let us first calculate
  \begin{equation*}
    \overline{\Omega}=\frac{1}{A_{\epsilon,\varphi}}\int_{S_{\epsilon,\varphi}}\Omega\ud\mu_{\gs}\simeq\frac{2\pi}{A_{\epsilon,\varphi}}\int_0^\pi r^3(\vartheta)\sin\vartheta\ud\vartheta
    =\frac{2\lvert\epsilon\rvert}{\lvert\epsilon\rvert^2-\varphi^2}
  \end{equation*}
where we used the result of Lemma~\ref{lemma:area}.

Next we look at
\begin{equation*}
  \overline{\cos\vartheta^1}=\frac{1}{A_{\epsilon,\varphi}}\int_{S_{\epsilon,\varphi}^\infty}\cos\vartheta^1\ud\mu_{\gs}(\vartheta^1)
\end{equation*}
Integration by parts yields
\begin{equation*}
  \int_{S_{\epsilon,\varphi}^\infty}\cos\vartheta^1\ud\mu_{\gs}(\vartheta^1)=8\pi\int_0^\pi\frac{\cos\vartheta\sin\vartheta}{\Bigl(\lvert\epsilon\rvert+\varphi\cos\vartheta\Bigr)^2}\ud\vartheta
  =-\frac{8\pi}{\varphi}\frac{2\lvert\epsilon\rvert}{\lvert\epsilon\rvert^2-\varphi^2}+\frac{8\pi}{\varphi^2}\log\frac{\lvert\epsilon\rvert-\varphi}{\lvert\epsilon\rvert+\varphi}
\end{equation*}
and thus
\begin{equation*}
    \overline{\cos\vartheta^1}=-\frac{|\epsilon|}{\varphi}+\frac{1}{2\varphi^2}\bigl(|\epsilon|^2-\varphi^2\bigr)\log\frac{\lvert\epsilon\rvert-\varphi}{\lvert\epsilon\rvert+\varphi}
\end{equation*}
Note in particular that, 
\begin{equation*}
   \varphi\,\overline{\cos\vartheta^1}\longrightarrow \epsilon \qquad\lvert\epsilon\rvert\searrow\lvert\varphi\rvert
\end{equation*}
The stated formulas then follow from Lemma~\ref{lemma:transformation:chi}, according to which
\begin{gather*}
  \overline{\Omega\tr\chi}\simeq \overline{2\Omega}+\overline{\frac{4\varphi}{r}\Omega \cos\vartheta^1}\simeq 2\overline{\Omega}+4\varphi\overline{\cos\vartheta^1}\\
   \overline{\Omega\tr\chib}\simeq \overline{2\Omega}\,.
\end{gather*}

\end{proof}

Lastly we discuss the conformal geometry of the spheres $S_{\epsilon,\varphi}$. The relevance of the behavior of the conformal factor will be explained in Section~\ref{sec:uniformization} in the context of a more general construction of solutions to the eikonal equation.

Here we introduce a function $\Phi:S_{\epsilon,\varphi}\to \mathbb{R}$ such that
\begin{equation}
  K[e^{2\Phi}\gs]=1
\end{equation}
If we take $\Phi=\Psi-\log r_{\epsilon,\varphi}$, where $4\pi r_{\epsilon,\varphi}^2=\Area[S_{\epsilon,\varphi}]$, then we can easily infer a formula for $\Phi$ from the fact that here $\gs=r^2(\vartheta^1)\gammac$, where $r(\vartheta^1)$ is known. In fact,
\begin{equation}
    e^{2\Psi}=\frac{\Area[S_{\epsilon,\varphi}]}{4\pi r^2(\vartheta^1)}\simeq \frac{(|\epsilon|+\varphi\cos\vartheta^1)^2}{|\epsilon|^2-\varphi^2}
\end{equation}
and we see that
\begin{equation}
    e^{2\Psi}(\vartheta^1=0)=\frac{|\epsilon|+\varphi}{|\epsilon|-\varphi}\qquad   e^{2\Psi}(\vartheta^1=\pi)=\frac{|\epsilon|-\varphi}{|\epsilon|+\varphi}
  \end{equation}
  Thus $\Psi$ is \emph{unbounded} from above and below, because here
\begin{equation}
  \Psi(\vartheta^1=0)\to \infty \qquad \Psi(\vartheta^1=\pi)\to -\infty \qquad (|\epsilon|\searrow \varphi)
\end{equation}

Therefore one can also not expect to have good bounds on the comparsion of $r\overline{\Omega\tr\chi}$ and $\Omega^2$, or good control on $\Omega-\overline{\Omega}$ in terms of $r$. Indeed, for these foliations
\begin{equation}
  \Omega^{-2}\frac{r}{2}\overline{\Omega\tr\chi} \simeq \frac{|\epsilon|(|\epsilon|+\varphi\cos\vartheta^1)}{|\epsilon|^2-\varphi^2}=\begin{cases}
  \frac{|\epsilon|}{|\epsilon|-\varphi}\to \infty & (\vartheta^1=0) \\   \frac{|\epsilon|}{|\epsilon|+\varphi}\to \frac{1}{2}&  (\vartheta^1=\pi)\end{cases}
\end{equation}
and
\begin{gather}
  \frac{\Omega-\overline{\Omega}}{r}\simeq-\varphi\frac{\varphi+|\epsilon|\cos\vartheta^1}{|\epsilon|^2-\varphi^2}\simeq\begin{cases}
  -\frac{\varphi}{|\epsilon|-\varphi}\to\infty & (\vartheta^1=0) \\ -\frac{\varphi}{|\epsilon|+\varphi}\to -\frac{1}{2} &(\vartheta^1=\pi) \end{cases}
\end{gather}

\section{The eikonal equation and the optical structure equations}

In Section~\ref{sec:eikonal:general} we summarise a few facts relevant for the proof of the main Proposition, used in Section~\ref{sec:desitter:double} for the explicit construction of global solutions to the eikonal equation on de Sitter, or \emph{initial} data gauge. In Section~\ref{sec:null:structure} we begin the discussion proper of the null structure equations in this setting relevant for the proof of the main Theorem, namely the construction of families of \emph{final} data gauges.

\subsection{The eikonal equation}
\label{sec:eikonal:general}

The detailed study of the geometry of  null hypersurfaces $C$ in de Sitter $(H,h)$ is central to this paper.
We restrict ourselves to the case where $C$ is a non-critical level set of a differentiable function $u:H\to \mathbb{R}$. Then $u$ is the solution to the eikonal equation
\begin{equation}\label{eq:eikonal}
  h^{\mu\nu} \partial_\mu u\partial_\nu u =0
  \end{equation}
  and will then be referred to as an \emph{optical function} in de Sitter.
  We recall that a general solution to \eqref{eq:eikonal} can be constructed by the method of characteristics from a congruence of null geodesic segments. In that regard the following elementary observation will be useful:
  \begin{lemma}
    The de Sitter hyperboloid $H$ is a ``totally null geodesic'' submanifold of $(\mathbb{R}^{3+1},m)$, i.e.~every \emph{null} geodesic in $H$ is straight line in $\mathbb{R}^{4+1}$.
  \end{lemma}

  \begin{proof}
    Let $p\in H$, where $H$ is expressed in ambient coordinates $(t,x)$ by \eqref{eq:H}.
    By a rotation of the axes we can first achieve that $p$ has the coordinates
    \begin{equation*}
      x:=x^1\geq 1\qquad {x'}^i:=x^{i+1}=0\qquad i=1,2,3
    \end{equation*}
    If $x\neq 1$, we can apply a Lorentz boost with velocity $v=t/x$, namely a coordinate change $(t,x)\to(\tilde{t},\tilde{x})$,
    \begin{equation*}
      \tilde{t}=\frac{t- v x}{\sqrt{1-v^2}}\qquad \tilde{x}=\frac{x-vt}{\sqrt{1-v^2}}
    \end{equation*}
    which achieves that $p$ is at $(\tilde{t}=0,\tilde{x}=1)$. Of course, also in the new coordinates $H$ is given by \eqref{eq:H}
    which shows that w.l.o.g.~we can assume that $p$ has the coordinates $(t=0,x=1,x'=0)$.
    
    The tangent space to $H$ at $p$, is then clearly spanned by
    \begin{equation*}
      \mathrm{T}_p H=\langle \frac{\partial}{\partial t},\frac{\partial}{\partial {x'}^i}:i=1,2,3\rangle
    \end{equation*}
    Therefore any line in $\mathrm{T}_pH\cap C_p$, where $C_p$ is the light cone in $T_p\mathbb{R}^{3+1}$, can be identified with a vector
    \begin{equation*}
      L=\frac{\partial}{\partial t}+N\qquad N=N^i\frac{\partial}{\partial {x'}^i}\qquad \sum_{i=1}^3 (N^i)^2=1
    \end{equation*}
    The straight line emanating from $p$ with initial tangent vector $L$ is given by
    \begin{equation*}
      \gamma(t)=\Bigl(t,x=1,x'=Nt\Bigr)
    \end{equation*}
    and obviously $\gamma(t)\in H$ for all $t\in\mathbb{R}$. Moreover, by direct computation
    \begin{equation*}
      \dot{\gamma}(t)=\frac{\partial}{\partial t}+N^i\frac{\partial}{\partial {x'}^i}
    \end{equation*}
    and thus
    \begin{equation*}
      \nabla_{\dot{\gamma}}\dot{\gamma}=0
    \end{equation*}
    where $\nabla$ refers to the trivial connection of $\mathbb{R}^{4+1}$ in cartesian coordinates. By uniqueness of the connection $\stackrel{(h)}{\nabla}$ on $H$, and $h=m\rvert_H$, we have $\stackrel{(h)}{\nabla}=\Pi \nabla$, where $\Pi$ is the projection to the tangent space of $H$ at any given point, and thus also 
    \begin{equation*}
      \stackrel{(h)}{\nabla}_{\dot{\gamma}}\dot{\gamma}=\Pi \nabla_{\dot{\gamma}}\dot{\gamma}=0
    \end{equation*}
    which says that $\gamma(t)$ is also a geodesic in $H$, which of course is a null geodesic: $h(\dot{\gamma},\dot{\gamma})=m(\dot{\gamma},\dot{\gamma})=0$.

  \end{proof}

  \subsubsection{Null geodesic congruences}

  The solutions to \eqref{eq:eikonal} can now be constructed, at least locally, as follows:
  Let $S\subset H$ be a closed spacelike surface of codimension $2$ in $H$.
  For simplicity, let us take $S\subset \mathbb{S}^3$ to be a surface diffeomorphic to $\mathbb{S}^2$ in $H\cap\{t=0\}$. Let $N$ be the unit normal vectorfield to $S$ in $\mathbb{S}^3$, i.e.~for $p\in S$, and $\mathrm{T}_pS$ viewed as a subspace of $\mathrm{T}_p \mathbb{S}^3$, we have $\mathrm{T}_pS^\perp=\langle N\rangle$. $\mathbb{S}^3\setminus S$ has two components and we choose $N$ to have an orientation in the sense that it points to the same component everywhere on $S$. Let
  \begin{equation*}
    L=\frac{\partial}{\partial t}+N
  \end{equation*}
  and let $\Gamma_p^\epsilon$ be the null geodesic segment emanating from a point $p=(t=0,x_p)\in S$ with initial tangent vector $L$.
    We have seen above that $L_p\in\mathrm{T}_pH$  and $\Gamma_p$ is a straight line
  \begin{equation*}
    \Gamma_p^\epsilon := \Bigl\{ (t,x_p+N_pt) : 0\leq t\leq\epsilon \Bigr\}
  \end{equation*}
  in the ambient $\mathbb{R}^{3+1}$.
  Then define
  \begin{equation*}
    C=\bigcup_{p\in S} \Gamma_p^\epsilon
  \end{equation*}
  and set $u=0$ on $C$. In fact, we could choose a smooth family of surfaces $S_\delta\subset \mathbb{S}^3$, with $S_0=S$, which foliate a neighborhood of $S\subset\mathbb{S}^3$. Then for each $\delta$ we can repeat the construction above to obtain the surfaces $C_\delta=\bigcup_{p\in S_\delta}\Gamma_p^\epsilon$, and set $u=\delta$ on $C_\delta$. Then, by definition,
  \begin{equation*}
    T_pC = \Bigl\{ X\in T_p H : \ud u\cdot X=X u=0 \Bigr\}
  \end{equation*}
  which shows that $\mathrm{T}_p C$ is the orthogonal complement of 
  \begin{equation*}
    L'=-\ud u^\sharp\qquad {L'}^\mu = - h^{\mu\nu}\partial_\nu u
  \end{equation*}
  in $\mathrm{T}_p H$. It remains to show that $L'\in T_pC\subset\mathrm{T}_pH$, which then implies
  \begin{equation*}
    0=h(L',L')=h^{\mu\nu}\partial_\mu u\partial_\nu u\,,
  \end{equation*}
  namely $u$ is a solution to \eqref{eq:eikonal}, as desired. The former is so because $C_\delta$ are \emph{null} hypersurfaces, i.e. $T_pC_\delta$ are null hyperplanes; indeed by construction the tangent space $T_pC$ is spanned by $L_p$ and $T_pS$, and here $L'$ is colinear to $L_p$ at each point.

  We see that as long as the null rays with tangents $L_p=\partial_t +N_p$ over $p\in S$ do not intersect, the function $u$ remains smooth; this can always be arranged for $\epsilon >0$ sufficiently small. For suitable choices of surfaces $S_\delta$, as considered below, the normal null rays may in fact never intersect, in which case this construction yields global solutions.


  \subsection{Null structure equations}
  \label{sec:null:structure}
    

The structure equations of a double null foliation are presented here in close analogy to  \cite{C:09}. Definitions are not repeated from Chapter~1 in \cite{C:09}, but attention is given to differences that occur in the $\Lambda>0$ setting. The equations will be reduced to the case of vanishing Weyl curvature.

\subsubsection{Metric}

  \begin{equation}
    D\gs=2\Omega\chi\qquad \Db\gs=2\Omega\chib
  \end{equation}
  \begin{equation}\label{eq:D:Omega}
    D \log \Omega = \omega\qquad \Db\log\Omega = \omegab
\end{equation}
\begin{equation}\label{eq:D:r}
  D r=\frac{r}{2}\overline{\Omega\tr\chi}\qquad    \Db r=\frac{r}{2}\overline{\Omega\tr\chib}
\end{equation}

\subsubsection{Null second fundamental form}

The null second fundamental form satisfies the propagation equations:
\begin{subequations}\label{eq:D:chi}
  \begin{gather}
  D\chi=\omega\chi+\Omega\bigl(\chi\times\chi-\alpha\bigr)\\
  \hat{D}\chih=\omega\chih-\Omega\alpha\label{eq:Dh:chih}\\
  D\tr\chi=\omega\tr\chi-\Omega(\chi,\chi)\label{eq:D:tr:chi}
\end{gather}
\end{subequations}
Here $\alpha[W]=0$.

Moreover  the Codazzi equations read:
\begin{align}
  \divs\chi-\ds\tr\chi+\chi^\sharp\cdot\zeta-\tr\chi\,\zeta&=-\beta\label{eq:codazzi:chi}\\
  \divs\chib-\ds\tr\chib-\chib^\sharp\cdot\zeta+\tr\chib\zeta&=\betab\label{eq:codazzi:chib}
\end{align}
\begin{align}
  \divs(\Omega\chih)-\frac{1}{2}\ds(\Omega\tr\chi)-\Omega\chih^\sharp\cdot\etab+\frac{1}{2}\Omega\tr\chi\,\etab&=-\Omega\beta\label{eq:codazzi:chih}\\
  \divs(\Omega\chibh)-\frac{1}{2}\ds(\Omega\tr\chib)-\Omega\chibh^\sharp\cdot\eta+\frac{1}{2}\Omega\tr\chib\eta&=\Omega\betab\label{eq:codazzi:chibh}
\end{align}
Here $\beta[W]=0$, $\betab[W]=0$.

\begin{remark}
  In the shear-free case $\chibh=0$ the conjugate Codazzi equation reduces to:
  \begin{equation}
    \label{eq:codazzi:chib:shear-free}
    \ds(\Omega\tr\chib)=\Omega\tr\chib\eta
  \end{equation}

\end{remark}

Furthemore
\begin{gather}
  \chibp=\Omega^{-1}\chib\\
  \Db\chibp=\Omega^2\chibp\times\chibp-\alphab\label{eq:Db:chibp}\\
  \Db\tr\chibp=-\frac{1}{2}\Omega^2(\tr\chibp)^2-\Omega^2|\chibhp|^2\\
  \Dbh\chibhp=0\label{eq:Db:chibh}
\end{gather}
and here $\alphab[W]=0$.

We also note
\begin{equation}
  D (\Omega\chib) = \Omega^2 \Bigl\{\nablas\etab+\tilde{\nablas}\etab+2\etab\otimes\etab+\frac{1}{2}\bigl(\chib\times\chi+\chi\times\chib\bigr)+\rho[W]\gs+2\frac{\Lambda}{3}\gs\Bigr\}\label{eq:D:chib}
\end{equation}

\subsubsection{Null expansions}

We have
\begin{align}
  D\omegab &= \Omega^2\Bigl(2(\eta,\etab)-\vert\eta\rvert^2-\rho[W]+\frac{\Lambda}{3}\Bigr) \label{eq:D:omegab}\\
  D(\Omega\tr\chib) &= \Omega^2\Bigl(-(\chib,\chi)+2\divs\etab+2|\etab|^2+2\rho[W]+4\frac{\Lambda}{3}\Bigr) \label{eq:D:Omega:tr:chib}
\end{align}

Here $\rho[W]=0$, and we set
\begin{equation}
  \frac{\Lambda}{3}=1
\end{equation}

The Gauss equation reads:
\begin{equation}\label{eq:gauss:hat}
    K+\frac{1}{4}\tr\chi\tr\chib-1=-\rho[W]+\frac{1}{2}(\chih,\chibh)
\end{equation}

Therefore
\begin{equation}
  D\bigl(2\omegab-\Omega\tr\chib\bigr)= -2\Omega^2\Bigl(-(\chibh,\chih)+K+\divs\etab-2(\eta,\etab)+\vert\eta\rvert^2+\vert\etab\rvert^2+3\rho[W]\Bigr)
\end{equation}

The conjugate equations are:
\begin{equation}
    \Db\omega= \Omega^2\Bigl(2(\eta,\etab)-\lvert\etab\rvert^2-\rho[W]+\frac{\Lambda}{3}\Bigr)\label{eq:Db:omega}
\end{equation}
\begin{equation}\label{eq:Db:ds:omega}
  \Db\ds\omega=2\Omega^2\ds\log\Omega\Bigl(2(\eta,\etab)-\lvert\etab\rvert^2+\frac{\Lambda}{3}\Bigr)+\Omega^2\Bigl(2(\nablas\eta,\etab)+2(\eta,\nablas\etab)-2(\etab,\nablas\etab)\Bigr)
\end{equation}

\subsubsection{Torsion}

The propagation equations for the torsion are:
\begin{subequations}\label{eq:propagation:torsion}
\begin{align}
  D\eta&=\Omega\bigl(\chi^\sharp\cdot\etab-\beta\bigr)\label{eq:D:eta}\\
  \Db\etab&=\Omega\bigl(\chib^\sharp\cdot \eta+\betab\bigr)\\
  D\etab&=-\Omega\bigl(\chi^\sharp\cdot\etab-\beta\bigr)+2\ds\omega\label{eq:D:etab}\\
  \Db\eta&=-\Omega\bigl(\chib^\sharp\cdot\eta+\betab\bigr)+2\ds\omegab\label{eq:Db:eta}
\end{align}
\end{subequations}
Here $\beta[W]=0$, and $\betab[W]=0$.


Moreover
 \begin{gather}
   \eta=\zeta+\ds\log\Omega\qquad
   \etab=-\zeta+\ds\log \Omega\\
    \ds\log\Omega=\frac{1}{2}\bigl(\eta+\etab\bigr)\label{eq:ds:log:Omega:eta}
\end{gather}
and
 \begin{gather}
   D\ds\log \Omega=\ds\omega\\
   D\zeta=-\ds\omega-\Omega\chi^\sharp\cdot\bigl(\zeta-\ds\log\Omega\bigr)-\Omega \beta
 \end{gather}


\subsection{Gauss curvature and mass aspect function}

As explained in Section~\ref{sec:results} --- c.f.~discussion following the informal statement of the Theorem --- control on the conformal factor relies on sufficient bounds for the Gauss curvature. While the Gauss curvature $K(u,v)$ of $S_{u,v}$ satisfies a propagation equation along the generators of $C_u$, and $\Cb_v$, the equation loses derivatives, which can be avoided for a specific renormalised quantity which we discuss in Section~\ref{sec:propagation:gauss}; this cancellation in the propagation equation was first observed in \cite{CK:93}. This further leads to the mass aspect functions, whose propagation equations are discussed in Section~\ref{sec:propagation:mu}, \ref{sec:prelim:decoupling}.

\subsubsection{Propagation equation for Gauss curvature}
\label{sec:propagation:gauss}

Let us define
\begin{equation}\label{eq:kappa}
  \kappa:=K-\divs\eta\,.
\end{equation}

As derived in \Ceq{5.28} the propagation equation for the Gauss curvature reads
\begin{equation}
  \begin{split}
    DK+\Omega\tr\chi K&=\divs\divs(\Omega\chi)-\Laplaces(\Omega\tr\chi)\\
    &=\divs\divs(\Omega\chih)-\frac{1}{2}\Laplaces(\Omega\tr\chi)
  \end{split}
\end{equation}

Recall now that \eqref{eq:D:eta} does not contain a $\ds\omega$ term.
  We can now also rewrite the propagation  equation for $\eta$ using the Codazzi equation \eqref{eq:codazzi:chi},
  \begin{equation}
    D\eta=\divs(\Omega\chi^\sharp)-\ds(\Omega\tr\chi)+\Omega\tr\chi\etab
  \end{equation}
  and thus, using \Ceq{6.107}, we arrive at \Ceq{6.108}:
  \begin{equation}
    \begin{split}
      -D\divs\eta=&-\divs\divs\bigl(\Omega\chi^\sharp\bigr)+\Laplaces\bigl(\Omega\tr\chi\bigr)\\
      &+\divs\bigl(2\Omega\chih\cdot \eta-\Omega\tr\chi\etab\bigr)+\Omega\tr\chi\divs\eta
    \end{split}
  \end{equation}  
  and we conclude that the propagation equation for $K-\divs\eta$ does not lose derivatives:
  \begin{lemma}\label{lemma:D:kappa}
    \begin{subequations}
    \begin{gather}
      D\kappa+\Omega\tr\chi\kappa=\divs j\qquad j=\Omega\bigl(2\chih\cdot \eta-\tr\chi\etab\bigr)\\
      D\bigl(\Omega^2\kappa\bigr)=\bigl(2\omega-\Omega\tr\chi\bigr)\Omega^2\kappa+\Omega^2\divs j
    \end{gather}
  \end{subequations}
\end{lemma}
\begin{proof}
  Adding the propagation equations for $K$ and $\divs\eta$ derived above, we obtain
  \begin{equation}
  \begin{split}
    D\bigl(K-\divs\eta\bigr)=&\divs\bigl(2\Omega\chih\cdot \eta-\Omega\tr\chi\etab\bigr)\\
    &-\Omega\tr\chi K+\Omega\tr\chi\divs\eta
  \end{split}
\end{equation}
which is already the equation for $\kappa=K-\divs\eta$.
  
\end{proof}

Let us also derive the conjugate equation, namley the propagation equation for
\begin{equation}\label{eq:kappab}
  \kappab:=K-\divs\etab
\end{equation}

First we rewrite the propagation equation for $\etab$ using the Codazzi equations:
\begin{equation}
  \Db\etab=\divs\Omega\chib-\ds(\Omega\tr\chib)+\Omega\tr\chib\eta
\end{equation}
This step is essential and uses the Einstein equations and the fact that the Weyl curvature vanishes.
Then it follows  from the propagation equation for the Gauss curvature that
\begin{equation}\label{eq:Db:kappab}
  \Db\kappab+\Omega\tr\chib\kappab=\divs\bigl(2\Omega\chibh^\sharp\cdot\etab-\Omega\tr\chib\eta\bigr)
\end{equation}

\begin{remark}
  Note that in the shear-free case $\chibh=0$ we can use the Codazzi equation \eqref{eq:codazzi:chib:shear-free} again and the propagation equations for $\etab$, and $\kappab$ reduces to
  \begin{align}
    \Db\etab&=-\frac{3}{2}\Omega\tr\chib\eta\\
    \Db|\etab|^2+\Omega\tr\chib|\etab|^2&=-3\Omega\tr\chib(\eta,\etab)\\
    \Db\kappab+\Omega\tr\chib\kappab&=-\Omega\tr\chib\bigl(|\eta|^2+\divs\eta\bigr)    \label{eq:Db:kappab:shear-free}
  \end{align}

\end{remark}

\subsubsection{Propagation equation for mass aspect function}
\label{sec:propagation:mu}

Recall the mass aspect function \eqref{eq:mu}.
Also set
\begin{equation}\label{eq:mub}
  \mub := -\rho[W]+\frac{1}{2}(\chih,\chibh)-\divs\etab
\end{equation}

\begin{lemma}\label{lemma:D:mu}
  \begin{equation}
    \begin{split}
      D(\Omega^2\mu)=&\bigl(2\omega-\Omega\tr\chi\bigr)\Omega^2\mu+\Omega^2\divs\bigl(2\chih\cdot \Omega\eta-\tr\chi\Omega\etab\bigr)\\
      &-\frac{1}{2}\Omega\tr\chi\bigl(\Omega^2\mub-|\Omega\etab|^2\bigr)-\frac{1}{4}\Omega\tr\chib|\Omega\chih|^2
  \end{split}
\end{equation}
  \begin{equation}
    \begin{split}
      \Db(\Omega^2\mub)=&\bigl(2\omegab-\Omega\tr\chib\bigr)\Omega^2\mub+\Omega^2\divs\bigl(2\chibh\cdot \Omega\etab-\tr\chib\Omega\eta\bigr)\\
      &-\frac{1}{2}\Omega\tr\chib\bigl(\Omega^2\mu-|\Omega\eta|^2\bigr)-\frac{1}{4}\Omega\tr\chi|\Omega\chibh|^2
  \end{split}
\end{equation}

\end{lemma}
\begin{proof}
  Write $\tr\chi\tr\chib=\tr\chip\Omega\tr\chib$ and use \eqref{eq:D:Omega:tr:chib}:
  \begin{equation}\label{eq:D:trchitrchib}
    D\bigl(\tr\chi\tr\chib)=-2\Omega\tr\chi\bigl(\mub-|\etab|^2\bigr)-4\Omega\tr\chi\bigl(\frac{1}{4}\tr\chi\tr\chib-1\bigr)-\Omega\tr\chib|\chih|^2
  \end{equation}
  using the definition of $\underline{\mu}$. Hence by Lemma~\ref{lemma:D:kappa}
  \begin{equation*}
    \begin{split}
      D\mu&=D\kappa+\frac{1}{4}D(\tr\chi\tr\chib)\\
      &=-\Omega\tr\chi \mu+\divs j-\frac{1}{2}\Omega\tr\chi(\mub-|\etab|^2)-\frac{1}{4}\Omega\tr\chib|\chih|^2
    \end{split}
  \end{equation*}
  and the statement of the Lemma follows from $D\log\Omega=\omega$.
  Similarly for the conjugate equation.
\end{proof}

The propagation equation for $\mu$ as written above couples to the conjugate mass aspect function $\mub$. We can rewrite this equations as follows:
    \begin{equation}\label{eq:D:mu:all}
      \begin{split}
        D\mu=&-\Omega\tr\chi \mu+\divs j-\frac{1}{2}\Omega\tr\chi(\mub-|\etab|^2)-\frac{1}{4}\Omega\tr\chib|\chih|^2\\
        =&-\frac{3}{2}\Omega\tr\chi \mu-\frac{1}{2}\Omega\tr\chi\bigl(\divs\eta+\divs\etab+|\etab|^2)-\frac{1}{4}\Omega\tr\chib|\chih|^2\\
        &+2\divs(\Omega\chih)\cdot (\eta-\etab)+2(\Omega\chih,\nablas\eta)+2\Omega\chih(\etab,\etab)
      \end{split}
    \end{equation}
    where
    \begin{equation}
    \mub + \divs\etab = \mu + \divs\eta 
\end{equation}
and we used that by the Codazzi equation \eqref{eq:codazzi:chi},
\begin{equation}\label{eq:codazzi:chi:j}
    2\divs(\Omega\chih)-\ds(\Omega\tr\chi)-2\Omega \chih^\sharp\cdot\etab+\Omega \tr\chi\,\etab=0
  \end{equation}
  to rewrite
\begin{equation}\label{eq:div:j}
  \begin{split}
    \divs j=&\divs\bigl(2\Omega\chih\cdot \eta-\Omega\tr\chi\etab\bigr)\\
    =&2\divs(\Omega\chih)\cdot \eta+2(\Omega\chih,\nablas\eta)-\ds(\Omega\tr\chi)\cdot \etab-\Omega\tr\chi\divs\etab\\
    =&2\divs(\Omega\chih)\cdot (\eta-\etab)+2(\Omega\chih,\nablas\eta+\etab\otimes\etab)-\Omega\tr\chi\bigl(\divs\etab+|\etab|^2)\,.
  \end{split}
\end{equation}


In the preliminary result Lemma~\ref{lemma:mu} we have seen that a bound on $r^3\mu$ can be obtained at the cost of introducing an assumption on  $\Laplaces\log\Omega$, which arises here from the $\divs\eta+\divs\etab$ term. Another strategy is to introduce the  ``modified mass aspect function'' $\breve{\mu}$, analogous to \cite{KN:03}, to cancel this term altogether. We discuss this in the next Section~\ref{sec:prelim:decoupling}.

\subsubsection{Decoupling of the mass aspect functions}
\label{sec:prelim:decoupling}

Recall Lemma~\ref{lemma:D:kappa}, according to which we have
\begin{equation}
      D\kappa+\Omega\tr\chi\kappa=2\divs(\Omega\chih)\cdot (\eta-\etab)+2(\Omega\chih,\nablas\eta+\etab\otimes\etab)-\Omega\tr\chi\bigl(\divs\etab+|\etab|^2)
    \end{equation}
    where we used the result of the calculation \eqref{eq:div:j}.
    Recalling the formula \eqref{eq:D:trchitrchib}, we see that to cancel the $\divs\etab$ term on the right hand side we should consider the quantity
    \begin{equation}\label{eq:mu:breve}
      \begin{split}
        \breve{\mu}:=&\kappa+\frac{1}{2}\tr\chi\tr\chib-2=K-\divs\eta+\frac{1}{2}\tr\chi\tr\chib-2\\
        =&-\rho[W]+\frac{1}{2}(\chih,\chibh)-\divs\eta+\frac{1}{4}\tr\chi\tr\chib-1
    \end{split}
  \end{equation}
which statisfies
\begin{align}
  D\breve{\mu}+\Omega\tr\chi\breve{\mu}=&-\Omega\tr\chi\bigl(-\rho[W]+\frac{1}{2}(\chih,\chibh)\bigr)-\frac{1}{2}\Omega\tr\chib|\chih|^2+\breve{m}\\
  \breve{m}=&2\divs(\Omega\chih)\cdot (\eta-\etab)+2(\Omega\chih,\nablas\eta+\etab\otimes\etab)
\end{align}

In particular in the shear-free case on de Sitter:
\begin{lemma}\label{lemma:D:mu:breve}
\begin{subequations}
  \begin{align}
  D\bigl(\Omega^2\breve{\mu}\bigr)=&\bigl(2\omega-\Omega\tr\chi\bigr)\Omega^2\breve{\mu}-\frac{1}{2}\Omega\tr\chib\Omega^2|\chih|^2+\Omega^2\breve{m}\\
  \breve{m}=&2\divs(\Omega\chih)\cdot (\eta-\etab)+2(\Omega\chih,\nablas\eta+\etab\otimes\etab) 
\end{align}
\end{subequations}
\end{lemma}

  This equation is suitable to integrate from $\Cb_+$, however it cannot be expected that $\Omega^3\breve{\mu}$ is finite; indeed even with respect to the spherically symmetric foliation of de Sitter we have
  \begin{equation}
    \breve{\mu}=-\frac{1}{r^2}\,.
  \end{equation}
  The defining property could be taken to be $\breve{\mu}=\overline{\breve{\mu}}$ which in the shear-free case on de Sitter becomes 
\begin{equation}
  \divs\eta=\frac{1}{4}\bigl(\tr\chi\tr\chib-\overline{\tr\chi\tr\chib}\bigr)\,.
\end{equation}

Let us then consider the propagation equation for $\breve{\mu}-\overline{\breve{\mu}}$:
  \begin{gather}
    D\bigl(\breve{\mu}-\overline{\breve{\mu}}\bigr)+\Omega\tr\chi\bigl(\breve{\mu}-\overline{\breve{\mu}}\bigr)=\breve{r}-\overline{\breve{r}}+\breve{m}-\overline{\breve{m}}\label{eq:D:tilde:mu}\\
    \breve{r}:=-\Omega\tr\chi\bigl(-\rho[W]+\frac{1}{2}(\chih,\chibh)\bigr)-\frac{1}{2}\Omega\tr\chib|\chih|^2
  \end{gather}
For brevity, we define
\begin{equation}\label{eq:tilde:mu}
  \tilde{\mu}:=\breve{\mu}-\overline{\breve{\mu}}\,.
\end{equation}





In analogy to $\breve{\mu}$ we can introduce the \emph{conjugate} function:
\begin{equation}\label{eq:mub:breve}
\underline{\breve{\mu}}         :=-\rho[W]+\frac{1}{2}(\chih,\chibh)-\divs\etab+\frac{1}{4}\tr\chi\tr\chib-1
\end{equation}


Also note
\begin{equation}
  \begin{split}
    \Db\bigl(\tr\chi\tr\chib\bigr)=&-2\Omega\tr\chib\Bigl(\breve{\mu}-|\eta|^2+\frac{1}{4}\tr\chib\tr\chi-1\Bigr)-\Omega\tr\chi|\chibh|^2\\
    =&-2\Omega\tr\chib\Bigl(\mu-|\eta|^2+\frac{1}{2}\tr\chib\tr\chi-2\Bigr)-\Omega\tr\chi|\chibh|^2
  \end{split}
\end{equation}
which we will use in the form
\begin{equation}\label{eq:Db:trchitrchib:general}
  \Db\bigl(\frac{1}{4}\tr\chi\tr\chib-1\bigr)+\Omega\tr\chib\Bigl(\frac{1}{4}\tr\chib\tr\chi-1\Bigr)=-\frac{1}{2}\Omega\tr\chib\Bigl(\mu-|\eta|^2\Bigr)-\frac{1}{4}\Omega\tr\chi|\chibh|^2
\end{equation}

Also,
\begin{equation}\label{eq:breve:mu:mub}
  \breve{\mu}+\breve{\mub}=-2\rho[W]+(\chih,\chibh)-\divs\eta-\divs\etab+\frac{1}{2}\tr\chi\tr\chib-2
\end{equation}




\subsubsection{Angular derivatives of the null expansions}





Let us also derive the propagation equation for $\Laplaces\omegab$.
From
\begin{equation}
  D\omegab=\Omega^2\Bigl(2(\eta,\etab)-|\eta|^2+1\Bigr)
\end{equation}
we derive first
\begin{equation}
  D\ds\omegab=2\Omega^2\bigl(\ds\log\Omega\bigr)\bigl(2(\eta,\etab)-|\eta|^2+1\bigr)+\Omega^2\bigl(2(\nablas\eta,\etab-\eta)+2(\eta,\nablas\etab)\bigr)
\end{equation}
Thus by \Ceq{6.107} applied to $\xi=\ds\omegab$,
\begin{equation}
  \begin{split}
    D\Laplaces\omegab =& D\divs\ds\omegab \\
    =& \divs D\ds\omegab -2\divs\bigl(\Omega\chih^\sharp\cdot \ds\omegab\bigr) -\Omega\tr\chi\divs\ds\omegab
  \end{split}
\end{equation}
hence
\begin{multline}\label{eq:D:Laplace:omegab}
  D\Laplaces\omegab +\Omega\tr\chi\Laplaces\omegab= -2\divs\bigl(\Omega\chih^\sharp\cdot \ds\omegab\bigr)\\
  +2\Omega^2 \bigl(2|\nablas\log\Omega|^2+\Laplaces\log\Omega\bigr)\bigl(2(\eta,\etab)-|\eta|^2+1\bigr)\\
  +8\Omega^2\bigl(\ds\log\Omega\bigr)\cdot\bigl((\nablas\eta,\etab-\eta)+(\eta,\nablas\etab)\bigr)\\
  +2\Omega^2\Bigl(\bigl(\Laplaces\eta,\etab-\eta\bigr)+\bigl(\nablas\eta,\nablas(\etab-\eta)\bigr)+\bigl(\nablas\eta,\nablas\etab\bigr)+\bigl(\eta,\Laplaces\etab\bigr)\Bigr)
\end{multline}

Similarly, the propagation equation for $\Laplaces\omega$,  in the shear-free case reduces to
\begin{equation}\label{eq:Db:omegabs:shear-free}
  \begin{split}
    \Db\bigl(\Omega^2\Laplaces\omega\bigr)&-\bigl(2\omegab-\Omega\tr\chib\bigr)\Omega^2\Laplaces\omega =
    2\Omega^4\bigl(2|\ds\log\Omega|^2+\Laplaces\log\Omega\bigr)\Bigl(2(\eta,\etab)-\lvert\etab\rvert^2+1\Bigr)\\
    &+2\Omega^4\Bigl(\ds\log\Omega,2(\nablas\eta,\etab)+2(\eta,\nablas\etab)-2(\etab,\nablas\etab)\Bigr)\\
        &+\Omega^4\Bigl(2(\Laplaces\eta,\etab)+2(\eta-\etab,\Laplaces\etab)\Bigr)+2\Omega^4\Bigl(2(\nablas\eta,\nablas\etab)-|\nablas\etab|^2\Bigr)
  \end{split}
\end{equation}

These propagation equations will be used in Section~\ref{sec:coupled:null:expansion}.


\subsection{$\mathrm{L}^\infty$ estimates}
\label{sec:Linfty}



We begin with a Gronwall Lemma that will be used frequently for the o.d.e.'s considered in this section.

\begin{lemma}\label{lemma:ode}
  Let  $f\geq 0$ be a positive function on $I=[u,u_+]$ which satisfies the inequality
  \begin{equation}
    \bigl\lvert \frac{\ud}{\ud u} f^2 \bigr\rvert \leq  2f (a f+ b)
  \end{equation}
  where $a,b\geq 0$ are bounded functions on $I$.
  Then
  \begin{equation}
    f(u)\leq e^{A(u)}\Bigl[f(u_+)+\int_u^{u_+}b(u') e^{-A(u')} \ud u'\Bigr]
  \end{equation}
  where $A(u)=\int_u^{u_+}a(u')\ud u'$.
In particular, if $f(u_+)=0$, and $a_M=\sup_I|a|$, $b_M=\sup_I |b|$, then
\begin{equation}
  f(u)\leq e^{a_M u_+} b_M (u_+-u)\,.
\end{equation}
  Note also that in the case $f(u_+)\neq 0$, but  $u_+-u\ll 1$ we have
\begin{equation}
  |f(u_+)-f(u)|\leq e^{A(u)}\int_u^{u_+}b(u')e^{-A(u')}\ud u'+f(u_+)a_M (u_+-u)+\mathcal{O}(u_+-u)^2
\end{equation}

\end{lemma}

The following comparison Lemma states the relation between length in $u$, $v$ and $r$.

\begin{lemma}\label{lemma:comparison}
  Suppose \eqref{BA:log:Omega:r}, and \eqref{BA:tr:chi}, \eqref{BA:tr:chib} hold. Then
  \begin{equation}
    u_+-u\lesssim \frac{1}{r}\qquad v_+-v\lesssim \frac{1}{r}
  \end{equation}

\end{lemma}

\begin{proof}
  From \eqref{eq:D:r} we have
  \begin{equation}
    v_+-v=\int_v^{v_+}\ud v=\int_r^{r_+}\frac{2}{r}\frac{\ud r}{\overline{\Omega\tr\chi}}\lesssim \frac{1}{r}
  \end{equation}
 and similarly for $u_+-u$.
\end{proof}

\subsubsection{$\Omega$}

Consider the bootstrap assumption \eqref{BA:log:Omega:r}.

  By \eqref{eq:D:Omega} and \eqref{eq:D:r} we have
\begin{equation}\label{eq:D:Omega:r}
  D\Bigl(\frac{\Omega^2}{r^2}\Bigr)=\frac{\Omega^2}{r^2}\bigl(2\omega-\overline{\Omega\tr\chi}\bigr)
\end{equation}
hence  \eqref{BA:omega:tr:chi}, \eqref{BA:tr:chi:average} directly imply
\begin{equation}
  \lvert \log \Bigl( \frac{\Omega^2}{r^2} \Bigr) \rvert \leq \lvert \log \Bigl( \frac{\Omega^2_+}{r^2_+} \Bigr) \rvert+\Delta_I (v_+-v)\,.
\end{equation}

This shows that \eqref{BA:log:Omega:r} can be recovered from (\textbf{A:I}); in fact, we we may choose $r_+:=r\rvert_{\Cb_+}$, and $\Omega_+=\Omega\rvert_{\Cb_+}$ such that
\begin{equation}
  \overline{\log\Bigl(\frac{\Omega_+}{r_+}\Bigr)}=0\,.
\end{equation}

Alternatively we can assume (\textbf{A:\underline{I}}) then our assumptions on the data at $C_+$ suffices to show:

\begin{lemma}\label{lemma:BA:0}
  Suppose \eqref{BA:omegab:tr:chib}, and \eqref{BA:tr:chib:average} holds, and $\Omega^2= r^2-1$ on $C_+$. Then \eqref{BA:log:Omega:r} holds on $\Dp$ for some constant $\Delta_0>0$.
\end{lemma}

\begin{proof}
We have
\begin{equation*}
  \Db\log\frac{\Omega^2}{r^2}=2\omegab-\overline{\Omega\tr\chib}
\end{equation*}
hence
\begin{equation*}
  \lvert \log \Bigl( \frac{\Omega}{r} \Bigr) \rvert \leq  \lvert \log \Bigl( \frac{\Omega_+}{r_+} \Bigr) \rvert +2\Delta_I (u_+-u)
\end{equation*}
where $\log (\frac{\Omega_+}{r_+}) \to 0$ as $v\to v_+$.
\end{proof}

\subsubsection{$\chib$}

\begin{lemma}
  Assume \eqref{BA:tr:chib}, and \eqref{BA:tr:chib:average},  then for $u_+-u\ll 1$, 
  \begin{equation}
  |\tr\chib-2\frac{\Omega}{r}|\lesssim \frac{1}{r}
\end{equation}

\end{lemma}

\begin{proof}
Since $\tr\chi=2\Omega/r$ on $C_+$ it makes sense to consider the quantity $r\tr\chip$:

First note that
\begin{gather*}
  \Db(r\tr\chibp)=\frac{r}{2}\bigl(\overline{\Omega\tr\chib}-\Omega\tr\chib\bigr)\tr\chibp-r\Omega^2|\chibhp|^2\\
  \Db|\chibhp|^2+2\Omega^2\tr\chibp |\chibhp|^2=0
  \end{gather*}
Recall Lemma~\ref{lemma:shear-free} which shows that with our choice of initial data $C_+$,  $\chibh=0$ in $D_+$, hence
\begin{equation*}
  \Db\Bigl(r\tr\chibp e^{-\frac{1}{2}\int_u^{u_+} \underline{X}(u)\ud u}\Bigr)=0 \qquad   \underline{X}(u):=\Omega\tr\chib-\overline{\Omega\tr\chib}
\end{equation*}
which implies, with $r\tr\chip=2$ on $C_+$, that
\begin{equation*}
  |r\tr\chibp-2|\leq  \Delta_I (u_+-u)
\end{equation*}
and the stated bound follows from Lemma~\ref{lemma:comparison}.
  
\end{proof}

Alternatively one may consider the bootstrap assumption \eqref{BA:omegab:tr:chib}:

Since
\begin{equation}\label{eq:Db:trchib}
  \Db\tr\chib=\frac{1}{2}\bigl(2\omegab-\Omega\tr\chib\bigr)\tr\chib-\Omega|\chibh|^2
\end{equation}
we obtain immediately that
\begin{equation}\label{eq:trchib:shear-free}
  \bigl|\tr\chib-\tr\chib_{+}\bigr|\leq \frac{1}{2}\Delta_I (u_+-u) \lesssim \frac{1}{r}
\end{equation}

Either way,
these bound of course already recover the assumption (\textbf{A:\underline{I}}.i),
under the condition \eqref{D:small}.

\subsubsection{$\chi$}
\label{sec:Linfty:chi}

We proceed similarly under the assumptions (\textbf{A:I}). 







Recall \eqref{eq:D:chi}, with  $ \alpha[W]=0$.

\begin{gather}
  D\bigl(\Omega^2|\chih|^2\bigr)=2\bigl(2\omega-\Omega\tr\chi\bigr)\Omega^2|\chih|^2\label{eq:D:chih:vert}\\
  D\tr\chi=\frac{1}{2}\bigl(2\omega-\Omega\tr\chi\bigr)\tr\chi-\Omega|\chih|^2\label{eq:D:tr:chi:no:chih}
\end{gather}

We see that the natural scaling suggested by these equations when solved forward is $\Omega|\chih|=\mathcal{O}(1)$; however as we have seen in Section~\ref{sec:prelim} this is insufficient for our purposes. A better bound can only be obtained after estimates for the torsion have been obtained, see Lemma~\ref{lemma:prelim:Linfty:eta} below.

\begin{lemma}\label{lemma:Linfty:chi}
  Suppose \eqref{BA:omega:tr:chi} holds.  Then
  \begin{gather}
    |\tr\chi-\tr\chi_+|\lesssim \frac{1}{r}\\
        \Omega^2|\chih|^2 \lesssim \Omega^2_+|\chih_+|^2
      \end{gather}
      on $\Dp$ small \eqref{D:small}.
\end{lemma}

\begin{proof}
First it follows from integrating \eqref{eq:D:chih:vert} that
\begin{equation*}
  \Omega^2|\chih|^2 \leq \Omega_+^2|\chih_+|^2 \exp\Bigl[2\Delta_I (v_+-v)\Bigr]\,.
\end{equation*}

   Second  we can write \eqref{eq:D:tr:chi:no:chih} as
  \begin{equation*}
    D\Bigl( \tr\chi e^{\frac{1}{2}A(v)}\Bigr) =-\Omega |\chih|^2 e^{\frac{1}{2}A(v)}
  \end{equation*}
  where we denote by
  \begin{equation*}
    A(v)=\int_{v}^{v_+}2\omega-\Omega\tr\chi\ud v'\qquad |A(v)|\leq \Delta_I (v_+-v)
  \end{equation*}
  Thus, 
  \begin{equation*}
    |\tr\chi_+-\tr\chi|\leq e^{\Delta_I(v_+-v)}\int_v^{v_+}\Omega|\chih|^2 \ud v' +\tr\chi_+  \Delta_I (v_+-v)
  \end{equation*}
and the statement follows from Lemma~\ref{lemma:comparison}.

\end{proof}

In particular the assumption \eqref{BA:tr:chi} is already recovered.

Another simple oberservation that will be useful below is
\begin{lemma}\label{lemma:Omegatrchi:r}
  Suppose \eqref{BA:tr:chi:average}, and \eqref{D:small} holds. Then
  \begin{equation}
    \Bigl\lvert \int_v^{v_+}\Omega\tr\chi\ud v-2\log\frac{r(u,v_+)}{r(u,v)}\Bigr| \lesssim  \frac{\Delta_I}{r}
  \end{equation}
  Similarly if \eqref{BA:tr:chib:average} and \eqref{D:small} holds, then
    \begin{equation}
    \Bigl\lvert \int_u^{u_+}\Omega\tr\chib\ud u-2\log\frac{r(u_+,v)}{r(u,v)}\Bigr| \lesssim \frac{\Delta_I}{r}
  \end{equation}

\end{lemma}

\begin{proof}
  We have
  \begin{equation*}
    \begin{split}
      \Bigl\lvert \int_v^{v_+}\Omega\tr\chi\ud v-2\log \frac{r(u,v_+)}{r(u,v)}\Bigr\rvert &=       \Bigl\lvert \int_v^{v_+}\Omega\tr\chi\ud v-2 \int_v^{v_+} D\log r(u,v')\ud v'\Bigr\rvert \\
      &=      \Bigl\lvert \int_v^{v_+}\Omega\tr\chi- \overline{\Omega\tr\chi} \ud v\Bigr\rvert \leq \Delta_I (v_+-v)
    \end{split}
  \end{equation*}
Similarly for the conjugate inequality.
\end{proof}

\subsubsection{$\eta$, $\etab$}
\label{sec:Linfty:eta}


Consider the propagation equation \eqref{eq:propagation:torsion}, from which we derive
\begin{equation}
  D |\etab|^2=-4\Omega\chi(\etab,\etab)+4(\etab,\ds\omega)
\end{equation}
and after incorporating the weight $\Omega^2$ for $\etab$, and we obtain
 \begin{equation}\label{eq:D:Omega:etab}
   D\bigl(\Omega^4|\etab|^2\bigr)=2\bigl(2\omega-\Omega\tr\chi\bigr)\Omega^4|\etab|^2-4\Omega\chih(\Omega^2\etab,\Omega^2\etab)+4\Omega^4(\etab,\ds\omega)
 \end{equation}

Recall the bootstrap assumption \eqref{BA:Omega:ds:omega}, and \eqref{BA:Omega:ds:omegab}.

\begin{lemma}\label{lemma:torsion}
  Suppose \eqref{BA:omega:tr:chi}, and \eqref{BA:Omega:ds:omega} holds.
  Then
  \begin{equation}
    \Omega^2|\etab|\leq C(\Delta_I,\Omega_+\chih_+)\Bigl(\Omega^2_+|\etab_+|+\Delta_{II} (v_+-v)
    \Bigr)
  \end{equation}
holds on $\Dp$ for \eqref{D:small}.
\end{lemma}

\begin{proof}
Since for $n=\Omega^2|\etab|$
 \begin{equation*}
   D\bigl(n^2\bigr)\leq 2n\Bigl(\bigl\lvert 2\omega-\Omega\tr\chi\bigr\rvert n+2\Omega|\chih|n+2\Omega^2|\ds\omega|\Bigr)
 \end{equation*}
it follows from Lemma~\ref{lemma:ode}, and Lemma~\ref{lemma:Linfty:chi}, that
\begin{equation*}
  n(u,v)\leq e^{\Delta_I+\Omega_+|\chih_+|}\Bigl[ n(u,v_+) + \Delta_{II}(v_+-v)\Bigr]
\end{equation*}
See Lemma~\ref{lemma:Linfty:chi} for the bound on $\Omega\chih$ used here.
\end{proof}

Similarly we can derive $\mathrm{L}^\infty$ estimates for $\Omega^2\eta$ under the assumptions \eqref{BA:Omega:ds:omegab}, and \eqref{BA:omegab:tr:chib}. See also Lemma~\ref{lemma:prelim:Linfty:eta} below.

\begin{remark}
  It is clear from \eqref{eq:D:Omega:etab} that
  to obtain an $\mathrm{L}^\infty$ estimate for $\eta$, the assumption \eqref{BA:Omega:ds:omega} on $\ds\omega$ has to be made in $\mathrm{L}^\infty$. This means that we will have to control $\nablas\omega$ and  $\nablas^2\omega$ in $\mathrm{L}^4$ from the the equation for $\Laplaces\omega$. See Section~\ref{sec:coupled:null:expansion}.
\end{remark}




\begin{remark} In \CCh{3.4} the torsion $\eta$, and $\etab$ are estimated using the equations for $D\eta$, and $\Db \etab$.
  This approach cannot be adapted for our purposes. For if we consider 
\begin{equation*}
  \Db\bigl(\Omega^2|\etab|^2\bigr)=\bigl(2\omegab-\Omega\tr\chib\bigr)\Omega^2|\etab|^2-2\Omega\chibh(\Omega\etab,\Omega\etab)+\Omega\chib(\Omega\etab,\Omega\eta)
\end{equation*}
then the last term poses an obstruction, because $\Omega\tr\chib$ is not bounded:
\begin{gather*}
  \Omega^2 |\etab|^2 \leq e^{A(u)}\int_0^u e^{-A(u')}\frac{1}{2}\Omega|\chib|\Omega|\eta| \ud u\\
    A(u) = \int_0^u\frac{1}{2}(2\omegab-\Omega\tr\chib)+\Omega|\chibh|\ud u
  \end{gather*}
  Similarly for $\Omega\eta$.
An advantage of using \eqref{eq:D:Omega:etab} as opposed to the approach of \CCh{3.4},
is that  here the equations for $\eta$ and $\etab$ are \emph{not} coupled.
 
\end{remark}

\subsubsection{$2\omega-\Omega\tr\chi$}
\label{sec:L:infty:omega}

We derive from \eqref{eq:Db:omega} and the conjugate equation of \eqref{eq:D:Omega:tr:chib} that
\begin{equation}\label{eq:Db:omega:trchi}
\Db\bigl(2\omega-\Omega\tr\chi\bigr)   = -2\Omega^2\Bigl(3\rho[W]-\bigl(\chibh,\chih)+K+\divs\eta+\lvert \eta\rvert^2 -2(\eta,\etab)+\lvert\etab\rvert^2\Bigr)
\end{equation}

Since de Sitter is conformally flat we have $\rho[W]=0$, and we have shown in Lemma~\ref{lemma:shear-free} that with our data, $\chibh=0$.
   Then this equations reduces to
\begin{equation}\label{eq:Db:omega:trchi:simple}
  \Db\bigl(2\omega-\Omega\tr\chi\bigr)   = -2\Omega^2 \bigl( K + \divs\eta+\lvert \eta\rvert^2 -2(\eta,\etab)+\lvert\etab\rvert^2 \bigr)
\end{equation}

We now observe that by \eqref{eq:gauss:hat}, and the definition of $\breve{\mu}$ in \eqref{eq:mu:breve}
\begin{equation}
  K+\divs\eta=1-\frac{1}{4}\tr\chi\tr\chib-\rho[W]+\frac{1}{2}(\chih,\chibh)+\divs\eta=-\breve{\mu}-2\rho[W]+(\chih,\chibh)
\end{equation}

Thus by Lemma~\ref{lemma:null:expansion:mu}, the assumption \eqref{BA:omega:tr:chi} is recovered for \eqref{D:small}, provided we establish the boundedness of $\Omega^2\breve{\mu}$.

\begin{remark}
  The quantity $2\omega-\Omega\tr\chi$ remains bounded, while $2\omega$ and $\Omega\tr\chi$ individually diverge. This is due to a cancellation in the propagation equation \eqref{eq:Db:omega:trchi}, which has motivated us to consider this favorable quantity.
 
\end{remark}

Alternatively,  we may write
\begin{equation}
  K+\divs\eta=K-\divs\etab+\divs(\eta+\etab)=\kappab+2\Laplaces\log\Omega
\end{equation}
where $\kappab$ is defined as in \eqref{eq:kappab}.

We can now proceed in two ways:
\begin{enumerate}

\item We look at the propagation equation for $\kappab$ directly, which is given by \eqref{eq:Db:kappab:shear-free} in the shear-free case,
  \begin{equation}\label{eq:Db:kappab:mu}
    \Db\bigl(\Omega^2\kappab\bigr)=\bigl(2\omegab-\Omega\tr\chib\bigr)\Omega^2\kappab-\Omega^3\tr\chib\bigl(|\eta|^2-\mu\bigr)
  \end{equation}
   from which we see that to estimate $\Omega^2\kappab$, we need a bound on $\Omega^3\mu$.

   \item Alternatively, we can write
\begin{equation}
  \kappab=\mub+1-\frac{1}{4}\tr\chi\tr\chib
\end{equation}
which modulo a bound on $\Omega^2\mub$, requires a bound on $\frac{1}{4}\tr\chi\tr\chib-1$.
The latter we can derive from the conjugate equation to \eqref{eq:D:trchitrchib}, which  in the shear-free case reads
  \begin{equation}\label{eq:Db:trchitrchib}
    \Db\bigl(\frac{1}{4}\tr\chi\tr\chib-1\bigr)+\Omega\tr\chib\bigl(\frac{1}{4}\tr\chi\tr\chib-1\bigr)=-\frac{1}{2}\Omega\tr\chib\bigl(\mu-|\eta|^2\bigr)
  \end{equation}
which again requires an estimate on $\Omega^3\mu$, to obtain a bound on $\Omega^2\kappab$.
\end{enumerate}

Thus in both cases, to close the estimate for $2\omega-\Omega\tr\chi$, we need to consider the Hodge system for $\eta$, and obtain a suitable estimate on the mass aspect function $\mu$.

\begin{lemma}
  Suppose $\Omega^3 | \mub | \lesssim 1$, and $\Omega^3 |\mu|\lesssim 1$. Then
\begin{equation*}
  |2\omega-\Omega\tr\chi|\lesssim \frac{1}{r}
\end{equation*}

\end{lemma}

\begin{proof}
  Given the bounds on $\mu$, and $\mub$, we have $|\Omega^2\eta|\lesssim 1$, and $|\Omega^2\etab|\lesssim 1$; c.f.~discussion in Section~\ref{sec:prelim:torsion}.
Then it follows directly from \eqref{eq:Db:kappab:mu} that
  \begin{equation*}
    \Omega^2\kappab\lesssim e^{\Delta_I} \Bigl[ \Omega^2\kappab(u_+,v)+\int_u^{u_+}\tr\chib \Omega^3 \bigl(|\eta|^2+|\mu|\bigr)\ud u\Bigr]\lesssim 1
  \end{equation*}
  because on $C_+$, $\Omega^2\kappab=1+r^{-2}$.
Moreover, we infer from \eqref{eq:Laplace:log:Omega} that \[\Omega^2|\Laplaces\log\Omega|\leq \Omega^2|\mu|+\Omega^2|\mub|\lesssim 1\]

  Then however, the statement follows immediately from \eqref{eq:Db:omega:trchi:simple}, and the fact that on $C_+$, $2\omega-\Omega\tr\chi=1/r$:
\begin{equation*}
  |2\omega-\Omega\tr\chi|\lesssim  \frac{1}{r}+\int_u^{u_+}\Omega^2\bigl(|\kappab|+|\Laplaces\log\Omega| +|\eta|^2+|\etab|^2\bigr)\ud u \lesssim \frac{1}{r}+(u_+-u)
\end{equation*}

\end{proof}

\subsubsection{$\eta, \etab$:  without mass aspect function}
\label{sec:torsion:alternate}

As we have seen in Section~\ref{sec:prelim:torsion}, the assumption on $\chih$ consistent with the required bounds on $\eta$, $\etab$, is on the level of an $L^\infty$ bound on $\Omega^2\chih$; see \eqref{eq:mu:chi}. However, as dicussed in Section~\ref{sec:Linfty:chi}, from \eqref{eq:D:chih:vert} we see that by considering the propagation equation of $\chih$ along $C_u$ we can at most obtain a bound on $\Omega|\chih|$. Alternatively we could consider the propagation equation for $\chih$ along $\Cb_v$, which in shear-free case reads:
\begin{equation}
   \Db(\Omega\chih)  =\Omega^2\nablas\otimesh\eta+\Omega^2\eta\otimesh\eta+\frac{1}{2}\Omega^2\tr\chib\chih \label{eq:Dbh:chih:shear-free}
 \end{equation}
Similarly we can estimate $\eta$ using the propagation equation along $\Cb_v$:
\begin{equation}\label{eq:Db:eta:shear-free}
  \Db\eta=-\frac{1}{2}\Omega\tr\chib\eta+2\ds\omega
\end{equation}

The following Lemma shows that the required bounds on $\eta$ can also be derived (in the shear-free case) purely from the propagation equations, i.e.~not appealing to the Hodge system satisfied by $\eta$ involving the mass aspect function. However, such an argument relies on an $\mathrm{L}^\infty$ bound on $\nablas^2\omegab$, or alternatively and $\mathrm{L}^4$ bound on $\nablas^3\omegab$ in addition to the assumptions \eqref{BA:Omega:dd:omegab}. An estimate based on that assumption cannot close by using propagation equations alone, and can only be recovered using the elliptic equation satisfied by $\omegab$, coupled to the propagation equation for $\omegabs=\Laplaces\omegab$.

\begin{lemma}\label{lemma:prelim:Linfty:eta}
  Consider the shear-free case $\chibh=0$, and suppose \eqref{BA:omegab:tr:chib}, \eqref{BA:Omega:ds:omegab}, and
  \begin{equation}
     \Linf{u,v}{\Omega^3\nablas^2\omegab}\lesssim 1
  \end{equation}
hold.
Then with the initial data $\chih=0$, $\eta=0$ on $C_+$,
\begin{subequations} \label{eq:prelim:Linfty:eta:chih}
\begin{gather}
  \Omega^2|\eta|+|\Omega^3\nablas\eta|\lesssim 1\\
  \Omega^2|\chih|\lesssim 1
\end{gather}
\end{subequations}
Moreover if in addition to \eqref{BA:omega:tr:chi}, and \eqref{BA:Omega:ds:omega} we suppose that
\begin{equation}
  \Linf{u,v}{\Omega^3\nablas^2\omega}\lesssim 1
\end{equation}
then we also have
\begin{equation} \label{eq:prelim:Linfty:etab}
  \Omega^2|\etab| +\Omega^3|\nablas\etab|\lesssim \Linf{u,v_+}{\Omega^2\etab}+\Linf{u,v_+}{\Omega^3\nablas\etab}
\end{equation}

\end{lemma}

\begin{remark}
  The stated bounds \eqref{eq:prelim:Linfty:eta:chih} do not require the smallness assumption \eqref{D:small}. In fact, if \eqref{D:small} is assumed, then
  \begin{equation}
    \Linf{u,v}{\Omega^2\eta}+\Linf{u,v}{\Omega^3\nablas\eta}+\Linf{u,v}{\Omega^2\chih} \lesssim 1/r\,.
  \end{equation}

\end{remark}

\begin{remark}
  The  bounds \eqref{eq:prelim:Linfty:etab} for $\etab$ are derived using propagation equations along $C_u$, which are not shear-free, and rely on
  \begin{equation}
      \Linf{u,v}{\Omega\nablas(\Omega\chih)}+\Linf{u,v}{\Omega\ds(\Omega\tr\chi)}\lesssim 1 \label{BA:nablas:Omega:chi}
  \end{equation}
  These are established in the present setting in Proposition~\ref{prop:L4:chi} below, but in general cannot be recovered using propagation equations alone; rather these are established using systems of elliptic equations coupled to propagation equations for $\ds\tr\chi$;  see Section~\ref{sec:coupled:codazzi}. In fact, for the proof of \eqref{BA:nablas:Omega:chi} using the Codazzi equations see Proposition~\ref{prop:coupled:chi}.
\end{remark}

\begin{proof}
  Since $\eta$ is a 1-form, it follows from \eqref{eq:Db:eta:shear-free} that
  \begin{equation*}
    \Db|\eta|^2+2\Omega\tr\chib|\eta|^2=4(\eta,\ds\omegab)
  \end{equation*}
  and it follows that
  \begin{equation*}
    \Omega^2|\eta|\lesssim \int_u^{u_+} \Omega^2|\ds\omegab|\ud u  \lesssim u_+-u
  \end{equation*}
Moreover by \CLemma{4.1}, and the Codazzi equation \eqref{eq:codazzi:shear-free},
\begin{gather*}
  \Db\nablas\eta-\nablas\Db\eta=-\frac{1}{2}\Omega\tr\chib\eta\otimesh\eta\\
  \nablas\Db\eta=-\frac{1}{2}\Omega\tr\chib \eta\otimes\eta-\frac{1}{2}\Omega\tr\chib\nablas\eta+2\nablas\ds\omegab
\end{gather*}
and thus by \CLemma{4.2}
\begin{equation*}
  \Db|\nablas\eta|^2+3\Omega\tr\chib|\nablas\eta|^2\leq 2 |\nablas\eta|\Bigl(\Omega\tr\chib |\eta|^2+2|\nablas\ds\omegab|\Bigr)
\end{equation*}
which implies that
\begin{equation*}
  \Linf{u,v}{\Omega^3\nablas\eta}\lesssim \int_u^{u_+}\tr\chib \Bigl( \Linf{u',v}{\Omega^2\eta}^2 + \Linf{u',v}{\Omega^3\nablas^2\omegab}\Bigr)\ud u'\lesssim u_+-u
\end{equation*}

  Now we can turn to the propagation equation for $\chih$.
  Since $\chih$ is a 2-covariant tensorfield \eqref{eq:Dbh:chih:shear-free}  implies
  \begin{equation*}
    \Db\bigl(\Omega^4|\chih|^2\bigr)=\bigl(2\omegab-\Omega\tr\chib)\Omega^4|\chih|^2+2\Bigl(\Omega^2\chih,\Omega^3\bigl(\nablas\otimesh\eta+\eta\otimesh\eta\bigr)\Bigr)
  \end{equation*}
which immediately implies the boundedness of $\Omega^2|\chih|$, in view of the estimates for $\eta$, and $\nablas\eta$ established above.

Moreover, from \eqref{eq:D:etab}, and \CLemma{4.2} we now have
\begin{gather*}
  D|\etab|^2+\Omega\tr\chi|\etab|^2=2(\etab,D\etab)-2\Omega\chih(\etab,\etab)\\
  D\etab=-\Omega\chih^\sharp\cdot\etab-\frac{1}{2}\Omega\tr\chi\etab+2\ds\omega\\
  D|\etab|^2+2\Omega\tr\chi|\etab|^2=-4\Omega\chih(\etab,\etab)+4(\etab,\ds\omega)
\end{gather*}
and therefore, using in particular the boundedness of $\Omega|\chih|$ proven above,
\begin{equation*}
  \Linf{u,v}{\Omega^2\etab}\lesssim \Linf{u,v_+}{\Omega^2\etab}+\int_v^{v_+}\Omega^2|\ds\omega|\ud v
\end{equation*}

Similarly to the above we can commute the equation for $\etab$, and find
\begin{equation*}
  D|\nablas\etab|^2+3\Omega\tr\chi|\nablas\etab|^2 \leq 2|\nablas\etab|\Bigl(\bigl|\nablas(\Omega\chi)\cdot\etab\bigr|+\bigl|\Omega\chih\cdot\nablas\etab\bigl|+\bigl|\nablas^2\omega|\Bigr)
\end{equation*}
which implies the stated bound for $\nablas\eta$, in view of the assumption $\nablas^2\omega$, and \eqref{BA:nablas:Omega:chi} which is established in Proposition~\ref{prop:L4:chi}.
\end{proof}


\subsection{$\mathrm{L}^4$ estimates from propagation equations}
\label{sec:L4}


\subsubsection{General Lemmas}

Consider the following dimensionless norms:
\begin{equation}
  \dLp{p}{u,v}{\theta}:= \Bigl( \frac{1}{4\pi r^2}\int_{S_{u,v}}|\theta|^p_{\gs}\dm{\gs}\Bigr)^{1/p}
\end{equation}

\begin{lemma}\label{lemma:norm:comparison}
  Let $\rho$ be a non-negative function on $C_u$, $\Phi_v$ the flow generated by $L$, and $\rho(v)$ be the function on $S_{u,v_+}$ defined by $[\rho(v)](q)=\rho\circ\Phi_v(q)$. Assume \eqref{BA:tr:chi:average}, then
\begin{equation}
\dLp{p}{u,v}{\rho}\simeq_I  \dLp{p}{u,v_+}{\rho(v)}
\end{equation}
Similarly if $\underline{\Phi}_u$ is the flow generated by $\Lb$, and $\rho(u):= \rho\circ\underline{\Phi}_u$, then if \eqref{BA:tr:chib:average} holds, 
\begin{equation}
\dLp{p}{u,v}{\rho}\simeq_I  \dLp{p}{u_+,v}{\rho(u)}
\end{equation}

\end{lemma}


\begin{proof}
  One has
  \begin{equation*}
    \dLp{p}{u,v}{\rho}^p=\frac{1}{4\pi r^2(u,v)}\int_{S_{u,v}} \rho^p\dm{\gs}=\frac{1}{4\pi r^2(u,v)}\int_{S_{u,v_+}}\bigl(\Phi_v^\ast\rho)^p\Phi_v^\ast\dm{\gs}
  \end{equation*}
and, c.f.~proof of \CLemma{4.3}
\begin{equation*}
  \Phi_v^\ast\dm{\gs}(q)=\exp\Bigl[\int_v^{v_+}\Omega\tr\chi(\Phi_v(q))\ud v\Bigr]\dm{\gs}(q)
\end{equation*}
In view of Lemma~\ref{lemma:Omegatrchi:r}
\begin{equation*}
  \begin{split}
    \exp\Bigl[\int_v^{v_+}\Omega\tr\chi(\Phi_v(q))\ud v\Bigr]&= \exp\Bigl[\int_v^{v_+}\Omega\tr\chi(\Phi_v(q))\ud v-2\log\frac{r(u,v_+)}{r(u,v)}\Bigr]\frac{r^2(u,v_+)}{r^2(u,v)}\\
    &\simeq_I \frac{r^2(u,v_+)}{r^2(u,v)}
  \end{split}
\end{equation*}
and the statement follows.
\end{proof}

\begin{lemma}
  \label{lemma:dLp:gronwall}
  Suppose $\psi$ is a non-negative function on $C_u$ and satisfies the inequality
  \begin{equation}
    |D\psi^2| \leq 2\psi(a\psi+\rho)
  \end{equation}
  where $a\geq 0$, and $\rho\geq 0$. Suppose moreover that \eqref{BA:tr:chi:average} holds, and
  \begin{equation}
    |a|\leq \Delta
  \end{equation}
  Then
  \begin{equation}
    \dLp{p}{u,v}{\psi} \lesssim_I e^{\Delta (v_+-v)} \Bigr[ \dLp{p}{u,v_+}{\psi}+\int_v^{v_+}\dLp{p}{u,v'}{\rho}\ud v'\Bigr]
  \end{equation}
  Similarly for a non-negative function $\psi$ on $\Cb_v$ satisfying the inequality
    \begin{equation}
    |\Db\psi^2|\leq 2\psi(a\psi+\rho)
  \end{equation}
  where $|a|\leq \Delta$, we have provided that \eqref{BA:tr:chib:average} holds 
  \begin{equation}
    \dLp{p}{u,v}{\psi} \lesssim_I e^{\Delta (u_+-u)} \Bigr[ \dLp{p}{u_+,v}{\psi}+\int_u^{u_+}\dLp{p}{u',v}{\rho}\ud u'\Bigr]
  \end{equation}

\end{lemma}

\begin{proof}
  Set
  \begin{equation*}
    \psi(v)(q)=\psi\circ\Phi_v(q)\qquad q\in S_{u,v_+}
  \end{equation*}
Then by Lemma~\ref{lemma:ode}
\begin{equation*}
  \psi(v)\leq e^{A(v)}\Bigl(\psi(v_+)+\int_v^{v_+}e^{-A(v')}\rho(v')\ud v'\Bigr)\qquad A(v)=\int_v^{v_+}a(v')\ud v'
\end{equation*}
and it follows that
\begin{equation*}
  \dLp{p}{u,v_+}{\psi(v)}\leq e^{\Delta (v_+-v)}\Bigl[   \dLp{p}{u,v_+}{\psi(v_+)} + \int_v^{v_+}   \dLp{p}{u,v_+}{\rho(v')} \ud v' \Bigr]
\end{equation*}
The statement the follows from Lemma~\ref{lemma:norm:comparison}.
\end{proof}

\subsubsection{$\chi$}
\label{sec:L4:D:chi}





Consider
\begin{subequations}
\begin{align}
  \tilde{\theta}_p&=\Omega^{p-2}\ds\bigl(\Omega\tr\chi\bigr)\\
  \thetas_p&=\Omega^{p-2}\nablas(\Omega\chih)
\end{align}
\end{subequations}
We first note
\begin{subequations}
\begin{gather}
  D\bigl(\Omega\tr\chi\bigr)=2\omega\Omega\tr\chi-\frac{1}{2}(\Omega\tr\chi)^2-\Omega^2|\chih|^2\\
  D\ds\bigl(\Omega\tr\chi\bigr)=\bigl(2\omega-\Omega\tr\chi\bigr)\ds\bigl(\Omega\tr\chi\bigr)+\Omega\tr\chi\ds(2\omega)-2\Omega\chih^\sharp\cdot \nablas(\Omega\chih)
\end{gather}
\end{subequations}
and since $\tr D\chih=2\Omega|\chih|^2$, we obtain from \eqref{eq:Dh:chih} that
\begin{subequations}
\begin{gather}
  D\chih=\hat{D}\chih+\Omega|\chih|^2\gs=\omega\chih+\Omega|\chih|^2\gs-\Omega\alpha\\
  D\bigl(\Omega\chih\bigr)=2\omega \,\Omega\chih-\Omega^2\alpha+\Omega^2|\chih|^2\gs
\end{gather}
\end{subequations}
and a corresponding propagation equation for $\nablas(\Omega\chih)$ from \CLemma{4.1}.
For convenience let us denote by
\begin{subequations}
  \begin{align}
    \dLp{p}{u,v}{\Omega \nablas(\Omega\chi)}&:=\dLp{p}{u,v}{\Omega \nablas(\Omega\chih)}+\dLp{p}{u,v}{\Omega\ds(\Omega\tr\chi)} \\
    |\Omega \nablas (\Omega\chi)|^2_{L^\infty(S_{u,v})} &:= \sup_{S_{u,v}}|\Omega\nablas(\Omega\chih)|_{\gs\rvert_{S_{u,v}}}^2+\sup_{S_{u,v}} |\Omega\ds(\Omega\tr\chi)|_{\gs\rvert_{S_{u,v}}}^2
  \end{align}
\end{subequations}
\begin{proposition}\label{prop:L4:chi}
  Assume \eqref{BA:omega:tr:chi} and \eqref{D:small} hold. Then 
  \begin{equation}
     \dLp{p}{u,v}{\Omega \nablas(\Omega\chi)} \lesssim  \dLp{p}{u,v_+}{\Omega \nablas(\Omega\chi)}+\int_v^{v_+}\dLp{p}{u,v'}{\Omega^2\ds\omega}\ud v'
  \end{equation}
  If in addition \eqref{BA:Omega:ds:omega} holds, then
  \begin{equation}
    \label{eq:d:chi:Linfty}
    |\Omega \nablas (\Omega\chi)|_{L^\infty(S_{u,v})} \lesssim |\Omega \nablas (\Omega\chi)|_{L^\infty(S_{u,v_+})}
  \end{equation}

 \end{proposition}

 \begin{remark}
   In view of Lemma~\ref{lemma:morrey} this in particular recovers the assumption \eqref{BA:tr:chi:average}.
 \end{remark}

\begin{proof}
  Consider the propagation equation for $\ds(\Omega\tr\chi)$ derived above.
  Using \CLemma{4.2} we infer that
  \begin{multline*}
    D|\tilde{\theta}_p|^2+\Omega\tr\chi|\tilde{\theta}_p|^2=2(\tilde{\theta}_p,D\tilde{\theta}_p)-2\Omega\chih_B^{\sharp A}(\tilde{\theta}_p)^B(\tilde{\theta}_p)_A\\
    =2(p\omega-\Omega\tr\chi)|\tilde{\theta}_p|^2+2\Omega^{p-2}\Omega\tr\chi\bigl(\tilde{\theta}_p,\ds(2\omega)\bigr)-2\Omega^{p-2}\Bigl(\tilde{\theta}_p,2\Omega\chih\cdot\nablas(\Omega\chih)\Bigr)-2\Omega\chih_B^{\sharp A}(\tilde{\theta}_p)^B(\tilde{\theta}_p)_A
  \end{multline*}
thus
\begin{equation*}
  D|\tilde{\theta}_p|^2\leq 2\Bigl(\bigl\lvert p\omega-\frac{3}{2}\Omega\tr\chi\bigr\rvert+\Omega|\chih|\Bigr)|\tilde{\theta}_p|^2\\
  +4|\tilde{\theta}_p|\Bigl(\tr\chi\Omega^{p-1}|\ds\omega|+|\Omega\chih|\bigl\lvert \thetas_p\bigr\rvert\Bigr)
\end{equation*}

Now let us derive the propagation equation for $\thetas_p$.
By \CLemma{4.1}
\begin{equation*}
  \begin{split}
    \bigl(D\nablas(\Omega\chih)\bigr)_{ABC}=\bigl(\nablas D(\Omega\chih)\bigr)_{ABC}&-(D\Gammas)_{AB}^D\Omega\chih_{DC}-(D\Gammas)_{AC}^D\Omega\chih_{BD}\\
    =\bigl(\nablas D(\Omega\chih)\bigr)_{ABC}&-\nablas_A(\Omega\chi)_B^D\Omega\chih_{DC}-\nablas_B(\Omega\chi)_A^D\Omega\chih_{DC}+\nablas^D(\Omega\chi)_{AB}\Omega\chih_{DC}\\
    &-\nablas_A(\Omega\chi)_C^D\Omega\chih_{BD}-\nablas_C(\Omega\chi)_A^D\Omega\chih_{BD}+\nablas^D(\Omega\chi)_{AC}\Omega\chih_{BD}
  \end{split}
\end{equation*}
hence, in view of the propagation equation for $\Omega\chih$ derived above,
\begin{gather*}
  D\bigl(\nablas(\Omega\chih)\bigr)=2\omega\nablas(\Omega\chih)+\Omega\chih\cdot\nablas(2\omega)+h\cdot\ds(\Omega\tr\chi)+i\cdot\nablas(\Omega\chih)\\
  |h|\leq \Omega\lvert\chih\rvert\qquad |i|\leq \Omega|\chih|
\end{gather*}
and thus 
\begin{equation*}
   D\bigl(\Omega^{p-2}\nablas(\Omega\chih)\bigr)= p \omega\thetas_p+\Omega\chih\cdot \Omega^{p-2}\nablas(2\omega)+h\cdot \tilde{\theta}_p+i\cdot\thetas_p
\end{equation*}

Therefore in view of \CLemma{4.2}, applied to the 3-form $\thetas_p$:
\begin{equation*}
  D|\thetas_p|^2+3\Omega\tr\chi|\thetas_p|^2=2(\thetas_p,D\thetas_p)-2\sum_{i=1}^3\Omega\chih_{B_i}^{A_i}\thetas^{A_1\cdot\rangle B_i\langle\cdot A_3}\thetas_{A_1A_2A_3}
\end{equation*}
we obtain
\begin{equation*}
  D|\thetas_p|^2\lesssim \bigl\lvert 2p\omega-3\Omega\tr\chi\bigr\rvert |\thetas_p|^2
  + |\thetas_p|\,\Omega|\chih|\,\bigl\{\Omega^{p-2}|\nablas\omega|+|\tilde{\theta}_p|+|\thetas_p|\bigr\}
\end{equation*}

Now set
\begin{equation*}
  \hat{\theta}_p:=\sqrt{|\tilde{\theta}_p|^2+|\thetas_p|^2}
\end{equation*}
then the above is summarized by
\begin{equation*}
  D\hat{\theta}_p^2 \lesssim \Bigl( \bigl\lvert 2p\omega-3\Omega\tr\chi\bigr\rvert +\Omega|\chih| \Bigr)\hat{\theta}_p^2+ \Omega|\chi|\hat{\theta}_p\Omega^{p-2}|\ds\omega|
\end{equation*}

If we now fix $p=3$, then $| 2p\omega-3\Omega\tr\chi | \leq 3\Delta_I$, and by \eqref{BA:Omega:ds:omega}:
\begin{equation*}
  \Omega\Omega^{p-2}|\ds\omega| = \Omega^{2}|\ds\omega| \leq \Delta_{II}
\end{equation*}
Hence we can directly apply Lemma~\ref{lemma:dLp:gronwall} to conclude on \eqref{eq:d:chi:Linfty} with Lemma~\ref{lemma:Linfty:chi}.

\end{proof}


\subsubsection{$\chib$}

In analogy to the treatment in the previous section we derive from \eqref{eq:Db:chibp}
\begin{subequations}
\begin{align}
  \Db\ds(\Omega\tr\chib)&=-\Omega\chibh\cdot\nablas(\Omega\chibh)+\bigl(2\omegab-\Omega\tr\chib)\ds(\Omega\tr\chib)+\Omega\tr\chib\ds(2\omegab)\\
  \Db(\Omega\chibh)&=2\omegab\Omega\chibh+|\Omega\chibh|^2\gs
\end{align}
\end{subequations}
where we used $\alphab[W]=0$, hence by \CLemma{4.1}
\begin{subequations}
\begin{gather}
  \Db\nablas(\Omega\chibh)=\Omega\chibh \cdot\nablas(2\omegab)+2\omegab\nablas(\Omega\chibh)+i\cdot\nablas(\Omega\chibh)+j\cdot \nablas(\Omega\tr\chib)\\
  |i|,|j|\leq \Omega|\chibh|
\end{gather}
\end{subequations}

Note however that in the shear-free case $\chibh=0$ this system drastically simplifies and decouples:
\begin{equation}
  \Db\ds(\Omega\tr\chib)=\bigl(2\omegab-\Omega\tr\chib)\ds(\Omega\tr\chib)+\Omega\tr\chib\ds(2\omegab)
\end{equation}
This  yields by \CLemma{4.2}
\begin{equation}
  \Db|\Omega\ds(\Omega\tr\chib)|^2=2\bigl(2\omegab-\Omega\tr\chib\bigr)|\Omega\ds(\Omega\tr\chib)|^2+2\tr\chib\bigl(\Omega\ds(\Omega\tr\chib),\Omega^2\ds(2\omegab)\bigr)
\end{equation}
which immediately implies:
\begin{lemma}\label{lemma:ds:trchib}
 Suppose \eqref{BA:omegab:tr:chib} and \eqref{BA:Omega:ds:omegab} hold. Then, in the shear-free case $\chibh=0$,
  \begin{equation}
    |\Omega\ds(\Omega\tr\chib)|\lesssim u_+-u
  \end{equation}
\end{lemma}



\begin{remark}
  In view of Lemma~\ref{lemma:morrey} this in particular recovers the bootstrap assumption \eqref{BA:tr:chib:average}.
\end{remark}

\subsubsection{$\nablas(\eta,\etab)$}
\label{sec:L4:nablas:eta}

Let us first derive an estimate for $\nablas\eta$ in the shear-free case $\chibh=0$.

Recall that with $\betab[W]=0$ we have
\begin{equation}
  \Db\eta=-\Omega\chib^\sharp\cdot\eta+2\ds\omegab=-\frac{1}{2}\Omega\tr\chib\eta+2\ds\omegab
\end{equation}
hence in view of \CLemma{4.1} we obtain an equation of the form
\begin{gather}
  \Db\nablas\eta=-\frac{1}{2}\Omega\tr\chib\nablas\eta+2\nablas\ds\omegab-2i\cdot \eta\\
  |i|\leq |\nablas(\Omega\tr\chib)|\notag
\end{gather}
and using \CLemma{4.2} we arrive at the differential equation
\begin{equation}\label{eq:Db:nablas:eta}
  \Db|\nablas\eta|^2+3\Omega\tr\chib|\nablas\eta|^2 = 4\Bigl(\nablas\eta,\nablas^2\omegab+i\cdot\eta\Bigr)
\end{equation}
which implies that
\begin{equation}\label{eq:Db:Omega:nablas:eta}
  \Db\Bigl(\Omega^6|\nablas\eta|^2\Bigr) = 3\bigl(2\omegab-\Omega\tr\chib\bigr)\Omega^6|\nablas\eta|^2+4\Bigl(\Omega^3\nablas\eta,\Omega^3\nablas^2\omegab+\Omega i\cdot\Omega^2\eta\Bigr)
\end{equation}

\begin{lemma}\label{lemma:nablas:eta}
  Suppose the assumptions  \eqref{BA:omegab:tr:chib}  and  \emph{(\textbf{A:\underline{II}})} hold, then
    \begin{equation}
    \dLp{4}{u,v}{\Omega^3\nablas\eta} \lesssim_{I,II} u_+-u
  \end{equation}

\end{lemma}

\begin{proof}
  We apply Lemma~\ref{lemma:dLp:gronwall} to \eqref{eq:Db:Omega:nablas:eta} and immediately obtain
    \begin{equation*}
    \dLp{4}{u,v}{\Omega^3 \nablas\eta} \lesssim_I e^{\Delta_I (u_+-u)} \int_u^{u_+}\dLp{4}{u',v}{\Omega \nablas(\Omega\tr\chib) \, \Omega^2\eta+\Omega^3\nablas^2\omegab}\ud u'
  \end{equation*}
  because $\nablas\eta$ vanishes on $C_0$.
  Moreover by Lemma~\ref{lemma:torsion} and Lemma~\ref{lemma:ds:trchib} we can estimate
  \begin{equation*}
    \dLp{4}{u,v}{\Omega \nablas(\Omega\tr\chib) \, \Omega^2\eta}\leq \Linfty{\Omega^2\eta}\dLp{4}{u,v}{\Omega \nablas(\Omega\tr\chib)}\lesssim u_+-u
  \end{equation*}
  The second term is bounded by the assumption \eqref{BA:Omega:dd:omegab}.
\end{proof}

Similarly we derive from \eqref{eq:D:etab} and \CLemma{4.1, 4.2} that
\begin{equation}\label{eq:D:nablas:etab:abs}
  D|\nablas\etab|^2+3\Omega\tr\chi|\nablas\etab|^2\lesssim |\Omega\chih| |\nablas\etab|^2 + \Bigl( \bigl|\nablas(\Omega\chi)\bigr|_{L^\infty} |\etab|+|\nablas^2\omega|\Bigr)|\nablas\etab|
\end{equation}
and thus by Lemma~\ref{lemma:dLp:gronwall}:

\begin{lemma}
  \label{lemma:nablas:etab}
  Assume \emph{(\textbf{A:I,II})} and \eqref{D:small} hold. Then
  \begin{equation}
    \dLp{4}{u,v}{\Omega^3\nablas\etab}\lesssim \dLp{4}{u,v_+}{\Omega^3\nablas\etab}+\Linf{u,v_+}{\Omega\nablas(\Omega\chi)}\Linf{u,v_+}{\Omega^2\etab}
  \end{equation}
  provided the data on $\Cb_+$ is such that the r.h.s.~is finite.
\end{lemma}

\begin{proof}
From Prop~\ref{prop:L4:chi}  and Lemma~\ref{lemma:torsion}  we have
\begin{equation*}
  \dLp{4}{u,v'}{\Omega\nablas(\Omega\chi)\cdot \Omega^2\etab}\lesssim | \Omega\nablas(\Omega\chi) |_{L^\infty(S_{u,v_+})}| \Omega^2\etab|_{L^\infty(S_{u,v_+})}\,.
\end{equation*}
Thus we can proceed as in the previous proof, now multiplying \eqref{eq:D:nablas:etab:abs} by $\Omega^6$, and in view of \eqref{BA:Omega:dd:omega} the statement then follows from Lemma~\ref{lemma:dLp:gronwall}.
\end{proof}

\subsubsection{$\ds\omega$, $\ds\omegab$}

Recall \eqref{eq:Db:omega:trchi:simple} in the shear-free case $\chibh=0$:
\begin{equation}
  \begin{split}
    \Db\bigl(2\omega-\Omega\tr\chi\bigr)   &= -2\Omega^2 \bigl( K + \divs\eta+\lvert \eta\rvert^2 -2(\eta,\etab)+\lvert\etab\rvert^2 \bigr)\\
    &=-2\Omega^2 \bigl( -\breve{\mu} +\lvert \eta\rvert^2 -2(\eta,\etab)+\lvert\etab\rvert^2 \bigr)
  \end{split}
\end{equation}

Hence
\begin{equation}\label{eq:Db:ds:omega:tr:chi}
    \Db\ds\bigl(2\omega-\Omega\tr\chi\bigr) =2\ds(\Omega^2  \breve{\mu}) -4( \Omega \eta,\ds(\Omega\eta)) +4\ds(\Omega\eta,\Omega\etab)-2(\Omega\etab,\ds(\Omega\etab) \bigr)  
\end{equation}
and $\mathrm{L}^4$ bounds for $\ds(2\omega-\Omega\tr\chi)$ follow when corresponding bounds have been established for $\ds\breve{\mu}$:

\begin{lemma}
  Assume \emph{(\textbf{A:I,II,\underline{I},\underline{II}})} and \eqref{D:small} hold. If moreover
  \begin{equation}
    \dLp{4}{u,v}{\Omega\ds(\Omega^2\breve{\mu})}\lesssim 1
  \end{equation}
then
  \begin{equation}
    \dLp{4}{u,v}{\Omega\ds(2\omega-\Omega\tr\chi)}\lesssim u_+-u
  \end{equation}
  
\end{lemma}

\begin{proof}
  From the equation \eqref{eq:Db:ds:omega:tr:chi} we derive that
  \begin{multline*}
    \dLp{4}{u,v}{\Omega\ds(2\omega-\Omega\tr\chi)}\lesssim \int_u^{u_+} \dLp{4}{u',v}{\Omega\ds(\Omega^2\breve{\mu})}\\+\Bigl(\Linf{u',v}{\Omega\eta}+\Linf{u',v}{\Omega\etab}\Bigr)\Bigl(\dLp{4}{u',v}{\Omega\ds(\Omega\eta)}+\dLp{4}{u',v}{\Omega\ds(\Omega\etab)}\Bigr)\ud u'
  \end{multline*}
and so the statement follows using the results of Sections~\ref{sec:Linfty:eta}, and~\ref{sec:L4:nablas:eta}.
\end{proof}

\subsubsection{$\kappab$}
\label{sec:kappab}

An estimate for the Gauss curvature can also be obtained in the shear-free case by considering $\kappab$ in \eqref{eq:kappab},
which satisfies the propagation equation \eqref{eq:Db:kappab:shear-free}.
Also note that $\overline{\kappab}=\overline{K}$, and by the Gauss Bonnet theorem $\overline{K}=1/r^2$. Therefore
\begin{equation}
  \Db\overline{K}=-\overline{\Omega\tr\chib}\:\overline{K}
\end{equation}
and
\begin{equation}\label{eq:Db:kappab:average}
  \Db(\Omega^2\overline{\kappab})=(2\omegab-\overline{\Omega\tr\chib})\Omega^2\overline{\kappab}
\end{equation}
as well as,  in the shear-free case,
  \begin{multline}\label{eq:Db:kappab:difference}
    \Db\bigl(\Omega^2(\kappab-\overline{\kappab})\bigr)=\bigl(2\omegab-\Omega\tr\chib\bigr)\Omega^2(\kappab-\overline{\kappab})-\bigl(\Omega\tr\chib-\overline{\Omega\tr\chib}\bigr)\bigl(\frac{\Omega}{r}\bigr)^2\\-\Omega^3\tr\chib\bigl(|\eta|^2+\divs\eta\bigr)
  \end{multline}

\begin{lemma}\label{lemma:L4:kappab}
  Suppose \emph{(\textbf{A:\underline{I},\underline{II}})}  hold, and $\chibh=0$. Then for \eqref{D:small} on $\mathcal{D}_+$:
  \begin{subequations}
\begin{gather}
  \dLp{4}{u,v}{\Omega^2\kappab}+\Linf{u,v}{\Omega^2\overline{\kappab}} \lesssim_{I,II} 1\\
  \dLp{4}{u,v}{\Omega^2(\kappab-\overline{\kappab})}\lesssim_{I,II} 1/r
\end{gather}
\end{subequations}
\end{lemma}

\begin{proof}
  Since $|2\omegab-\Omega\tr\chib|\leq \Delta_I$ we can apply Lemma~\ref{lemma:dLp:gronwall} to \eqref{eq:Db:kappab:mu},
  \begin{equation*}
    \Db\bigl(\Omega^2\kappab\bigr)=\bigl(2\omegab-\Omega\tr\chib\bigr)\Omega^2\kappab-\Omega^3\tr\chib\bigl(|\eta|^2+\divs\eta\bigr)
  \end{equation*}
  and infer with Lemma~\ref{lemma:nablas:eta}, and the assumptions on the data,
  \begin{equation*}
    \dLp{4}{u,v}{\Omega^2\kappab} \lesssim_I e^{\Delta_I (u_+-u)} \Bigr[ \dLp{4}{u_+,v}{\Bigl(\frac{\Omega}{r}\Bigr)^2}+\int_u^{u_+}\Linf{u',v}{\Omega^2\eta}^2+\dLp{4}{u',v}{\Omega^3\nablas\eta}\ud u'\Bigr]\lesssim 1
  \end{equation*}

  The same bound in $L^4$ for $\overline{\kappab}$ can be derived from \eqref{eq:Db:kappab:average}, but since $\overline{\kappab}=1/r^2$ we can also appeal directly to Lemma~\ref{lemma:BA:0} for the stated $L^\infty$ estimate.
  
   Finally we apply Lemma~\ref{lemma:dLp:gronwall} to \eqref{eq:Db:kappab:difference} to show that
   \begin{multline*}
     \dLp{4}{u,v}{\Omega^2(\kappab-\overline{\kappab})} \lesssim_I  e^{\Delta_I (u_+-u)} \int_u^{u_+} \Linf{u',v}{\Omega\tr\chib-\overline{\Omega\tr\chib}}\Linf{u',v}{\frac{\Omega}{r}}^2\ud u'\\+e^{\Delta_I (u_+-u)}\int_u^{u_+}\Linf{u',v}{\Omega^2\eta}^2+ \dLp{4}{u',v}{\Omega^3\nablas\eta}\ud u' \lesssim 1/r
   \end{multline*}
   The crucial difference to the estimate for $\kappab$ is that here the initial data vanishes: $\kappab=\overline{\kappab}$ on $C_+$. For the rate we have used Lemma~\ref{lemma:comparison}.

\end{proof}


\section{Uniformization Theorem and Elliptic estimates}
\label{sec:uniformization}


\subsection{Isoperimetric constants}

Recall we prescribe initial data on $C_+$: \[\gs_{u_+,v}=r^2(u_+,v)\gammac\]

Define $\Phi_{u-u_+}:S_{u_+,v}\to S_{u,v}$ by the past directed geodesic flow generating $\Cb_v$,
\begin{equation}
  \gs(u)=\underline{\Phi}_{u-u_+}^\ast \gs\rvert_{S_{u,v}}
\end{equation}
and note that
\begin{equation}
  \frac{\partial}{\partial u}\gs(u)=2(\Omega\chib)(u)
\end{equation}

Denote by $\lambda(u)$, and $\Lambda(u)$, the smallest and largest eigenvalue of $\gs(u)$ with respect to $\gs(u_+)$, and by $I(S_{u,v})$ the isoperimetric constant of $S_{u,v}$.

\begin{lemma} \label{lemma:isoperimetric}
  Suppose \emph{(\textbf{A:\underline{I}})} holds. Then in the shear-free case,
  \begin{equation}
     \Lambda(u) \simeq \lambda(u) \simeq \Bigl(\frac{r(u_+,v)}{r(u,v)}\Bigr)^2
   \end{equation}
   and
   \begin{equation}
     I(S_{u,v})\lesssim I(S_{u_+,v})
   \end{equation}
   where $I(S_{u_+,v})=(2\pi)^{-1}$.
\end{lemma}
\begin{proof}
  As in \Ceq{5.81} and by Lemma~\ref{lemma:Omegatrchi:r}
  \begin{equation*}
    \sqrt{\Lambda(u)\lambda(u)}=\exp\Bigl[\int_u^{u_+}(\Omega\tr\chib)(u')\ud u'\Bigr]\simeq_I \Bigl(\frac{r(u_+,v)}{r(u,v)}\Bigr)^2
  \end{equation*}
  Also in the shear-free case by \Ceq{5.86}
  \begin{equation*}
    \sqrt{\Lambda(u)/\lambda(u)}=1
  \end{equation*}

  Let $U_0\subset S_{u_+,v}$ be a domain with $C^1$ boundary $\partial U_0$, and consider $U_u=\underline{\Phi}_{u-u_+}(U_0)\subset S_{u,v}$.
  Then
  \begin{gather*}
    \text{Perimeter}(\partial U_{u})\geq \inf \sqrt{ \lambda(u) } \text{Perimeter}(\partial U_0)\\
    \text{Area}(U_u)\leq \sup\sqrt{\Lambda(u)\lambda(u)}\text{Area}(U_0)
  \end{gather*}
hence
\begin{equation*}
  \frac{\text{Area}(U_{u})}{(\text{Perimeter}(\partial U_u))^2}\leq \sup \frac{ \sqrt{\Lambda(u)\lambda(u)}}{ \lambda(u) } \frac{\text{Area}(U_0)}{ ( \text{Perimeter}(\partial U_0) )^2}
\end{equation*}
and similarly for $U_0$ replaced by $U_0^c\subset S_{0,v}$. So the bound follows from the definition of $I(S_{u,v})$, c.f.~\Ceq{5.36}.
  \end{proof}

  \begin{corollary}\label{cor:sobolev:d}
    Under the same assumptions of Lemma~\ref{lemma:isoperimetric}, the following Sobolev inequalities hold for any $p>2$, for any tensorfield $\xi$ on all spheres $S=S_{u,v}$, uniformly in $(u,v)$, $u\leq u_+$:
    \begin{gather}
      \nLp{\xi}\lesssim_{p,I}\nLpq{\xi}{2}+\nLpq{r\nablas\xi}{2}\\
      \Linfty{\xi}\lesssim_{p,I} \nLp{\xi}+\nLp{r\nablas\xi}\label{eq:sobolev:d:Linf}
    \end{gather}

  \end{corollary}

  \begin{proof}
    See proofs of \CLemma{5.1, 5.2}.
  \end{proof}

  \subsection{Uniformization Theorem}

  As we have seen in Section~\ref{sec:prelim:conformal} -- see in particular Lemma~\ref{lemma:prelim:conformal} -- bounds for the conformal factor $\psi$ in \eqref{eq:conformal:factor} can be obtained directly from the propagation equations in the shear-free case.

  More generally, these bounds on the conformal factor can be obtained from the uniformization theorem, provided sufficient control on the Gauss curvature has been established. In the present setting, on a spacetime with vanishing Weyl curvature, it is immediately clear from the Gauss equation \eqref{eq:gauss:hat} that $K\in\mathrm{L}^\infty(S_{u,v})$. In fact, in the shear-free case a suitable $\mathrm{L}^\infty$ estimate can be derived from \eqref{eq:Db:trchitrchib}:
  \begin{equation}
    \Linf{u,v}{\Omega^2K}\lesssim \Linf{u_+,v}{\frac{\Omega}{r}}^2+\frac{1}{r}\sup_{u'}\Bigl(\Linf{u',v}{\Omega^3\mu}+\Linf{u',v}{\Omega^2\eta}^2\Bigr)
  \end{equation}
  This requires that $\Omega^3\mu=-\Omega^3\divs\eta$ is bounded in $L^\infty$.
  In this section we point out --- following \CCh{5.3}, and earlier work of Bieri \cite{B:09} --- that the relatively weak bounds of Section~\ref{sec:kappab} for the Gauss curvature in $L^4$ suffice to obtain the desired estimates on $\psi$.
  
\begin{remark}
  In \CCh{5.1} $K$ is estimated from the propagation equation for $K$, but then this is involves an estimate for $\nablas^2(\chi,\chib)$ which can only be obtained in $\mathrm{L}^2$.
\end{remark}
  
  Without assumptions on the mass aspect function,  we  know the following about the Gauss curvature on each sphere $S_{u,v}$:
  Given that $\kappab=K-\divs\etab$,  we have derived in Section~\ref{sec:kappab} in the shear-free case:
  \begin{subequations}
  \begin{gather}
    r^2 \dLp{4}{u,v}{K}\lesssim \dLp{4}{u,v}{\Omega^2(\kappab+\divs\etab)}\lesssim 1\\
    K-\overline{K}=\kappab-\overline{\kappab}+\divs\etab \\ 
    r^3 \dLp{4}{u,v}{K-\overline{K}}\lesssim r \dLp{4}{u,v}{\Omega^2(\kappab-\overline{\kappab})}+\dLp{4}{u,v}{\Omega^3\nablas\etab}\lesssim_{I,II} 1    \label{eq:KKbar:L4:bound}
  \end{gather}
\end{subequations}

The last bound \eqref{eq:KKbar:L4:bound} is the key estimate for the Lemmas proven in this section.

We can now apply the procedure in \CCh{5.3}  to prove that there exists $\varphi\in\mathrm{L}^\infty(S_{u,v})$ such that
  \begin{equation}
    K[e^{2\varphi}\gs]\in\mathrm{L}^\infty(S_{u,v})
  \end{equation}

  In fact, let $\varphi$ satisfy
  \begin{equation}\label{eq:varphi}
    \Laplaces \varphi=K-\overline{K}
  \end{equation}
  then the Gauss curvature of $\gs'=e^{2\varphi}\gs$ is given by
  \begin{equation}\label{eq:K:prime}
    K'=e^{-2\varphi}(K-\Laplaces\varphi)=e^{-2\varphi}\overline{K}
  \end{equation}
  and by the Gauss Bonnet theorem $\overline{K}=1/r^2$.

  \begin{lemma}\label{lemma:varphi}
    Let $\varphi$ be a solution to \eqref{eq:varphi} such that $\overline{\varphi}=0$.
    Suppose \emph{(\textbf{A:\underline{I},\underline{II}})} hold, and $\chibh=0$, and assume \eqref{D:small}.
    Then
    \begin{equation}
      \Linf{u,v}{\varphi}\lesssim 1/r\,,
    \end{equation}
    and $\dLp{p}{u,v}{r\nablas\varphi}\lesssim 1/r$ for any $p>2$.
  \end{lemma}

  \begin{proof}
    From the standard elliptic estimate for the Laplacian,
    \begin{equation*}
      \int_S |\nablas^2\varphi|^2+K|\nablas\varphi|^2=\int_S |K-\overline{K}|^2
    \end{equation*}
    it follows that
    \begin{equation*}
      \dLp{2}{u,v}{\nablas^2\varphi}^2+\overline{K}\dLp{2}{u,v}{\nablas\varphi}^2\leq \dLp{2}{u,v}{K-\overline{K}}^2+\dLp{2}{u,v}{K-\overline{K}}\dLp{4}{u,v}{\nablas\varphi}^2
    \end{equation*}
    and by Corollary~\ref{cor:sobolev:d}
    \begin{equation*}
      \dLp{4}{u,v}{\nablas\varphi}^2\lesssim_I\dLp{2}{u,v}{\nablas\varphi}^2+\dLp{2}{u,v}{r\nablas^2\varphi}^2
    \end{equation*}
    hence
    \begin{equation*}
      \Bigl(1-r^2\dLp{2}{u,v}{K-\overline{K}}\Bigr)\Bigl(\dLp{2}{u,v}{\nablas^2\varphi}^2+\frac{1}{r^2}\dLp{2}{u,v}{\nablas\varphi}^2\Bigr)\leq \dLp{2}{u,v}{K-\overline{K}}^2
    \end{equation*}
and by \eqref{eq:KKbar:L4:bound} we can assume $r^2\dLp{2}{u,v}{K-\overline{K}}\leq 1/2$, in view of \eqref{D:small}.
    Thus 
        \begin{gather*}
          \dLp{2}{u,v}{\nablas^2\varphi}^2+\frac{1}{r^2}\dLp{2}{u,v}{\nablas\varphi}^2\leq 2\dLp{2}{u,v}{K-\overline{K}}^2\leq 2\dLp{4}{u,v}{K-\overline{K}}^2\lesssim  1/r^6\\
          \dLp{p}{u,v}{r\nablas\varphi}^2\lesssim_I \dLp{2}{u,v}{r^2\nablas^2\varphi}^2+\dLp{2}{u,v}{r\nablas\varphi}^2 \lesssim 1/r^2
        \end{gather*}
        for any $p>2$, and again by Corollary~\ref{cor:sobolev:d}
        \begin{equation*}
                \dLp{p}{u,v}{\varphi}^2\lesssim_{I}\dLp{2}{u,v}{\varphi}^2+\dLp{2}{u,v}{r\nablas\varphi}^2\,.
        \end{equation*}

        Furthermore by Lemma~\ref{lemma:isoperimetric} the isoperimetric constant is bounded, $I(S_{u,v})\lesssim_I 1$, and in view of $\overline{\varphi}=0$ the isoperimetric inequality on $S_{u,v}$ then yields
        \begin{equation*}
          \int_{S_{u,v}}\varphi^2\dm{\gs}\leq I(S_{u,v})\Bigl(\int_{S_{u,v}}|\nablas\varphi|\dm{\gs}\Bigr)^2
        \end{equation*}
        or simply
        \begin{equation*}
          \dLp{2}{u,v}{\varphi}\leq \dLp{2}{u,v}{r\nablas\varphi}\,.
        \end{equation*}

        Inserting above we obtain
        \begin{equation*}
                \dLp{p}{u,v}{\varphi}^2\lesssim_{I}\dLp{2}{u,v}{r\nablas\varphi}^2
              \end{equation*}
              and thus by Corollary~\ref{cor:sobolev:d}
              \begin{equation*}
                \begin{split}
                  \Linf{u,v}{\varphi}&\lesssim_I  \dLp{p}{u,v}{\varphi}+\dLp{p}{u,v}{r\nablas\varphi}\\
                  &\lesssim_I \dLp{2}{u,v}{r\nablas\varphi}+\dLp{p}{u,v}{r\nablas\varphi}\lesssim 1/r\,.
                \end{split}
              \end{equation*}

     \end{proof}

     \begin{proposition}\label{prop:uniformization}
       Suppose \emph{(\textbf{A:\underline{I,II}})} holds.
       Consider the shear-free case $\chibh=0$, and assume \eqref{D:small}.
       Then there exists $\psi:\mathcal{D}_+\longrightarrow \mathbb{R}$  such that
       \begin{equation}
         \gs=r^2 e^{2\psi}\gammac
       \end{equation}
       where $\gammac$ is the standard metric on $\mathbb{S}^2$, $r$ is the area radius of $(S_{u,v},\gs)$,  and for any $p>2$
       \begin{gather}
         \Linf{u,v}{\psi}\lesssim 1\qquad \Linf{u,v}{r \ds\psi}\lesssim 1\\
         \dLp{p}{u,v}{r\ds\psi}\lesssim 1\qquad          \dLp{4}{u,v}{r^2\nablas^2\psi}\lesssim 1
       \end{gather}
     \end{proposition}

     \begin{remark}
       In fact, given spherically symmetric data on $C_+$ one has
       \begin{equation}
         \Linf{u,v}{\psi}\lesssim 1/r \qquad  \Linf{u,v}{r \ds\psi}\lesssim 1/r \qquad \dLp{4}{u,v}{r^2\nablas^2\psi}\lesssim 1/r
       \end{equation}
       For simplicity, in the proof given below, we have not exploited the factor $1/r$ in the estimates of Lemma~\ref{lemma:varphi}.
     \end{remark}

     \begin{proof}
       Denoting by $r$, and $r'$ the area radius of $(S_{u,v},\gs)$, and $(S_{u,v},\gs')$, respectively, it follows from Lemma~\ref{lemma:varphi} that $r'/r$ is bounded from above and below, and hence by \eqref{eq:K:prime} that
       \begin{equation*}
         (r')^2K'=e^{-2\varphi}\frac{{r'}^2}{r^2} \simeq_{I,II} 1
       \end{equation*}
       Therefore we can apply the uniformization theorem in form of \CProp{5.3} to $(S_{u,v},\gs)$ to infer that there exists a conformal factor $\Omega'$ such that the metric $\gsc={\Omega'}^2\gs'$ has Gauss curvature $K[\gsc]=1$, and the bounds for $\Omega'_m=\inf_{S_{u,v}}r'\Omega'$, $\Omega_M'=\sup_{S_{u,v}}r'\Omega'$, and $\Omega_1'=\sup_{S_{u,v}}{\Omega'}^{-2}|\ds\Omega'|_{\gs'}$ depend only on the upper bounds for ${k_m'}^{-1}$, and $k_M$, where $k_m'=\inf_{S_{u,v}}{r'}^2 K'$, $k_M'=\sup_{S_{u,v}}{r'}^2K'$, namely only on $\Delta_I$, and $\Delta_{II}$, while for every $p\geq 2$,
       \begin{equation*}
         \Omega_{2,p}'=\Bigl(\int_{S_{u,v}}{\Omega'}^{-3p+2}|{\nablas'}^2\Omega'|_{\gs'}^p\dm{\gs'}\Bigr)^\frac{1}{p}\lesssim_p 1
       \end{equation*}
       Then we can proceed as on page~166 of \cite{C:09}, and set $\Omega'={r'}^{-1}e^{-\psi'}$. The above bounds are then equivalent to bounds for
       \begin{equation*}
         \sup_{S_{u,v}}|\psi'|\,,\quad \sup_{S_{u,v}}r'|\ds\psi'|_{\gs'}\,, \quad \dLp{p}{u,v}{{r'}^2{\nablas'}^2\psi'}\,.
       \end{equation*}
       In conclusion we have
       \begin{equation*}
         \gsc={\Omega'}^2\gs'={r'}^{-2}e^{-2\psi'}e^{2\varphi}\gs=r^{-2}e^{-2\psi}\gs
       \end{equation*}
       where $\psi=\log(r'/r)+\psi'-\varphi$, and the $\mathrm{L}^\infty$ bound for $\psi$ follows from Lemma~\ref{lemma:varphi} and the above.
       Moreover in  Lemma~\ref{lemma:varphi} we have established an $\mathrm{L}^p$ bound for $\varphi$, which yields with the above that for any $p>2$
       \begin{equation*}
         \dLp{p}{u,v}{r\ds\psi}\leq         \Linf{u,v}{r\ds\psi'}+         \dLp{p}{u,v}{r\ds\varphi}\lesssim_{I,II} 1\,.
       \end{equation*}

       This can now be improved to an $\mathrm{L}^\infty$ estimate using \eqref{eq:KKbar:L4:bound}, following the discussion on pages 168-170 in \cite{C:09}. Indeed returning to \eqref{eq:varphi} we infer from the conformal invariance of the equation that
       \begin{equation*}
         \Laplacec\varphi=r^2e^{2\psi}(K-\overline{K})
       \end{equation*}
       and by the Calderon-Zygmund inequality for any $p>1$,
       \begin{equation*}
         \|\nablac^2\varphi\|_{\mathrm{L}^p(\mathbb{S}^2,\gsc)}+         \|\nablac\varphi\|_{\mathrm{L}^p(\mathbb{S}^2,\gsc)}\lesssim_p \| r^2 e^{2\psi}(K-\overline{K})\|_{\mathrm{L}^p(\mathbb{S}^2,\gsc)}
       \end{equation*}
       After rescaling this reads
       \begin{equation*}
         \|\!\!\!\!\!- r^2\nablac^2\varphi\|\!\!\!\!\!-_{\mathrm{L}^p(\mathbb{S}^2,\gs)}+          \|\!\!\!\!\!- r\nablac \varphi \|\!\!\!\!-_{\mathrm{L}^p(\mathbb{S}^2,\gs)} \lesssim_{p,\sup|\psi|}  r^2 \|\!\!\!\!\!- K-\overline{K} \|\!\!\!\!-_{\mathrm{L}^p(\mathbb{S}^2,\gs)}
       \end{equation*}
       Since
       \begin{equation*}
         r^2 \nLp{ \nablas^2\varphi-\nablac^2\varphi} \leq \nLpq{r\ds\psi}{2p}\nLpq{r\ds\varphi}{2p}
       \end{equation*}
       we obtain in view of the estimate for $\ds\psi$ above, the estimate of Lemma~\ref{lemma:varphi} for $\ds\varphi$, and the estimate \eqref{eq:KKbar:L4:bound} for $K-\overline{K}$ ,  that
       \begin{equation*}
         \dLp{4}{u,v}{r^2\nablas^2\varphi}\lesssim_{p,\sup|\psi|} r^2\dLp{4}{u,v}{K-\overline{K}}+ \dLp{8}{u,v}{r\ds\psi}\dLp{8}{u,v}{r\ds\varphi}\lesssim 1\,.
       \end{equation*}
       This yields the desired bound for $\nablas^2\psi$ in $\mathrm{L}^4$, and for $\ds\psi$ in $\mathrm{L}^\infty$.
     \end{proof}
     
  \subsection{Elliptic estimates}
\label{sec:elliptic}

\begin{lemma}
  \label{lemma:calderon:div}
  Consider the following equation for a trace-free symmetric $2$-covariant tensorfield $\theta$, and a $1$-from $f$ on $(S_{u,v},\gs)$:
  \begin{equation}
    \divs\theta=f
  \end{equation}
   Given the conclusions of Proposition~\ref{prop:uniformization}, we have
  \begin{gather}
    \dLp{4}{u,v}{\nablas\theta}+\frac{1}{r}\dLp{4}{u,v}{\theta}\lesssim \dLp{4}{u,v}{f}\\
    \dLp{4}{u,v}{\nablas^2\theta}\lesssim \dLp{4}{u,v}{\nablas f}+\frac{1}{r}\dLp{4}{u,v}{f}
  \end{gather}

\end{lemma}

\begin{proof}
  See \CCh{5.4}, in particular the proof before \CLemma{5.6}.
\end{proof}

  \begin{lemma}\label{lemma:calderon:div:curl}
    Given the conclusions of Proposition~\ref{prop:uniformization}, consider the following system of equations for a 1-form  $\theta$, and functions $f$, $g$  on $(S_{u,v},\gs)$:
    \begin{equation}\label{eq:system:div:curl:theta}
      \divs \theta= f\qquad
      \curls \theta= g
    \end{equation}
    Then for any $p\geq 2$,
        \begin{equation}\label{eq:calderon:theta}
      \dLp{p}{u,v}{\theta}+\dLp{p}{u,v}{r\nablas\theta}\lesssim  \dLp{p}{u,v}{rf}+ \dLp{p}{u,v}{rg}
    \end{equation}
    and
    \begin{equation}
      \label{eq:calderon:nabla:theta}
      \dLp{4}{u,v}{r^2 \nablas^2 \theta}\lesssim \sum_{i=0,1}\dLp{4}{u,v}{r^{1+i}\ds^if}+\dLp{4}{u,v}{r^{1+i}\ds^ig}\,.
    \end{equation}
  \end{lemma}

  \begin{proof}
    Since $\gs$ is conformal to $\gammac$, and the system is conformally invariant we have
    \begin{equation*}
      \stackrel{\circ}{\divs}\theta=r^2 e^{2\psi}f\qquad \stackrel{\circ}{\curl}{\theta}=r^2e^{2\psi}g
    \end{equation*}
    Moreover the norms transform according to 
    \begin{gather*}
      \dLp{p}{u,v}{r\theta}\lesssim_{\sup|\psi|}\|\theta\|_{\mathrm{L}^p(\mathrm{S}^2,\gammac)}\,,\quad      \dLp{p}{u,v}{r^2 \nablac\theta}\lesssim_{\sup|\psi|}\|\nablac\theta\|_{\mathrm{L}^p(\mathrm{S}^2,\gammac)}\,,\\      \dLp{p}{u,v}{r^3\nablac^2\theta}\lesssim_{\sup|\psi|}\|\nablac^2\theta\|_{\mathrm{L}^p(\mathrm{S}^2,\gammac)}\,,
    \end{gather*}
    and
    \begin{equation*}
      \| r^2 e^{2\psi} f \|_{\mathrm{L}^p(\mathbb{S}^2,\gammac)}\lesssim_{\sup|\psi|}r^2\dLp{p}{u,v}{f}
    \end{equation*}
    hence the Calderon-Zygmund inequality on $(\mathbb{S}^2,\gammac)$ implies \eqref{eq:calderon:theta},  once we control the difference
    \begin{equation*}
      \nablas\theta-\nablac\theta=\stackrel{\circ} {\Delta}\cdot \theta
    \end{equation*}
    where as in \Ceq{5.179} $\Linf{u,v}{\stackrel{\circ}{\Delta}}\lesssim \Linf{u,v}{\ds\psi}$, thus $\Linf{u,v}{r\stackrel{\circ}{\Delta}}$ is bounded.

    Furthermore, 
    \begin{equation*}
      \|\ds\bigl(r^2e^{2\psi}f\bigr)\|_{\mathrm{L}^p(\mathbb{S}^2,\gammac)}\lesssim_{\sup|\psi|}r^3\dLp{p}{u,v}{\ds f+2f\ds\psi}\lesssim r^3\dLp{p}{u,v}{ \ds f}+\Linf{u,v}{r\ds\psi}r^2\dLp{p}{u,v}{f}
    \end{equation*}
    hence again by the Calderon-Zygmund inequality, c.f.~\Ceq{5.222}, on $(\mathbb{S}^2,\gammac)$, we have
    \begin{multline*}
      \dLp{p}{u,v}{r \nablac^2\theta}\lesssim_{\sup|\psi|} \dLp{p}{u,v}{r \ds f}+\bigl(1+\Linf{u,v}{r\ds\psi}\bigr)\dLp{p}{u,v}{f}\\
      +\dLp{p}{u,v}{r \ds g}+\bigl(1+\Linf{u,v}{r\ds\psi}\bigr)\dLp{p}{u,v}{g}
    \end{multline*}
    which implies \eqref{eq:calderon:nabla:theta}, in view of the boundedness of $\sup|r\ds\psi|$ established in Prop.~\ref{prop:uniformization}, and
the difference $\nablas^2\theta-\nablac^2\theta$ 
is controlled as in \Ceq{5.231}:
\begin{multline*}
  \dLp{p}{u,v}{r^2 \nablas^2\theta-r^2 \nablac^2\theta}\lesssim \Linf{u,v}{r\ds\psi}\dLp{p}{u,v}{r\theta}\\+\Bigl(\dLp{p}{u,v}{r^2\nablas^2\psi}+\dLp{2p}{u,v}{r\ds\psi}^2\Bigr)\Linf{u,v}{\theta}
\end{multline*}
where the terms involving $\ds\psi$, and $\nablas^2\psi$ are bounded for $p=4$ by the results of Prop.~\ref{prop:uniformization}. Indeed, putting the above together yields
\begin{multline*}
  \dLp{p}{u,v}{r^2\nablas^2\theta}\lesssim_{(\sup|\psi|+\sup r|\ds\psi|)} \dLp{p}{u,v}{rf}+\dLp{p}{u,v}{rf}\\+\dLp{p}{u,v}{r^2\ds f}+\dLp{p}{u,v}{r^2\ds g}+\Linf{u,v}{\theta}
\end{multline*}
and so the final statement follows from the Sobolev inequality \eqref{eq:sobolev:d:Linf}, and \eqref{eq:calderon:theta}.
  \end{proof}

  \begin{corollary}\label{cor:calderon:Laplaces}
    Given the conclusions of Proposition~\ref{prop:uniformization}, consider the equation for functions $\omega$, and $f$ on $(S_{u,v},\gs)$:
    \begin{equation}
      \Laplaces \omega = f\,.
    \end{equation}
Then for any $p\geq 2$,
\begin{equation}
  \dLp{p}{u,v}{\nablas\omega}+\dLp{p}{u,v}{r\nablas^2\omega}\lesssim \dLp{p}{u,v}{rf}
\end{equation}
and
\begin{equation}
  \dLp{4}{u,v}{r^2\nablas^3\omega}\lesssim \sum_{i=0,1}\dLp{4}{u,v}{r^{1+i}\ds^if}\,.
\end{equation}

  \end{corollary}
  \begin{proof}
    Set $\theta=\nablas\omega$, then $\theta$ satisfies the system \eqref{eq:system:div:curl:theta} with $g=0$.
  \end{proof}

\subsection{Morrey's inequality}

In this section we prove a version on Morrey's inequality on the sphere.

\begin{lemma}\label{lemma:morrey}
Let $(S,\gs)$ be diffeomorphic to $\mathbb{S}^2$ with $\text{Area}(S)=4\pi r^2$, and conformal to $(\mathbb{S}^2,\gammac)$, $\gs=r^2e^{2\psi}\gammac$, where $\psi\in\mathrm{L}^\infty(S)$. Then for any $p>2$, given a differentiable function $u:S\to\mathbb{R}$, we have
\begin{equation}
      \Linfty{u-\overline{u}}\lesssim  \: \nLp{ r \nablas u}\,.
    \end{equation}
    where the constant only depends on $p$, and $\sup_S|\psi|$.
\end{lemma}

\begin{proof}
We adapt the proof from the euclidean setting in $\mathbb{R}^n$ using stereographic coordinates $(\vartheta^1,\vartheta^2)$.
Recall that in these coordinates, see e.g.\Ceq{1.209}, the metric takes the form
\begin{equation}
  \gs=r^2e^{2\psi}\gsc=\frac{r^2e^{2\psi}}{\Bigl(1+\frac{1}{4}\lvert\vartheta\rvert^2\Bigr)^2}\lvert\ud\vartheta\rvert^2
\end{equation}
and the volume element in ``stereographic polar coordinates'' $(|\vartheta|\cos\theta,|\vartheta|\sin\theta)$ is
\begin{equation}
  \dm{\gs}=\frac{r^2e^{2\psi}}{\Bigl(1+\frac{1}{4}\lvert\vartheta\rvert^2\Bigr)^2}\lvert\vartheta\rvert\ud\lvert\vartheta\rvert\ud\theta
\end{equation}
Let $y\in\mathbb{R}^2$, then
\begin{multline}
  u(y)-u(0)=\int_0^1\frac{\ud}{\ud s}u(sy)\ud s=\int_0^1(\nabla u(sy),y)\ud s=\int_0^1\gs(\nablas u,y)\ud s\\
  =\int_0^1\lvert\nablas u \rvert_{\gs}\lvert y\rvert_{\gs}\ud s\leq\int_0^1\lvert\nablas u\rvert(sy)\frac{r e^{\psi}}{1+\frac{1}{4}\lvert s y\rvert^2}\lvert y\rvert\ud s
\end{multline}
and therefore for any $\rho>0$,
\begin{equation}
  \int_{\partial B_\rho}\lvert u-u(0)\rvert=\int_0^{2\pi}\lvert u(\rho\xi(\theta))-u(0)\rvert\rho\ud\theta\leq\int_0^{2\pi}\int_0^\rho\lvert\nablas u\rvert(t\xi(\theta))\frac{r e^{\psi}}{1+\frac{1}{4}t^2}\ud t\ud\theta
\end{equation}
or
\begin{equation}
  \int_{\partial B_\rho}\lvert u-u(0)\rvert\leq\int_{B_\rho}\frac{1+\frac{1}{4}\lvert\vartheta\rvert^2}{re^\psi}\frac{\lvert\nablas u\rvert(\vartheta)}{\lvert\vartheta\rvert}\dm{\gs}(\vartheta)
\end{equation}
Therefore
\begin{multline}
  \int_{B_\rho}\lvert u-u(0)\rvert \dm{\gs} \leq\int_0^\rho\ud \sigma\frac{r^2e^{2\psi_M}}{\Bigl(1+\frac{1}{4}\sigma^2\Bigr)^2}\int_{\partial B_\sigma}\lvert u-u(0)\rvert\\
  \leq e^{2\psi_M-\psi_m}\int_0^\rho\ud \sigma\frac{r}{\Bigl(1+\frac{1}{4}\sigma^2\Bigr)^2}\int_{B_\sigma}\frac{1+\frac{1}{4}\lvert\vartheta\rvert^2}{\lvert\vartheta\rvert}\lvert\nablas u\rvert(\vartheta)\dm{\gs}(\vartheta)\\
  \lesssim_{\sup|\psi|}\int_0^\rho\ud \sigma \frac{r}{\Bigl(1+\frac{1}{4}\sigma^2\Bigr)^2}\Bigl(\int_{B_\sigma}\lvert\nablas u\rvert^p\dm{\gs}\Bigr)^\frac{1}{p}\Bigl(\int_0^\sigma\ud t\Bigl(\frac{1+\frac{1}{4}t^2}{t}\Bigr)^\frac{p}{p-1}\frac{2\pi r^2 t}{\bigl(1+\frac{1}{4}t^2\bigr)^2}\Bigr)^\frac{p-1}{p}
\end{multline}
Now consider the last integral: 
\begin{align}
  \int_0^\sigma\ud t\bigl(\frac{1+\frac{1}{4}t^2}{t}\bigr)^\frac{p}{p-1}\frac{2\pi r^2 t}{\bigl(1+\frac{1}{4}t^2\bigr)^2}&\lesssim 2\pi r^2\sigma^\frac{p-2}{p-1}<\infty\qquad (\sigma\ll 1)\\
  \int_0^\sigma\ud t\bigl(\frac{1+\frac{1}{4}t^2}{t}\bigr)^\frac{p}{p-1}\frac{2\pi r^2 t}{\bigl(1+\frac{1}{4}t^2\bigr)^2}&\lesssim 2\pi r^2\sigma^\frac{2-p}{p-1}\qquad (\sigma\gg 1)
\end{align}
and moreover
\begin{equation}
  \int_0^\rho\ud \sigma\frac{r}{\Bigl(1+\frac{1}{4}\sigma^2\Bigr)^2}\Bigl( 2\pi r^2\sigma^\frac{2-p}{p-1}\Bigr)^\frac{p-1}{p}\lesssim r^{3-\frac{2}{p}}\int_0^\rho\ud \sigma\frac{\sigma^{\frac{2}{p}-1}}{\Bigl(1+\frac{1}{4}\sigma^2\Bigr)^2}\lesssim r^{3-\frac{2}{p}}
\end{equation}
Thus
\begin{equation}
  \frac{1}{r^2}\int_{B_\rho}\lvert u-u(0)\rvert\dm{\gs} \lesssim_{\sup|\psi|} r \Bigl(\frac{1}{r^2}\int_{B_\rho}\lvert\nablas u\rvert^p\dm{\gs}\Bigr)^\frac{1}{p} = r \nLp{\nablas u}
\end{equation}
As an immediate consequence, 
\begin{multline}
  \lvert u(0)\rvert\leq \frac{1}{4\pi r^2}\int_S\lvert u(0)-u(x)\rvert\dm{\gs}(x)+\frac{1}{4\pi r^2}\int_S \lvert u(x)\rvert\dm{\gs}(x)\\
   \lesssim_{\sup|\psi|} r\nLp{\nablas u}+\nLp{u}
 \end{multline}
 which confirms the Sobolev inequality \eqref{eq:sobolev:d:Linf}.

Now let $x, y\in\mathbb{R}^2$, then
\begin{equation}
  u(x)-u(y)\leq \frac{1}{4\pi r^2}\int_S \lvert u(x)-u(z)\rvert\dm{\gs}(z)+\frac{1}{4\pi r^2}\int_S \lvert u(z)-u(y)\rvert\dm{\gs}(z)\lesssim C_\Psi r \nLp{\nablas u}
\end{equation}
because for the calculation of each integral we can choose the ``south pole'' of the stereographic projection to be at $x=0$, or $y=0$ respectively.
The statement of the Lemma is a special case because by the mean value theorem the average of $u$ is achieved at least at one point $\vartheta_\ast\in S$ (after possibly redefining $u$ on a set of  measure zero).
\end{proof}

\begin{corollary}
  Suppose \emph{(\textbf{A:\underline{I,II}})}  hold, and $\chibh=0$. Then under the assumption \eqref{D:small},
  \begin{equation}
    |\Omega\tr\chib-\overline{\Omega\tr\chib}| \leq \frac{1}{2}\Delta_I\,.
  \end{equation}

  Moreover, with $\Delta_I$  chosen sufficiently large in comparison to $\sup_{u}\Linf{u,v_+}{\Omega\nablas(\Omega\chi)}$, we have
  \begin{equation}
    |\Omega\tr\chi-\overline{\Omega\tr\chi}| \leq \frac{1}{2}\Delta_I\,.
  \end{equation}

\end{corollary}

\begin{proof}
  In Proposition~\ref{prop:uniformization} we have established that $\psi\in\mathrm{L}^\infty(S)$ for $(u,v)\in\mathcal{D}_+$,
  so we can use Lemma~\ref{lemma:ds:trchib} to conclude that 
  \begin{equation*}
    |\Omega\tr\chib-\overline{\Omega\tr\chib}|\lesssim \dLp{4}{u,v}{\Omega\ds(\Omega\tr\chib)}\lesssim C(\sup_{\mathcal{D}_+}|\psi|)\, ( u_+-u )
  \end{equation*}
  and thus the statement follows by \eqref{D:small}.
  
  Similarly, by Proposition~\ref{prop:L4:chi},
  \begin{equation*}
      |\Omega\tr\chi-\overline{\Omega\tr\chi}| \lesssim \dLp{4}{u,v}{\Omega\ds(\Omega\tr\chi)} \lesssim \Linf{u,v_+}{\Omega\ds(\Omega\tr\chi)}+(v_+-v)
  \end{equation*}
  which implies the stated bound for $\Delta_I$ chosen sufficiently large in comparsion to the data for $\ds(\Omega\tr\chi)$ on $\Cb_+$, and $v_+-v$ sufficiently small by \eqref{D:small}.
\end{proof}

\subsection{Comparison estimates}

In the elliptic estimates derived in this section the factor $r$ appears naturally as a consequence of the conformal rescaling of the equations. However, the propagation equations in double null gauge typically involve the factor $\Omega$.
It is important that in the setting of our bootstrap assumptions the two factors are essentially interchangable.

\begin{corollary} \label{cor:commute:Omega:r}
  Assume
    \begin{equation}
    \dLp{4}{u,v}{\Omega\ds\log\Omega}+\dLp{4}{u,v}{\Omega^2 \nablas^2 \log\Omega }\lesssim 1/r^2
  \end{equation}
Then
  \begin{align}
      \dLp{4}{u,v}{\nablas(\Omega f)}\lesssim& r \dLp{4}{u,v}{\nablas f}+\frac{1}{r^2}\dLp{4}{u,v}{f}\\
    \dLp{4}{u,v}{r \nablas f}\lesssim& \dLp{4}{u,v}{\nablas (\Omega f)}+\frac{1}{r^2}\dLp{4}{u,v}{f}\\
    \dLp{4}{u,v}{\nablas^2(\Omega f)}\lesssim& r \dLp{4}{u,v}{\nablas^2 f}+\frac{1}{r^2}\dLp{4}{u,v}{\nablas f}+\frac{1}{r^3}\dLp{4}{u,v}{f}\\
    r\dLp{4}{u,v}{\nablas^2 f}\lesssim&  \dLp{4}{u,v}{\nablas^2 (\Omega f) }+\frac{1}{r^2}\dLp{4}{u,v}{\nablas f}+\frac{1}{r^3}\dLp{4}{u,v}{f}
  \end{align}

\end{corollary}

\begin{proof}
  Note that
  \begin{align}
    \nablas(\Omega f)=& \Omega \nablas f+(\Omega \nablas \log \Omega) f\\
    \nablas^2(\Omega f) =&  2\Omega \nablas\log\Omega\: \nablas f+\Omega\nablas^2f\notag\\
    &+\bigl(\Omega\nablas\log\Omega\bigr)^2 \Omega^{-1}f+\Omega^2\nablas^2\log\Omega\: \Omega^{-1}f
  \end{align}
and by Cor.~\ref{cor:sobolev:d}
\begin{equation}
  \Linf{u,v}{\Omega\nablas \log \Omega} \lesssim 1/r^2
\end{equation}
Recall also \eqref{BA:log:Omega:r} for the interchange of $\Omega$ and $r$ under the integral.
\end{proof}

We prove the assumption made here on $\log\Omega$ in Lemma~\ref{lemma:dd:log:Omega}, and show its validity under the bootstrap assumptions and appropriate data on $\Cb_+$.


\section{Coupled systems}
\label{sec:coupled}


\subsection{Null second fundamental form}
\label{sec:coupled:codazzi}

Recall the Codazzi equation \eqref{eq:codazzi:chih} which in the present setting reduces to
\begin{equation}
  \divs(\Omega\chih)=\frac{1}{2}\ds(\Omega\tr\chi)+\Omega\chih^\sharp\cdot\etab-\frac{1}{2}\Omega\tr\chi\,\etab
\end{equation}
and from \eqref{eq:D:tr:chi} we derive
\begin{equation}\label{eq:D:ds:Omega:tr:chi}
  D\ds(\Omega\tr\chi)=\bigl(2\omega-\Omega\tr\chi\bigr)\ds(\Omega\tr\chi)+2\Omega\tr\chi\ds\omega-2(\Omega\chih,\nablas(\Omega\chih))\,.
\end{equation}

This is a coupled system for $(\nablas\Omega\chih, \ds\Omega\tr\chi)$.

\begin{proposition}\label{prop:coupled:chi}
  Assume \emph{(\textbf{A:I})} and \eqref{BA:Omega:ds:omega} hold, and suppose
\begin{equation}
  D_1:=\sup_u\dLp{4}{u,v_+}{\Omega\ds(\Omega\tr\chi)}+\sup_u\Linf{u,v_+}{\Omega\chih}+\sup_u\Linf{u,v_+}{\Omega^2\etab}<\infty
\end{equation}
    Then with \eqref{D:small},
    \begin{equation}
      \dLp{4}{u,v}{\Omega\ds(\Omega\tr\chi)}+\dLp{4}{u,v}{\Omega\nablas(\Omega\chih)}\lesssim  D_1
    \end{equation}

If in addition \eqref{BA:Omega:dd:omega} holds, and
\begin{equation}
  D_2:=\dLp{4}{u,v_+}{(\Omega\nablas)^2(\Omega\tr\chi)}<\infty
\end{equation}
then with \eqref{D:small},
\begin{equation}
  \dLp{4}{u,v}{(\Omega\nablas)^2(\Omega\tr\chi)}+  \dLp{4}{u,v}{(\Omega\nablas)^2(\Omega\chih)} \leq C(D_1,D_2)\,.
\end{equation}

  \end{proposition}

  \begin{proof}
  From the elliptic estimate of Lemma~\ref{lemma:calderon:div} applied to \eqref{eq:codazzi:chih} we have
  \begin{multline*}
    \dLp{4}{u,v}{\Omega\chih}+r\dLp{4}{u,v}{\nablas(\Omega\chih)}\lesssim r \dLp{4}{u,v}{\ds(\Omega\tr\chi)}\\
    +r\dLp{4}{u,v}{\Omega\chih^\sharp\cdot\etab}+r\dLp{4}{u,v}{\Omega\tr\chi\,\etab}\,.
  \end{multline*}

  Now from \eqref{eq:D:ds:Omega:tr:chi}, with \CLemma{4.2} applied to the 1-form $\ds(\Omega\tr\chi)$:
\begin{multline*}
  D\lvert\ds\Omega\tr\chi\rvert^2+\Omega\tr\chi|\ds\Omega\tr\chi|^2=2\bigl(\ds\Omega\tr\chi,D\ds\Omega\tr\chi\bigr)-2\Omega\chih(\ds\Omega\tr\chi,\ds\Omega\tr\chi)\\
  =2\bigl(2\omega-\Omega\tr\chi\bigr)|\ds(\Omega\tr\chi)|^2+2\bigl(\ds(\Omega\tr\chi),2\Omega\tr\chi\ds\omega-2(\Omega\chih,\nablas(\Omega\chih))\bigr)\\-2\Omega\chih(\ds\Omega\tr\chi,\ds\Omega\tr\chi)
\end{multline*}
or
\begin{multline*}
  D\lvert\Omega\ds(\Omega\tr\chi)\rvert^2\leq \Bigl(3\bigl|2\omega-\Omega\tr\chi\bigr|+|\Omega\chih|\Bigr)|\Omega\ds(\Omega\tr\chi)|^2\\
  +2|\Omega\ds(\Omega\tr\chi)|\Bigl(2\Omega\tr\chi \Omega|\ds\omega|+2|\Omega\chih||\Omega\nablas(\Omega\chih)|\Bigr)
\end{multline*}
  and thus by Lemma~\ref{lemma:dLp:gronwall},
    \begin{multline*}
    \dLp{4}{u,v}{\Omega\ds(\Omega\tr\chi)} \lesssim  \dLp{4}{u,v_+}{\Omega\ds(\Omega\tr\chi)}\\+\int_v^{v_+}\Linfty{\tr\chi}\dLp{4}{u,v'}{\Omega^2\ds\omega}+\Linfty{\Omega\chih}\dLp{4}{u,v'}{\Omega\nablas(\Omega\chih)}\ud v'\,.
  \end{multline*}
In view of the $\mathrm{L}^\infty$-bound on $\Omega\chih$ from Lemma~\ref{lemma:Linfty:chi}, and for $\Omega^2\etab$ from Lemma~\ref{lemma:torsion}, we insert the elliptic estimate from above, and  conlude with Gronwall's inequality that
\begin{equation*}
  \dLp{4}{u,v}{\Omega\ds(\Omega\tr\chi)}\lesssim \dLp{4}{u,v_+}{\Omega\ds(\Omega\tr\chi)}+(v_+-v)
\end{equation*}
and inserting the result again in the elliptic estimate we obtain the desired bound on $\dLp{4}{u,v}{\Omega\nablas(\Omega\chih)}$.

Furthermore, from Lemma~\ref{lemma:calderon:div} we also obtain
\begin{multline*}
      r^2\dLp{4}{u,v}{\nablas^2(\Omega\chih)}\lesssim r^2\dLp{4}{u,v}{\nablas^2(\Omega\tr\chi)}\\+\Linfty{\Omega^2\etab}\dLp{4}{u,v}{\nablas (\Omega\chih)}+\Linfty{\Omega\chih}\dLp{4}{u,v}{\Omega^2 \nablas \etab}\\+\Linfty{\Omega^2\etab}\dLp{4}{u,v}{\ds (\Omega\tr\chi)}+\dLp{4}{u,v}{\Omega^3\nablas\etab}\\+\dLp{4}{u,v}{\Omega\ds(\Omega\tr\chi)}+\Linfty{\Omega\chih}\Linfty{\Omega\etab}+\Linfty{\Omega^2\etab}
\end{multline*}
and it remains to derive the propgation equation for $\nablas^2(\Omega\tr\chi)$.

From \CLemma{4.1} applied to \eqref{eq:D:ds:Omega:tr:chi}, we have
\begin{multline*}
  D\nablas^2(\Omega\tr\chi)=\nablas\bigl(2\omega-\Omega\tr\chi\bigr)\cdot\nablas(\Omega\tr\chi)+\bigl(2\omega-\Omega\tr\chi\bigr)\nablas^2(\Omega\tr\chi)\\+2\nablas(\Omega\tr\chi)\cdot\nablas\omega+2\Omega\tr\chi\nablas^2\omega-2(\nablas(\Omega\chih),\nablas(\Omega\chih))-2(\Omega\chih,\nablas^2(\Omega\chih))-\nablas(\Omega\chi)\cdot\nablas(\Omega\tr\chi)\,,
\end{multline*}
which then yields with \CLemma{4.2} that
\begin{multline*}
  D|\Omega^2\nablas^2(\Omega\tr\chi)|^2\leq 2\Bigl(2\bigl|2\omega-\Omega\tr\chi\bigr|+|\Omega\chih|\Bigr)|\Omega^2\nablas^2(\Omega\tr\chi)|^2\\
  +2\bigl|\Omega^2\nablas^2(\Omega\tr\chi)\bigr|\Bigl(\bigl(\Omega|\nablas(2\omega)|+\Omega|\nablas(\Omega\tr\chi)|+\Omega|\nablas(\Omega\chih)|\bigr)|\Omega\nablas(\Omega\tr\chi)|+2|\Omega\nablas(\Omega\chih)|^2\\
  +2\Omega|\nablas(\Omega\tr\chi)||\Omega\nablas\omega|+2\tr\chi\Omega^3|\nablas^2\omega|
  +2\Omega|\chih|\Omega^2|\nablas^2(\Omega\chih)|\Bigr)
\end{multline*}
Thus, by  Lemma~\ref{lemma:dLp:gronwall},
  \begin{multline*}
    \dLp{p}{u,v}{\Omega^2\nablas^2(\Omega\tr\chi)} \lesssim \dLp{4}{u,v_+}{\Omega^2\nablas^2(\Omega\tr\chi)}\\+\int_v^{v_+}\bigl(\Linf{u,v'}{\Omega\nablas(\Omega\tr\chi)}+\Linf{u,v'}{\Omega\nablas(\Omega\chih)}\bigr)\dLp{p}{u,v'}{\Omega\nablas(\Omega\tr\chi)}\ud v'\\
    +\int_v^{v_+}\Linf{u,v'}{\Omega\nablas\omega}\dLp{4}{u,v'}{\Omega\nablas(\Omega\tr\chi)}+\dLp{p}{u,v'}{\Omega^3\nablas^2\omega}\ud v'\\
    +\int_v^{v_+}\Linf{u,v'}{\Omega\nablas(\Omega\chih)}\dLp{4}{u,v'}{\Omega\nablas(\Omega\chih)}\ud v'
    +\int_v^{v_+}\Linf{u,v'}{\Omega\chih}\dLp{4}{u,v'}{\Omega^2\nablas^2(\Omega\chih)}\ud v'
  \end{multline*}
and for the last term we can substitute from the elliptic estimate above, to conclude that $\Omega^2\nablas^2(\Omega\tr\chi)$ is bounded in $\mathrm{L}^4(S_{u,v})$ by Gronwall's inequality.
Here we used that $\Omega\nablas(\Omega\tr\chi)$, and $\Omega\nablas(\Omega\chih)$ are bounded in $\mathrm{L}^\infty$ by Prop.~\ref{prop:L4:chi}, and $\Omega\nablas\omega$ and $\Omega^3\nablas\omega$ are bounded by assumptions \eqref{BA:Omega:ds:omega}, \eqref{BA:Omega:ds:omega} in $\mathrm{L}^\infty$ and $\mathrm{L}^4$, respectively. Inserting the resulting estimate in the above elliptic estimate then also yields the stated bound on $\Omega^2\nablas^2(\Omega\chi)$ in $\mathrm{L}^4$.
Finally also note that
\begin{equation*}
 \dLp{4}{u,v}{(\Omega\nablas)^2(\Omega\tr\chi)}\lesssim   r^2\dLp{4}{u,v}{\nablas^2(\Omega\tr\chi)}+\Linf{u,v}{\Omega(\eta+\etab)}\dLp{4}{u,v}{\Omega\ds(\Omega\tr\chi)}\,.
\end{equation*}

\end{proof}







\subsection{Torsion and decoupled mass aspect function}
\label{sec:coupled:mu:tilde}

Consider \eqref{eq:mu:breve}, namely
\begin{equation}
  \divs\eta=-\breve{\mu}+\frac{1}{4}\tr\chi\tr\chib-1
\end{equation}
coupled to the propagation equation for $\breve{\mu}$ derived in Lemma~\ref{lemma:D:mu:breve}.

In fact, consider
\begin{equation}
  \divs\eta=-\tilde{\mu}+\frac{1}{4}\bigl(\tr\chi\tr\chib-\overline{\tr\chi\tr\chib}\bigr)
\end{equation}
where $\tilde{\mu}$ is given by \eqref{eq:tilde:mu}, coupled to \eqref{eq:D:tilde:mu}, namely
\begin{subequations}
  \begin{gather}
    D\tilde{\mu}+\Omega\tr\chi\tilde{\mu}=\tilde{r}+\tilde{m}\label{eq:D:mu:tilde}\\
    \tilde{r}:=\breve{r}-\overline{\breve{r}}\,,\quad \breve{r}:=-\frac{1}{2}\Omega\tr\chib|\chih|^2\label{eq:tilde:r}\\
    \tilde{m}:=\breve{m}-\overline{\breve{m}}\,,\quad \breve{m}=2\divs(\Omega\chih)\cdot (\eta-\etab)+2(\Omega\chih,\nablas\eta+\etab\otimes\etab)\label{eq:tilde:m}
  \end{gather}
\end{subequations}

Moreover,
\begin{subequations}
\begin{align}\label{eq:D:ds:mu:tilde}
    D\ds\bigl(\Omega^2\tilde{\mu}\bigr)=&\ds\bigl(2\omega-\Omega\tr\chi\bigr)\Omega^2\tilde{\mu}+\bigl(2\omega-\Omega\tr\chi\bigr)\ds\bigl(\Omega^2\tilde{\mu}\bigr)+\ds\bigl(\Omega^2\tilde{r}\bigr)+\ds\bigl(\Omega^2\tilde{m}\bigr)\\
                                        \ds\bigl(\Omega^2\breve{r}\bigr)=&-\frac{1}{2}\ds\bigl(\Omega\tr\chib\bigr)\Omega^2|\chih|^2-\frac{1}{2}\Omega\tr\chib\bigl(\Omega\chih,\nablas\Omega\chih\bigr)\displaybreak[0]\\
  \ds\bigl(\Omega^2\breve{m}\bigr)=&2\nablas\bigl(\Omega\divs(\Omega\chih)\bigr)\cdot \Omega(\eta-\etab)+2\Omega\divs(\Omega\chih)\cdot \nablas(\Omega\eta-\Omega\etab)\notag\\&+2(\nablas\Omega\chih,\Omega^2\nablas\eta+\Omega\etab\otimes\Omega\etab)+2\Bigl(\Omega\chih,\nablas\bigl(\Omega^2\nablas\eta\bigr)+\nablas(\Omega\etab)\otimes\Omega\etab+\Omega\etab\otimes\nablas(\Omega\etab)\Bigr)\label{eq:ds:m:breve}
\end{align}
\end{subequations}





\begin{proposition}\label{prop:coupled:eta:mu:breve}
Assume   $r^3 \dLp{4}{u,v_+}{\tilde{\mu}}\lesssim 1$.
  
  Suppose
  \begin{equation}\label{eq:coupled:eta:trchitrchib}
    r^3 \dLp{4}{u,v}{\tr\chi\tr\chib-\overline{\tr\chi\tr\chib}}\lesssim 1
  \end{equation}
  
  Then
  \begin{equation}\label{eq:coupled:nabla:eta}
    r^2\dLp{4}{u,v}{\eta}+r^3\dLp{4}{u,v}{\nablas\eta}+r^3\dLp{4}{u,v}{\tilde{\mu}}\lesssim 1
  \end{equation}

Moreover assume  $r^4\dLp{4}{u,v_+}{\ds\tilde{\mu}}\lesssim 1$.

  If also \eqref{BA:Omega:ds:omega}, \eqref{BA:Omega:dd:omega}, holds and
\begin{gather}
  r^4\dLp{4}{u,v}{\ds\bigl(\tr\chi\tr\chib\bigr)}\lesssim 1\\
  r\dLp{4}{u,v}{\ds(2\omegab-\Omega\tr\chi)}\lesssim 1
\end{gather}
then
\begin{equation}
    r^4\dLp{4}{u,v}{\nablas^2\eta}+r^4\dLp{4}{u,v}{\ds\tilde{\mu}}\lesssim 1
\end{equation}

In particular, under all above assumptions, which then imply
\begin{equation}
  \Linf{u,v}{\tr\chi\tr\chib-\overline{\tr\chi\tr\chib}}\lesssim 1/r^3
\end{equation}
we have
\begin{equation}
  r^3\Linf{u,v}{\nablas\eta}+r^3\Linf{u,v}{\tilde{\mu}}\lesssim 1
\end{equation}

\end{proposition}

\begin{proof}
  From Lemma.~\ref{lemma:calderon:div:curl} we have
  \begin{equation}\label{eq:nabla:eta:mu:tilde}
    r^3\dLp{4}{u,v}{\nablas\eta}\lesssim r^3\dLp{4}{u,v}{\tilde{\mu}}+1
  \end{equation}
and from \eqref{eq:D:mu:tilde} 
\begin{equation}
  \dLp{4}{u,v}{\Omega^2\tilde{\mu}}\lesssim   \dLp{4}{u,v_+}{\Omega^2\tilde{\mu}} + \int_v^{v_+}\dLp{4}{u,v'}{\Omega^2(\breve{r}-\overline{\breve{r}})}+\dLp{4}{u,v'}{\Omega^2(\breve{m}-\overline{\breve{m}})}\ud v'
\end{equation}

Recall here and in the following Cor.~\ref{cor:commute:Omega:r} for the commutation of $\Omega$ with $r$.

Now in view of \eqref{eq:tilde:r},
\begin{equation} 
  \dLp{4}{u,v}{\Omega^2\tilde{r}}\lesssim   \Linf{u,v}{\tr\chib}r^3\Linf{u,v}{ \chih }^2\lesssim 1
\end{equation}
and by Prop.~\ref{prop:coupled:chi}, and \eqref{eq:nabla:eta:mu:tilde},
\begin{equation}
  \begin{split}
  \dLp{4}{u,v}{\Omega^2\breve{m}}\lesssim& \bigl(\Linf{u,v}{\Omega\eta}+\Linf{u,v'}{\Omega\eta}\bigr)\dLp{4}{u,v}{\Omega\nablas(\Omega\chih)}\\
  &+\Linf{u,v}{\Omega\chih}\bigl(\dLp{4}{u,v}{\Omega^2\nablas\eta}+\Linf{u,v}{\Omega\etab}^2\bigr)\\
  \lesssim&\dLp{4}{u,v}{\Omega^2\tilde{\mu}}+1\\
  \dLp{4}{u,v}{\Omega^2\overline{\breve{m}}}\lesssim& r^2 \dLp{4}{u,v}{\breve{m}}\lesssim r^2\dLp{4}{u,v}{\tilde{\mu}}+1
\end{split}
\end{equation}


Note it follows from Gronwall's inequality that
\begin{equation}
  \dLp{4}{u,v}{\Omega^2\tilde{\mu}}\lesssim v_+-v
\end{equation}

Moreover again by Lemma~\ref{lemma:calderon:div:curl},
\begin{equation*}
  \dLp{4}{u,v}{r^4\nablas^2\eta}\lesssim r^3\dLp{4}{u,v}{\tilde{\mu}}+r^4\dLp{4}{u,v}{\ds\tilde{\mu}}+1
\end{equation*}
and from \eqref{eq:D:ds:mu:tilde}, and Lemma~\ref{lemma:dLp:gronwall},
\begin{equation*}
  \dLp{4}{u,v}{\Omega\ds(\Omega^2\tilde{\mu})}\lesssim   \dLp{4}{u,v_+}{\Omega\ds(\Omega^2\tilde{\mu})}+\int_v^{v_+}\dLp{4}{u,v'}{\tilde{\rho}}\ud v'
\end{equation*}
where
\begin{equation}
  \tilde{\rho}=\Linfty{\Omega^2\tilde{\mu}}|\Omega\ds\bigl(2\omega-\Omega\tr\chi\bigr)|
+|\Omega\ds\bigl(\Omega^2\tilde{r}\bigr)|+|\Omega\ds\bigl(\Omega^2\tilde{m}\bigr)|\,.
\end{equation}
For the first term we apply Cor.~\ref{cor:sobolev:d}, and the assumption to bound
\begin{equation}
  \Linf{u,v}{\Omega^2\tilde{\mu}}\dLp{4}{u,v}{\Omega\ds\bigl(2\omega-\Omega\tr\chi\bigr)}\lesssim 1+\dLp{4}{u,v}{r\ds\bigl(\Omega^2\tilde{\mu}\bigr)}
\end{equation}
For the second term we apply Prop.~\ref{prop:coupled:chi},
\begin{multline}
  \dLp{4}{u,v}{\Omega\ds\bigl(\Omega^2\breve{r}\bigr)}\lesssim \\ \lesssim \dLp{4}{u,v}{r \ds\bigl(\Omega\tr\chib\bigr) } \Linf{u,v}{\Omega\chih}^2+\tr\chib\Linf{u,v}{\Omega^2\chih}\dLp{4}{u,v}{r\nablas\Omega\chih}\lesssim 1
\end{multline}

We now turn to the contributions from the terms $\Omega\ds\bigl(\Omega^2\ds\tilde{m}\bigr)$ in \eqref{eq:ds:m:breve}.

First, by Prop.~\ref{prop:coupled:chi},
\begin{equation}
  \dLp{4}{u,v}{\Omega\nablas\bigl(\Omega\divs(\Omega\chih)\bigr)\cdot \Omega(\eta-\etab)}\lesssim \Bigl(\Linfty{\Omega\eta}+\Linfty{\Omega\eta}\Bigr)\dLp{4}{u,v}{\bigl(r\nablas\bigr)^2(\Omega\chih)}\lesssim 1
\end{equation}

Second, again by Prop.~\ref{prop:coupled:chi}, and \eqref{eq:coupled:nabla:eta}, and Lemma~\ref{lemma:nablas:etab},
\begin{equation}
  \begin{split}
    \dLp{4}{u,v}{\Omega\divs(\Omega\chih)\cdot \Omega\nablas(\Omega\eta-\Omega\etab)} \lesssim& \Linf{u,v}{\Omega\nablas(\Omega\chih)}\Bigl(\dLp{4}{u,v}{ r^2\nablas\eta}+\dLp{4}{u,v}{ r^2\nablas\etab } \Bigr)\\
    \lesssim & \sum_{k=0}^1 \dLp{4}{u,v}{(r\nablas)^k(\Omega\chih)} \lesssim 1
  \end{split}
\end{equation}

Third, similar by Prop.~\ref{prop:coupled:chi},
\begin{equation}
\dLp{4}{u,v}{\Omega(\nablas\Omega\chih,\Omega^2\nablas\eta+\Omega\etab\otimes\Omega\etab)}\lesssim \Linf{u,v}{\Omega\nablas(\Omega\chih)}\lesssim 1
\end{equation}

Note that in these steps, on one hand we use  bounds on $(\Omega\nablas)^2(\Omega\chih)$ in $\mathrm{L}^4$ that we have established with the help of the coupled system for $\chi$ in Section~\ref{sec:coupled:codazzi}. On the other, we use bounds on $\Omega\nablas(\Omega\etab)$ in $\mathrm{L}^4$ that have been established with the help of the propagation equations in Section~\ref{sec:L4:nablas:eta}.

Lastly,
\begin{multline}
\dLp{4}{u,v}{\Omega\Bigl(\Omega\chih,\nablas\bigl(\Omega^2\nablas\eta\bigr)+\nablas(\Omega\etab)\otimes\Omega\etab+\Omega\etab\otimes\nablas(\Omega\etab)\Bigr)}\\
\lesssim \Linf{u,v}{\Omega\chih}\dLp{4}{u,v}{\Omega\nablas\bigl(\Omega^2\nablas\eta\bigr)}\\+\Linf{u,v}{\Omega\chih}\Linf{u,v}{\Omega\etab}\dLp{4}{u,v}{\Omega \nablas(\Omega\etab)}\\+\Linf{u,v}{\Omega\chih}\Linf{u,v}{\Omega\etab}\dLp{4}{u,v}{\Omega\nablas(\Omega\etab)}
\end{multline}
and we continue with
\begin{equation*}
  \begin{split}
    \dLp{4}{u,v}{\Omega\nablas\bigl(\Omega^2\nablas\eta\bigr)}&\lesssim r^3\dLp{4}{u,v}{\nablas^2\eta}+r^3\Linf{u,v}{\ds\log\Omega}\dLp{4}{u,v}{\nablas\eta}\\
    &\lesssim r^3\dLp{4}{u,v}{\ds\tilde{\mu}}+1
\end{split}
\end{equation*}
and the last term is controlled by Lemma~\ref{lemma:Laplaces:log:Omega}.

Thus by Gronwall's inequality we obtain the desired bound $r^3\dLp{4}{u,v}{\ds\tilde{\mu}}\lesssim v_+-v$.

\end{proof}

\begin{remark}
  Essentially the same proof applies to show that
  \begin{equation}
   r^2 \Linf{u,v}{\tilde{\mu}} \lesssim  r^2 \dLp{4}{u,v}{\breve{\mu}}+r^3 \dLp{4}{u,v}{\ds \breve{\mu}} \lesssim 1
  \end{equation}
Indeed, both proofs rely on the propagation equations for $\Omega^2\tilde{\mu}$, and $\Omega^2\breve{\mu}$, respectively, at the same level of scaling in $\Omega$, the difference being merely the data on $\Cb_+$, for which we have $\dLp{4}{u,v_+}{r^3\tilde{\mu}}\lesssim 1$, but only $\dLp{4}{u,v_+}{r^2\breve{\mu}}\lesssim 1$, and similar for the first order.
\end{remark}

\begin{remark}
  For simplicity we have stated the assumption on $\tr\chi\tr\chib$ in \eqref{eq:coupled:eta:trchitrchib} separately.
  In view of \eqref{eq:Db:trchitrchib:general} --- c.f.~\eqref{eq:Db:trchitrchib} in Section~\ref{sec:L:infty:omega} --- the stated bound is consistent with the bound for the mass aspect function proven in Prop.~\ref{prop:coupled:eta} below.
\end{remark}

\subsection{Null lapse}

\label{sec:coupled:null:lapse}

Let us consider the propagation equation for $\Laplaces\log\Omega=\divs\ds\log\Omega$, in the shear-free case: If $\chibh=0$ then with $\theta=\ds\log\Omega$,
\begin{equation}\label{eq:Db:Laplaces:log:Omega}
  \begin{split}
    \Db \Laplaces\log\Omega =& \Db\divs\theta 
    =\tr\bigl(\Db\nablas\theta)-2\Omega(\chib,\nablas\theta)\\
    =&\divs\Db\theta-\Omega\tr\chib\divs\theta\\
    =&\Laplaces\omegab-\Omega\tr\chib\Laplaces\log\Omega
  \end{split}
\end{equation}

Moreover  it follows immediately from \Ceq{6.107} that
    \begin{equation}
      D\Laplaces\log\Omega+\Omega\tr\chi\Laplaces\log\Omega=\Laplaces\omega
      -2\divs\bigl(\Omega\chih^\sharp\cdot \ds\log\Omega\bigr)
  \end{equation}

  


    \begin{lemma} \label{lemma:Laplaces:log:Omega}
      Suppose \eqref{BA:Omega:dd:omegab} holds, then with spherically symmetric data for $\Omega$ on $C_+$,
      \begin{equation}
        r^4\dLp{4}{u,v}{\nablas^2\log\Omega}+r^3\dLp{4}{u,v}{\nablas\log\Omega}\lesssim 1
      \end{equation}

    \end{lemma}
    \begin{proof}
      From Cor.~\ref{cor:calderon:Laplaces} we have
      \begin{equation*}
        \dLp{4}{u,v}{\nablas^2\log \Omega}+\frac{1}{r}\dLp{4}{u,v}{\nablas\log\Omega}\lesssim \dLp{4}{u,v}{\Laplaces\log\Omega}
      \end{equation*}
and immediately from \eqref{eq:Db:Laplaces:log:Omega} we have
  \begin{equation*}
    \dLp{4}{u,v}{\Omega^2\Laplaces\log\Omega}\lesssim \int_u^{u_+}\dLp{4}{u,v}{\Omega^2\Laplaces\omegab}\lesssim \frac{1}{r^2}\,;
  \end{equation*}
for the first inequality we have used the assumption on the data on $C_+$, and for the second the assumption on $\omegab$, and Lemma~\ref{lemma:comparison}.

    \end{proof}

Similarly we could have derived the same result using appropriate data on $\Cb_+$.
  
  \begin{lemma}\label{lemma:dd:log:Omega}

    Assume  \eqref{BA:Omega:dd:omega} and
    \begin{equation}
      D(u) := \dLp{4}{u,v_+}{\Omega^2(\tilde{\mu}+\tilde{\mub})}+\dLp{4}{u,v_+}{\Omega^2\bigl(\tr\chi\tr\chib-\overline{\tr\chi\tr\chib}\bigr)} \lesssim 1/r^2
    \end{equation}

    Then  on $\mathcal{D}_+$, 
  \begin{equation}
    \dLp{4}{u,v}{\Omega\ds\log\Omega}+\dLp{4}{u,v}{\Omega^2 \nablas^2 \log\Omega }\lesssim 1/r^2
  \end{equation}

\end{lemma}
\begin{proof}

  We have from Corollary~\ref{cor:calderon:Laplaces} that
  \begin{equation}
    \dLp{4}{u,v}{r\ds\log\Omega}+\dLp{4}{u,v}{r^2\nablas^2\log\Omega}\lesssim \dLp{4}{u,v}{r^2\Laplaces\log\Omega}
  \end{equation}
  and we have the propagation equation
  \begin{multline}
  D \bigl(\Omega^2\Laplaces\log\Omega\bigr)=\bigl(2\omega-\Omega\tr\chi)\Omega^2\Laplaces\log\Omega\\
  +\Omega^2\Laplaces\omega-2\Omega^2\divs\bigl(\Omega\chih^\sharp\bigr)\cdot \ds\log\Omega-2\Omega\chih^\sharp\cdot \Omega^2\nablas^2\log\Omega
\end{multline}

Hence in view of Lemma~\ref{lemma:norm:comparison},
\begin{multline}
  \dLp{4}{u,v}{\Omega^2\Laplaces\log\Omega}\lesssim   \dLp{4}{u,v_+}{\Omega^2\Laplaces\log\Omega}+\int_{v}^{v_+}\dLp{4}{u,v'}{\Omega^2\nablas^2\omega}\ud v'\\
  +\int_{v}^{v_+}\Bigl(\Linf{u,v'}{2\omega-\Omega\tr\chi}+\Linf{u,v'}{    \Omega\chih } +\Linf{u,v'}{(\Omega\nablas)(\Omega\chih)}\Bigr) \dLp{4}{u,v'}{r^2\Laplaces\log\Omega}\ud v'
\end{multline}
and in particular by Propositon~\ref{prop:coupled:chi},  $ \Linf{u,v}{\Omega \nablas(\Omega\chih)}\lesssim 1 $.
For the first line on the r.h.s. recall \eqref{BA:Omega:dd:omega}, and the assumption.
The estimate then follows from Gronwall's inequality.

To express the boundary term on $\Cb_+$ note that by \eqref{eq:breve:mu:mub}
\begin{equation}
  \tilde{\mu}+\tilde{\mub}=-\divs\eta-\divs\etab+\frac{1}{2}\bigl(\tr\chi\tr\chib-\overline{\tr\chi\tr\chib}\bigr)
\end{equation}
and hence
\begin{equation}\label{eq:Laplace:log:Omega:eta}
  \Laplaces\log\Omega=\frac{1}{2}\bigl(\divs\eta+\divs\etab\bigr)=-\frac{1}{2}\bigl(\tilde{\mu}+\tilde{\mub}\bigr)+\frac{1}{4}\bigl(\tr\chi\tr\chib-\overline{\tr\chi\tr\chib}\bigr)
\end{equation}
and the statement of the Lemma follows.
\end{proof}

\subsection{Null expansions}

\label{sec:coupled:null:expansion}

Consider the elliptic equation
\begin{equation}\label{eq:Laplaces:omega}
  \Laplaces\omega=\omegas
\end{equation}

Here $\omegas$ simply denotes
\begin{equation}
  \omegas=\divs\ds\omega=\Laplaces\omega
\end{equation}
which satisfies the propagation equation
\begin{equation}\label{eq:Db:omegas}
  \begin{split}
    \Db\omegas=&\Db\divs\ds\omega=\divs\Db\ds\omega-2\divs\bigl(\Omega\chibh^\sharp\cdot\ds\omega\bigr)-\Omega\tr\chib\divs\ds\omega\\
    =&-2\divs\bigl(\Omega\chibh^\sharp\cdot\ds\omega\bigr)-\Omega\tr\chib \omegas\\
    &+2\Omega^2\bigl(2|\ds\log\Omega|^2+\Laplaces\log\Omega\bigr)\Bigl(2(\eta,\etab)-\lvert\etab\rvert^2+1\Bigr)\\
    &+2\Omega^2\Bigl(\ds\log\Omega,2(\nablas\eta,\etab)+2(\eta,\nablas\etab)-2(\etab,\nablas\etab)\Bigr)\\
        &+\Omega^2\Bigl(2(\Laplaces\eta,\etab)+2(\eta-\etab,\Laplaces\etab)\Bigr)+2\Omega^2\Bigl(2(\nablas\eta,\nablas\etab)-|\nablas\etab|^2\Bigr)
  \end{split}
\end{equation}

\begin{proposition}
\begin{equation}
  r^2\dLp{4}{u,v}{\nablas^2\omega}+r\dLp{4}{u,v}{\nablas\omega}\lesssim 1/r
\end{equation}
\end{proposition}

\begin{proof}

    From Corollary~\ref{cor:calderon:Laplaces} applied to \eqref{eq:Laplaces:omega} we have
  \begin{equation}
    r^2\dLp{4}{u,v}{\nablas^2\omega}+r\dLp{4}{u,v}{\nablas\omega}\lesssim \dLp{4}{u,v}{r^2\omegas}
  \end{equation}
  and from \eqref{eq:Db:omegas}
  \begin{multline}
    \dLp{4}{u,v}{\Omega^2\omegas}\lesssim \dLp{4}{u_+,v}{\Omega^2\omegas}\\
    +\int_u^{u_+} \Linf{u',v}{2(\eta,\etab)-|\etab|^2+1} \Bigl( \Linf{u',v}{  \Omega^2 \ds\log\Omega}^2 + \dLp{4}{u',v}{\Omega^4\Laplaces\log\Omega }\Bigr)\ud u'\\
    +\int_u^{u_+}\Linf{u',v}{\Omega^2(\nablas\eta,\etab)+\Omega^2(\eta,\nablas\etab)-\Omega^2(\etab,\nablas\etab)}\dLp{4}{u',v}{\Omega^2\ds\log\Omega}\ud u'\\
    +\int_u^{u_+}\dLp{4}{u',v}{\Bigl(\Omega^4(\Laplaces\eta,\etab)+\Omega^4(\eta-\etab,\Laplaces\etab)\Bigr)+\Omega^4\Bigl(2(\nablas\eta,\nablas\etab)-|\nablas\etab|^2\Bigr)}\ud u'
  \end{multline}

The terms involving $\nablas\log\Omega$ and $\nablas^2\log\Omega$ are bounded by Lemma~\ref{lemma:dd:log:Omega}.
And the terms involving $\nablas\eta$, and $\nablas^2\eta$  are bounded by Proposition~\ref{prop:coupled:eta:mu:breve}.

Recall that we have spherically symmetric data on $C_+$ so $\Laplaces\omega=0$ at $u=u_+$.
Thus the statement of the proposition follows.

\end{proof}

\begin{remark}
  This finally recovers the bootstrap assumption \eqref{BA:Omega:dd:omega}.
\end{remark}

\begin{proposition}
  Assume
  \begin{equation}
    \dLp{4}{u,v_+}{\Omega^2\Laplaces\omegab}\lesssim 1/r
  \end{equation}
Then we have
\begin{equation}
  r^3\dLp{4}{u,v}{ \nablas^2\omegab}+r^2\dLp{4}{u,v}{\nablas\omegab}\lesssim 1
\end{equation}

\end{proposition}

\begin{proof}

  First recall that the bootstrap assumption that \eqref{BA:Omega:dd:omegab} holds for some  $\Delta>0$.
  Then by Prop.~\ref{prop:coupled:eta:mu:breve}  and Lemma~\ref{lemma:Laplaces:log:Omega},
  \begin{gather*}
    \Linf{u,v}{\Omega^3\ds\log\Omega}\lesssim 1\\
    \Linf{u,v}{\Omega^2\eta}+\Linf{u,v}{\Omega^2\etab}\lesssim 1
  \end{gather*}

  Now from Cor.~\ref{cor:calderon:Laplaces} we have
  \begin{equation}\label{eq:omegab:elliptic}
    \dLp{4}{u,v}{\nablas^2\omegab}+\frac{1}{r}\dLp{4}{u,v}{\nablas\omegab}\lesssim \dLp{4}{u,v}{\omegabs}
  \end{equation}
and from \eqref{eq:D:Laplace:omegab},
  \begin{multline}
  \dLp{4}{u,v}{\Omega^2\omegabs} \lesssim \dLp{4}{u,v_+}{\Omega^2\omegabs} +\int_{v}^{v_+}\dLp{4}{u,v'}{\Omega^2\divs\bigl(\Omega\chih^\sharp\cdot \ds\omegab\bigr)}\ud v'\\
  +\int_v^{v_+}\Linf{u,v'}{2(\eta,\etab)-|\eta|^2+1}\dLp{4}{u,v'}{\Omega^4 \bigl(2|\nablas\log\Omega|^2+\Laplaces\log\Omega\bigr)}\ud v'\\
  +\int_v^{v_+}\Linf{u,v'}{\Omega\ds\log\Omega}\dLp{4}{u,v'}{\Omega^3(\nablas\eta,\etab-\eta)+\Omega^3(\eta,\nablas\etab)}\ud v'\\
  +\int_v^{v_+}\dLp{4}{u,v'}{\Omega^4\Bigl(\bigl(\Laplaces\eta,\etab-\eta\bigr)+\bigl(\nablas\eta,\nablas(\etab-\eta)\bigr)+\bigl(\nablas\eta,\nablas\etab\bigr)+\bigl(\eta,\Laplaces\etab\bigr)\Bigr)}\ud v'
\end{multline}
and we estimate
\begin{multline}
  \dLp{4}{u,v'}{\Omega^2\divs\bigl(\Omega\chih^\sharp\cdot \ds\omegab\bigr)}\lesssim \Linf{u,v'}{\Omega\ds\omegab}\dLp{4}{u,v'}{\Omega\nablas(\Omega\chih)}\\
  +\Linf{u,v'}{\Omega\chih}\dLp{4}{u,v'}{\Omega^2\nablas^2\omegab}
\end{multline}
and the use the elliptic estimate \eqref{eq:omegab:elliptic} again, before applying Gronwall's inequality.

Note that in the last line we apply Prop~\ref{prop:coupled:eta:mu:breve} to control $\Linf{u,v}{\Omega^2\nablas\etab}$.
\end{proof}

\begin{remark}
  This in particular recovers the bootstrap assumption \eqref{BA:Omega:dd:omegab}.
\end{remark}

\subsection{Torsion and mass aspect function}
\label{sec:mu}

Recall the definitions of the mass aspect functions \eqref{eq:mu}, \eqref{eq:mub}, which we view as elliptic systems for $\eta$, $\etab$, which read in the shear-free case:
\begin{subequations}\label{eq:eta:etab:elliptic:shear-free}
\begin{gather}
  \divs\eta=-\mu  \qquad \curls\eta=0\\
  \divs\etab=-\mub \qquad \curls\etab=0
\end{gather}
\end{subequations}

Moreover we have the propagation equations \eqref{eq:D:mu:all} for $\mu$,
    \begin{equation}
      \begin{split}
        D\mu=&-\frac{3}{2}\Omega\tr\chi \mu-\frac{1}{2}\Omega\tr\chi\bigl(2\Laplaces\log\Omega+|\etab|^2)-\frac{1}{4}\Omega\tr\chib|\chih|^2\\
        &+2\divs(\Omega\chih)\cdot (\eta-\etab)+2(\Omega\chih,\nablas\eta)+2\Omega\chih(\etab,\etab)
      \end{split}
    \end{equation}
and a conjugate equation for $\mub$,
\begin{equation}
  \Db\mub=-\Omega\tr\chib \mub+\divs \underline{j}-\frac{1}{2}\Omega\tr\chib(\mu-|\eta|^2)-\frac{1}{4}\Omega\tr\chi|\chibh|^2\,,\quad \underline{j}=2\Omega\chibh\cdot \etab-\Omega\tr\chib\eta
\end{equation}
which in the shear-free case reduces to
\begin{equation}\label{eq:Db:mub:shear-free}
        \Db\mub+\frac{3}{2}\Omega\tr\chib \mub=-\frac{1}{2}\Omega\tr\chib\bigl(2\Laplaces\log\Omega+|\eta|^2\bigr)
    \end{equation}


    \begin{proposition} \label{prop:coupled:eta}
    We have
    \begin{multline}
      \dLp{4}{u,v}{r^2\eta}+ \dLp{4}{u,v}{r^3 \nablas\eta}+\dLp{4}{u,v}{r^3 \mu}\\+\dLp{4}{u,v}{r^2\etab}+ \dLp{4}{u,v}{r^3 \nablas\etab}  \dLp{4}{u,v}{r^3 \mub} \lesssim \\
      \lesssim   \sup_u\dLp{4}{u,v_+}{\Omega^3\mu}+ \sup_v\dLp{4}{u_+,v}{\Omega^3\mub}
    \end{multline}

  \end{proposition}

\begin{proof}
  We can apply Lemma~\ref{lemma:calderon:div:curl} to \eqref{eq:eta:etab:elliptic:shear-free},
  \begin{align*}
     \dLp{4}{u,v}{r^2\eta}+ \dLp{4}{u,v}{r^3 \nablas\eta}&\lesssim   \dLp{4}{u,v}{r^3 \mu}\\
          \dLp{4}{u,v}{r^2  \etab}+ \dLp{4}{u,v}{r^3 \nablas\etab}&\lesssim   \dLp{4}{u,v}{r^3 \mub}
  \end{align*}
  and get from \eqref{eq:D:mu:all}, \eqref{eq:Db:mub:shear-free}, and Lemma~\ref{lemma:norm:comparison},  that
    \begin{multline*}
          \dLp{4}{u,v}{\Omega^3\mu}\lesssim \dLp{4}{u,v_+}{\Omega^3\mu} + \int_v^{v_+}\dLp{4}{u,v'}{\Omega^4\Laplaces\log\Omega}+\Linf{u,v'}{r^2\etab}\dLp{4}{u,v'}{r^2 \etab}\ud v'\\
      +\int_v^{v_+}\Linf{u,v'}{\Omega^2\chih}\Bigl\{\dLp{4}{u,v'}{\Omega^2\chih}+\dLp{4}{u,v'}{\Omega^2\nablas\eta}+\Linf{u,v'}{\Omega\etab}\dLp{4}{u,v'}{\Omega\etab}\Bigr\}\ud v'\\
      +\int_v^{v_+}\dLp{4}{u,v'}{\Omega\divs(\Omega\chih)}\Bigl\{\Linf{u,v'}{\Omega^2\eta}+\Linf{u,v'}{\Omega^2\etab}\Bigr\}\ud v'
    \end{multline*}
  \begin{equation*}
    \dLp{4}{u,v}{\Omega^3\mub}\lesssim \dLp{4}{u_+,v}{\Omega^3\mub} + \int_u^{u_+}\dLp{4}{u',v}{\Omega^4\Laplaces\log\Omega}+\Linf{u',v}{r^2\eta}\dLp{4}{u',v}{r^2 \eta}\ud u'
  \end{equation*}

We have $\dLp{4}{u,v}{\Omega^4\Laplaces\log\Omega}\lesssim 1$ by Lemma~\ref{lemma:Laplaces:log:Omega}.

  Moreover by Prop.~\ref{prop:coupled:chi} above, $\dLp{4}{u,v'}{\Omega\divs(\Omega\chih)}\lesssim 1$.

Let
\begin{equation*}
  m(u,v):=\dLp{4}{u,v}{\Omega^3\mu}\qquad   \underline{m}(u,v):=\dLp{4}{u,v}{\Omega^3\mub}
\end{equation*}
then we can infer from the above inequalities, in view of Cor.~\ref{cor:sobolev:d} ,
\begin{equation*}
  m(u,v) \lesssim m(u,v_+) + \int_v^{v_+} \Delta+D^2+\bigl(1+D\bigr)\bigl(m(u,v')+\underline{m}(u,v')\bigr)+(1+D) \underline{m}^2(u,v')\ud v'
\end{equation*}
\begin{equation*}
  \underline{m}(u,v) \lesssim \underline{m}(u_+,v)+\int_u^{u_+}\Delta+m^2(u',v)\ud u'
\end{equation*}
and we will make a continuity argument, under the bootstrap assumption:
\begin{equation*}
  m(u,v)+ \underline{m}(u,v) \leq C_M
\end{equation*}

Now, denoting by
\begin{equation*}
  M(u):=\sup_v m(u,v)
\end{equation*}
\begin{equation*}
  \int_v^{v_+}\underline{m}(u,v')\ud v' \lesssim \int_v^{v_+}\underline{m}(u_+,v')\ud v'+(v_+-v)\max\{\Delta,C_M\}\int_u^{u_+}1+M(u')\ud u'
\end{equation*}
we obtain after inserting above, and already applying Gronwall's inequality to the term linear in $m$,
\begin{equation*}
  \begin{split}
    M(u)\leq& \Bigl(m(u,v_+)+2\max\{\Delta,D^2\}(v_+-v)+4\max\{\Delta_\chi,D,C_M,C_MD\}\int_v^{v_+}\underline{m}(u,v')\ud v'\Bigr)\times\\&\times\Bigl(1+2\max\{\Delta_\chi,D\bigr\}(v_+-v)e^{2\max\{\Delta_\chi,D\}(v_+-v)}\Bigr)\\
    \leq& \Bigl(m(u,v_+)+C \triangle v+C_MC\sup_v\underline{m}(u_+,v)\triangle v
    +C_M^2 C^2 F(u) \triangle v\Bigr)\Bigl(1+C\triangle v e^{C\triangle v}\Bigr)
  \end{split}
\end{equation*}
where $C=\max\{1,\Delta,D,D^2,\Delta_\chi\}$, $\triangle v =v_+-v$, 
\begin{equation*}
  F(u)=\int_u^{u_+}1+M(u')\ud u'
\end{equation*}

This yields, with $\triangle u=u_+-u$,
\begin{gather*}
  \frac{\ud}{\ud u}\bigl(-F(u)e^{-\int_u^{u_+}B\ud u'}\bigr)\leq \bigl(1+M(u)-B F(u)\bigr)e^{-B\triangle u}\leq E e^{-B\triangle u}\\
  B=C_M^2C^2\triangle v \Bigl(1+C\triangle v e^{C\triangle v}\Bigr)\\
  E(u)= \Bigl(1+m(u,v_+)+C \triangle v+C_MC\sup_v\underline{m}(u_+,v)\triangle v
    \Bigr)\Bigl(1+C\triangle v e^{C\triangle v}\Bigr)
\end{gather*}
and hence
\begin{equation*}
  \begin{split}
    F(u)\leq& e^{B\triangle u}\int_u^{u_+} E(u') e^{-B\triangle u'}\ud u'\\
    \leq& e^{B\triangle u} \Bigl(1+\sup_u m(u,v_+)+C \triangle v+C_MC\sup_v\underline{m}(u_+,v)\triangle v \Bigr)\Bigl(1+C\triangle v e^{C\triangle v}\Bigr)\triangle u
  \end{split}
\end{equation*}
which in turn implies a bound on $M(u)$, hence $m(u,v)$, and thus also $\underline{m}(u,v)$.
In particular, 
by choosing $C_M\gtrsim \sup_u m(u,v_+)+\sup_{v}\underline{m}(u_+,v)$,
we can improve the bootstrap assumption by choosing  $\triangle u$, and $\triangle v$ sufficiently small compared to $C_M$, and $C$, which ensures that $F(u)\lesssim 1$, hence $m(u,v)\leq M(u)\leq 2\bigl(m(u,v_+)+1\bigr)$, and $\underline{m}(u,v)\leq \underline{m}(u_+,v)+1$.
\end{proof}


\section{Further considerations}

In Section~\ref{sec:shear-free:graphs} we include a discussion of the surfaces $S_{u,\vs}$ as graphs over the round sphere $S_{\us,\vs}$ in the $\Cb_{\vs}$.\footnote{See also \cite{Le:18} for a similar discussion of sections of null hypersurfaces in Minkowski space.}
We show in particular that the prescription of the mass aspect function on the last slice can be viewed as an equation of motion, which \emph{asymptotically} identifies correctly the round spheres.

\subsection{Shear free incoming null hypersurfaces as graphs}
\label{sec:shear-free:graphs}

  Recall the metric in the spherically symmetric form
  \begin{equation}
    h=-4\Omega_\ast^2\ud \us\ud\vs+r^2\gammac_{AB}\ud\vartheta^A\ud\vartheta^B
  \end{equation}

  Let us assume that the sections $S_{u,\vs}$ can be written as a graph over $S_{\us,\vs}$:
   \begin{equation}
    \us=u+f(\vartheta)
  \end{equation}
  In fact, this graph may depend on $u$, as well as on $\vs$, when we consider the intersections with nearby null hypersurfaces $\Cb_{\vs}$ as well:
  \begin{equation}
    \us=u+f(u,\vartheta;\vs)
  \end{equation}
  We will always assume that $\partial_u f>-1$.

  \begin{lemma}\label{lemma:h:f}
    Let $f=f(u,\vartheta;\vs)$ be a smooth function such that $\partial_u f>-1$, and consider the graphs \[S_{u,\vs}=\Bigl\{(\us,\vs;\vartheta):\us=u+f(u,\vartheta;\vs)\qquad :\vartheta\in\mathbb{S}^2\Bigr\}\] In $(u,\vs;\vartheta)$ coordinates, the metric takes the form
\begin{equation}
  h=-4\Omega^2\ud u\ud \vs+\hs_{AB}\bigl(\ud \vartheta^A-b^A\ud\vs\bigr)\bigl(\ud \vartheta^B-b^B\ud \vs\bigr)\\
\end{equation}
where 
\begin{align}
  \Omega^2&=\Omega_\ast^2\Bigl(1+\frac{\partial f}{\partial u}\Bigr)\\
  b&=2\frac{\Omega_\ast^2}{r^2}\nablac f\\
  \hs&=r^2\gammac
\end{align}
provided:
\begin{equation}\label{eq:df:condition}
  \frac{\partial f}{\partial \vs}=\frac{\Omega_\ast^2}{r^2}|\nablac f|^2\,.
\end{equation}

\end{lemma}

\begin{remark}
   
Here
\begin{equation}
  r=r(u,\vartheta;\vs)=r(\us=u+f(u,\vartheta;\vs),\vs)
\end{equation}
and is \emph{not} to be confused with the area radius of the sphere $S_{u,\vs}$, which we will here denote by  $r_{u,\vs}$ and plays an important role below ; it  is defined by
\begin{equation}\label{eq:area:f}
  4\pi r^2_{u,\vs}=\int_{S_{u,\vs}}\dm{\hs}=\int_{\mathbb{S}^2}r^2(u,\vartheta;\vs)\dm{\gammac}(\vartheta)
\end{equation}
\end{remark}

\begin{proof}  
  Inserting
  \begin{equation}
      \ud \us=\bigl(1+\frac{\partial f}{\partial u}\bigr)\ud u+\frac{\partial f}{\partial \vartheta^A}\ud \vartheta^A+\frac{\partial f}{\partial \vs}\ud \vs
  \end{equation}
  the metric in coordinates $(u,\vs;\vartheta^A)$ reads:
  \begin{multline}
   h=-4\Omega_\ast^2\Bigl(1+\frac{\partial f}{\partial u}\Bigr) \ud u\ud\vs\\+r^2\gammac_{AB}\Bigl(\ud \vartheta^A-2\frac{\Omega_\ast^2}{r^2}\gammac^{AC}\frac{\partial f}{\partial \vartheta^C}\ud\vs\Bigr)\Bigl(\ud\vartheta^B-2\frac{\Omega_\ast^2}{r^2}\gammac^{BD}\frac{\partial f}{\partial \vartheta^D}\ud \vs\Bigr)
 \end{multline}
provided \eqref{eq:df:condition} holds;
this condition is necessary for the $(\ud\vs)^2$ component to cancel, and thus for $(u,\vs)$ to define a double null coordinate system.

\end{proof}

The null normals to $S_{u,\vs}$ are given by
\begin{equation}
  \Lb=\frac{\partial }{\partial u}\qquad L=\frac{\partial}{\partial \vs}+b^A\frac{\partial }{\partial \vartheta^A}
\end{equation}

\begin{remark}
  This can also be verified as follows:
 In new coordinates $(u,\vs;\vartheta)$, the lines of constant $\vs$ and $\vartheta$ are still null geodesics, so
 \begin{equation}
   \Lb=\frac{\partial}{\partial u}=\Bigl(1+\frac{\partial f}{\partial u}\Bigr)\frac{\partial}{\partial \us}
 \end{equation}
 but the angular vectorfields change,
 \begin{equation}
   \frac{\partial}{\partial \vartheta^A}=   \frac{\partial}{\partial \vartheta_\ast^A}+\frac{\partial f}{\partial \vartheta^A}\frac{\partial }{\partial \us}
 \end{equation}
 We do not need to solve for $u$ to find $L$, but it can be determined from the conditions
 \begin{equation}
   h(L,\Lb)=-2\Omega^2\qquad h(L,\partial_A)=0
 \end{equation}
which yields
\begin{equation}
  L=\frac{\partial}{\partial \vs}+2\frac{\Omega_\ast^2}{r^2}\gammac^{AB}\frac{\partial f}{\partial \vartheta^B}\frac{\partial}{\partial \vartheta^A_\ast}=\frac{\partial}{\partial \vs}+2\frac{\Omega_\ast^2}{r^2}\nablac f
\end{equation}
\end{remark}

\begin{lemma}
  On $S_{u,\vs}$ we have
\begin{gather}
  \chibh=0\\
  \Omega\tr\chib
  =\frac{2r(u,\vs;\vartheta)}{\cosh^2(u+f(u,\vartheta;\vs)+\vs)}\Bigl(1+\frac{\partial f}{\partial u}\Bigr)>0
\end{gather}

\end{lemma}

\begin{proof}
We have seen the metric on $S_{u,\vs}$ is
 \begin{equation}
   \hs_{AB}=h(\partial_A,\partial_B)=\Bigl(\frac{\cosh(u+f(u,\vartheta;\vs)+\vs)}{\sinh(u+f(u,\vartheta;\vs)+\vs)}\Bigr)^2\gammac_{AB}
 \end{equation}

 Then we compute
\begin{equation}
  \Db\hs=-2\frac{\cosh(u+f+\vs)}{\sinh^3(u+f+\vs)}\Bigl(1+\frac{\partial f}{\partial u}\Bigr)\gammac=-\frac{2}{\cosh(u+f+\vs)\sinh(u+f+\vs)}\Bigl(1+\frac{\partial f}{\partial u}\Bigr)\hs
\end{equation}
and the formulas follow.
\end{proof}

Consider the area radius \eqref{eq:area:f}. Since for ``small $f$''
\begin{equation}\label{eq:r:approx:f}
  \begin{split}
    r(u,\vartheta;\vs)&=r(\us=u+f,\vs)=r(\us=u,\vs)+\frac{\partial r}{\partial \us}(u,\vs) f(u,\vartheta)+\frac{\partial^2 r}{\partial {\us}^2}(u+t,\vs)f^2\\
    &=r(u,\vs)+\bigl(r^2(u,\vs)-1\bigr)f(u,\vartheta;\vs)+r\bigl(r^2-1\bigr)f^2
  \end{split}
\end{equation}
we evidently need to assume
\begin{equation}\label{eq:rf}
  r(u,\vs)|f(u,\vartheta;\vs)|\leq C
\end{equation}
to ensure that $r(u,\vartheta;\vs)$ is \emph{of the order of} $r(u,\vs)$.
In fact, if we assume that $r|f|\ll 1$, then
\begin{equation}
  4\pi r^2_{u,\vs}\simeq r^2(u,\vs)\int_{\mathbb{S}^2}1+2 r(u,\vs)f(u,\vartheta;\vs)\dm{\gammac}(\vartheta)
\end{equation}
In particular, under such an assumption we will have
\begin{equation}
  \frac{\Omega^2}{r^2_{u,\vs}}\simeq 1
\end{equation}



Furthermore, since $(S_{u,\vs},\hs)$ is conformal to $(\mathbb{S}^2,\gammac)$,
\begin{equation}
  \hs=e^{2w}\gammac\,,\qquad w=\log r(u,\vartheta;\vs)
\end{equation}
we have by the formula for the transformation of the Gauss curvature:
\begin{equation}
  K=e^{-2w}\bigl(1-\Laplacec w\bigr)=\frac{1}{r^2}\bigl(1-\Laplacec\log r\bigr)
\end{equation}
In particular,
\begin{equation}
  K \simeq \frac{1}{r^2}\Bigl(1- \Laplacec (rf)\Bigr)
\end{equation}
which says that for this sphere to be a perturbation of a round sphere we need
\begin{equation}
  |\nablac^2 (rf) |\ll 1\,.
\end{equation}

The formula for the Gauss curvature also lends itself to a direct computation of $\tr\chi$ via the Gauss equations, which here reduces to
\begin{equation}
    \frac{1}{4}\tr\chi\tr\chib=\frac{\Lambda}{3}-K
\end{equation}
and thus yields with the above:
\begin{equation}
  \tr\chi=\frac{2\Omega_\ast}{r}\cosh(f-\vs)\frac{\frac{\Lambda}{3}r^2-1+\Laplacec \log r}{r^2}\sqrt{\frac{\partial f}{\partial u}}
\end{equation}

\begin{lemma}
  \begin{align}
    \eta_A&=\frac{1}{2}\frac{\Omega_\ast^2+r^2}{\Omega_\ast^2}\frac{1}{1+\partial_u f}\frac{\partial^2 f}{\partial \vartheta^A\partial u}+\frac{1}{r}\frac{1}{1+\partial_u f}\frac{\partial f}{\partial \vartheta^A}    \\
    \etab_A&=\frac{1}{2}\frac{\Omega_\ast^2-r^2}{\Omega_\ast^2}\frac{1}{1+\partial_u f}\frac{\partial^2 f}{\partial \vartheta^A\partial u}-\frac{1}{r}\frac{1}{1+\partial_u f}\frac{\partial f}{\partial \vartheta^A}    
  \end{align}
  In particular, 
  \begin{align}
    \lim_{u,r\to\infty}\Omega|\eta|&=|\nablac\partial_u f|\\
    \lim_{u,r\to\infty}\Omega^2|\etab|&=|\nablac  f|
  \end{align}
  provided the right hand sides are finite.
\end{lemma}

\begin{proof}
  Recall the formula \Ceq{1.91}:
\begin{equation}
  [\Lb,L]=4\Omega^2\zeta^\sharp
\end{equation}
which with above expression for the null normal s simply gives
\begin{equation}
  4\Omega^2{\zeta^\sharp}^A=\frac{\partial b^A}{\partial u}
\end{equation}
With the formula for $b$ found in Lemma~\ref{lemma:h:f}
\begin{equation}
  \frac{\partial b^A}{\partial u}= 2\gammac^{AB} \frac{\partial^2 f}{\partial u\partial \vartheta^B} +\frac{4}{r}\frac{\Omega_\ast^2}{r^2}\gammac^{AB}\frac{\partial f}{\partial\vartheta^B }
\end{equation}
we thus have
\begin{equation}
    \zeta_A=\hs_{AB}\zeta^{\sharp B}=\frac{r^2}{4\Omega^2}\gammac_{AB} 4\Omega^2 \zeta^{\sharp B}
    =\frac{1}{2}\frac{r^2}{\Omega_\ast^2}\frac{1}{1+\partial_u f}\frac{\partial^2 f}{\partial \vartheta^A\partial u}+\frac{1}{r}\frac{1}{1+\partial_u f}\frac{\partial f}{\partial \vartheta^A}    
\end{equation}
Since also
\begin{equation}
  \ds_A \log \Omega=\frac{1}{2}\frac{\partial}{\partial \vartheta^A}\log \bigl(1+\partial_u f\bigr)=\frac{1}{2}\frac{1}{1+\partial_u f}\frac{\partial^2 f}{\partial u \partial \vartheta^A}
\end{equation}
the stated formulas follow from $\eta=\zeta+\ds\log\Omega$, and $\etab=-\zeta+\ds\log\Omega$.
\end{proof}


  We can use the above formula for the torsion, and the conformal covariance of the divergence to compute
  \begin{equation}
    \divs\eta=\frac{1}{r^2}\divc\eta
  \end{equation}
  We find
  \begin{equation}
    \begin{split}
    \divc\eta=&\frac{1}{2}\frac{\Omega_\ast^2+r^2}{\Omega_\ast^2}\frac{1}{1+\partial_u f}\Laplacec\frac{\partial f}{\partial u}+\frac{1}{r}\frac{1}{1+\partial_u f}\Laplacec f    \\
    &-\frac{1}{2}\frac{\Omega_\ast^2+r^2}{\Omega_\ast^2}\frac{1}{(1+\partial_u f)^2}|\nablac\partial_u f|^2-\frac{1}{r}\frac{1}{(1+\partial_u f)^2}(\nablac \partial_u f,\nablac f)    \\
    &+\frac{1}{r^3}\frac{1}{1+\partial_u f}\gammac^{AB}\partial_B r\frac{\partial^2 f}{\partial \vartheta^A\partial u}-\frac{1}{r^2}\frac{1}{1+\partial_u f}\gammac^{AB}\partial_B r\frac{\partial f}{\partial \vartheta^A}    
  \end{split}
\end{equation}
and using \eqref{eq:r:approx:f} we may approximate
\begin{equation}
  \partial_A r \simeq r^2 \partial_A f
\end{equation}
and find the following limit
\begin{equation}
  \lim_{u;r\to\infty}\Omega^2\divs \eta=\Laplacec \partial_u f-\frac{1}{1+\partial_u f}|\nablac\partial_uf|^2-|\nablac f|^2
\end{equation}

In particular, we can conclude the following if $\divs\eta=\mathcal{O}(r^{-3})$, as is the case for the foliations in this paper, because $\divs\eta=-\mu$, and $\mu=\mathcal{O}(r^{-3})$; see Section~\ref{sec:prelim:torsion} and \eqref{eq:div:curl:eta:intro}. Then the limiting equation evidently reads:
\begin{equation}
  (1+\partial_u f)\Laplacec \partial_u f=|\nablac\partial_uf|^2+(1+\partial_u f)|\nablac f|^2
\end{equation}
This equation has no non-trivial solutions because  by integration on the sphere it follows that
\begin{equation}
  \int_{\mathbb{S}^2}2|\nablac \partial_u f|^2+(1+\partial_u f)|\nablac f|^2 \dm{\gammac}=0
\end{equation}
so it follows that
\begin{equation}
  f\equiv \text{constant}\qquad \partial_uf\equiv \text{constant}\,.
\end{equation}
In other words, the only solution to the \emph{limiting} equation of motion of surfaces on shear-free null hypersurfaces in de Sitter are precisely the round spheres.









\providecommand{\bysame}{\leavevmode\hbox to3em{\hrulefill}\thinspace}
\providecommand{\MR}{\relax\ifhmode\unskip\space\fi MR }
\providecommand{\MRhref}[2]{%
  \href{http://www.ams.org/mathscinet-getitem?mr=#1}{#2}
}
\providecommand{\href}[2]{#2}

\end{document}